\documentclass{article}
\usepackage{dcolumn}% Align table columns on decimal point
\usepackage{bm}% bold math
\usepackage{verbatim}       % Defines \begin{comment}, \end{comment}

\usepackage[pdftex]{graphicx}
\usepackage{amsthm}
\usepackage{amssymb}
\usepackage{bm}
\usepackage{amstext}
\usepackage{psfrag}
\usepackage{amsmath}
\usepackage{amsfonts}
\usepackage{mathrsfs}
\usepackage{cite}
\usepackage{color}
\usepackage{hyperref}
\usepackage[toc]{appendix}

\newfont{\bb}{msbm10 at 12pt}

\def\R{\hbox{\bb R}}

\newcommand{\p}{\partial}

\newcommand{\dd}{{\rm d}}
\newcommand{\bd}{\begin{definition}}                %inizia definizione
\newcommand{\ed}{\end{definition}}                  %fine definizione
\newcommand{\bc}{\begin{corollary}}                 %inizia corollario
\newcommand{\ec}{\end{corollary}}                   %fine corollario
\newcommand{\bl}{\begin{lemma}}                     %inizia lemma
\newcommand{\el}{\end{lemma}}                       %fine lemma
\newcommand{\bp}{\begin{proposition}}            %inizia proposizione
\newcommand{\ep}{\end{proposition}}                %fine proposizione
\newcommand{\bere}{\begin{remark}}                  %inizia osservazione
\newcommand{\ere}{\end{remark}}                     %fine oservazione

\newcommand{\bt}{\begin{theorem}}
\newcommand{\et}{\end{theorem}}

\newcommand{\be}{\begin{equation}}
\newcommand{\ee}{\end{equation}}

\newcommand{\bit}{\begin{itemize}}
\newcommand{\eit}{\end{itemize}}
\newtheorem{theorem}{Theorem}[section]
\newtheorem{corollary}[theorem]{Corollary}

\newtheorem{lemma}[theorem]{Lemma}
\newtheorem{proposition}[theorem]{Proposition}
\theoremstyle{definition}
\newtheorem{definition}[theorem]{Definition}
\theoremstyle{remark}
\newtheorem{remark}[theorem]{Remark}
\newtheorem{example}[theorem]{Example}

\DeclareMathOperator*{\supp}{supp}

\DeclareMathOperator*{\diam}{diam}

\DeclareMathOperator*{\dist}{dist}

\begin{document}
%
%\DeclareGraphicsExtensions{.pdf}

\title{Lorentzian metric spaces and their Gromov-Hausdorff convergence}

\author{E. Minguzzi\footnote{Dipartimento di Matematica e Informatica ``U. Dini'', Universit\`a degli Studi di Firenze,  Via
S. Marta 3,  I-50139 Firenze, Italy. E-mail:
ettore.minguzzi@unifi.it,
\newline ORCID:0000-0002-8293-3802}
\ \ and \  \ S. Suhr
\footnote{Fakult\"at f\"ur Mathematik, Ruhr-Universit\"at Bochum, Universit\"atsstr. 150, 44780 Bochum, Germany. E-mail: Stefan.Suhr@ruhr-uni-bochum.de,
\newline ORCID:0000-0001-6787-9396
\newline Stefan Suhr is supported by the SFB/TRR 191 ``Symplectic Structures in Geometry, Algebra and Dynamics'', funded by the Deutsche Forschungsgemeinschaft. }
}

\date{}

\maketitle

\begin{abstract}
\noindent We present an abstract approach to Lorentzian Gromov-Hausdorff distance and convergence, and an alternative approach to Lorentzian length spaces that does not use auxiliary ``positive signature'' metrics or other unobserved fields. We begin  by defining a notion of (abstract) bounded Lorentzian-metric space which is sufficiently general to comprise compact causally convex subsets of globally hyperbolic spacetimes and causets. We define the Gromov-Hausdorff distance and show that two bounded Lorentzian-metric spaces at zero GH distance are indeed both isometric and homeomorphic. Then we show how to define from the Lorentzian distance, beside topology,  the causal relation and the causal curves for these spaces, obtaining useful limit curve theorems. Next, we  define Lorentzian (length) prelength spaces  via suitable (maximal) chronal connectedness properties. These definitions are proved to be stable under GH limits. Furthermore, we define bounds on sectional curvature for our Lorentzian length spaces and prove that they are also stable  under GH limits. We conclude with  a (pre)compactness theorem.
\end{abstract}

\noindent Key Words: {Lorentzian metric geometry, metric convergence, causality, isometric mappings}
%\pacs{}

%\noindent

\tableofcontents

\section{Introduction}
\label{bpo}

This paper is devoted to the development of a Lorentzian-signature version of the by now classical theory of metric geometry, as developed, for instance, in the  book by Burago et al.\ \cite{burago01}. This theory comprises results on sectional curvature bounds via comparison geometry, as obtained by the A.D. Alexandrov school, and results on Ricci bounds from below as obtained by Gromov.

Our work fits well in the current trend in low regularity spacetime (i.e.\ Lorentzian) geometry. Although  there were contributions in non-regular Lorentzian geometry in the last two decades, it is only in the last few years that mathematicians have started getting some systematic results.

It must be said that the investigation of Lorentzian analogs to the Gromov-Hausdorff convergence has likely been one of the least explored. Among the most relevant papers we find  those by
Noldus who introduced and studied a Gromov-Hausdorff distance for compact globally hyperbolic {\em manifolds} with spacelike boundary \cite{noldus04,noldus04b} and later went on, in a joint work with Bombelli \cite{bombelli04}, to explore the non-manifold case.

Recently, M\"uller  reconsidered some constructions by Noldus, including his strong metric, in a broad categorical analysis of the Gromov-Hausdorff convergence problem  \cite{muller19,muller22}.
In \cite{muller22b} he studied the Gromov-Hausdorff metric for compact Lorentzian pre-length spaces and for partially ordered sets, still within a categorical framework. Some of the results therein contained are closely related to parts of the theory
presented here, but were obtained independently.

In the context of Lorentzian length spaces \`a la Kunzinger and S\"amann \cite{kunzinger18}, Cavaletti and Mondino \cite{cavalletti20,cavalletti22} defined a notion of {\em measured Gromov-Hausdorff convergence} of measured
Lorentzian geodesic spaces and proved a stability result for the timelike curvature condition. Similar results were recently obtained in \cite{braun22}.

Sormani and Wenger \cite{sormani11}  followed a rather different route introducing the notion of {\em intrinsic flat distance} which is a kind of analog to the Gromov-Hausdorff distance when convergence of sets in the Hausdorff-distance sense is replaced by {\em Federer-Fleming’s flat convergence}. This type of convergence is used to study sequences of Lorentzian manifolds foliated by spacelike manifolds with non-negative scalar curvature.

Following a similar `positive signature' approach Sormani and Vega \cite{sormani16} show that in a cosmological spacetime it is possible to introduce a (so called {\em null}) distance by
taking advantage of Andersson et al.\  {\em cosmological time function} \cite{andersson98}.
Their distance is well adapted to express convergence of cosmological spacetimes, a point of view taken up by Allen and Burtscher   \cite{allen19} where they show that in the class of warped product (smooth) spacetimes, uniform convergence of warping functions can be related to  Gromov-Hausdorff and Sormani-Wenger intrinsic flat convergence  (see also \cite{sakovich22,burtscher22} for related causality results). Kunzinger and Steinbauer in \cite{kunzinger22} take up again the idea of a null distance, but in the broader framework of Lorentzian-length spaces and show that in certain warped product Lorentzian length spaces the Gromov-Hausdorff convergence interacts well with synthetic curvature bounds.  It is understood that in these null distance approaches, the Gromov-Hausdorff convergence used is the traditional positive-signature one.

Our approach is closer in spirit to Noldus' work, but starts with the identification of a convenient notion of (bounded) Lorentzian-metric space. This notion is more general than the subsequently introduced notion of (pre)length space.
In this sense our work sets the stage for a complete Lorentzian analog of metric geometry. Possibly, it is  this general framework that sets apart our work with respect to previous approaches. Also we define a notion of Gromov-Hausdorff convergence and show that all our relevant definitions, from the mentioned (pre)length spaces, to the sectional curvature bounds are stable under GH limits. In fact, it is this stability criteria that helped us in selecting what we believe to be a most convenient notion of (pre)length space.

Although this work presents results for just the timelike-diameter bounded spaces, we also explored the non-bounded case, obtaining some  results that we leave for  a next  work.

Since we shall be introducing new definitions for Lorentzian (pre)length spaces, we wish to comment on  the differences with the approach by Kunzinger and S\"amann \cite{kunzinger18}. There a prelength space is a 5-tuple $(X,d, \ll, \le, \rho)$, where $X$ is the set, $d$ the Lorentzian distance, $\ll$ the chronological relation, $\le$ the causal relation, and $\rho$ a metric, all these elements satisfying  some  compatibility conditions. Causal curves are defined by imposing the locally Lipschitz property with respect to $\rho$, see\cite[Def.\ 2.18]{kunzinger18}.  As a consequence, some causality properties such as non-total imprisonment depend on $\rho$.

In our work we derive the topology, the chronological relation and the causal relation from just $(X,d)$, which is what makes the GH stability property possible. Moreover, our causal curves are not defined imposing the locally Lipschitz property with respect to some metric. Actually, there is a metric that can be defined from $d$, this is the distinction or Noldus metric, but in simple smooth examples this choice shows to be untenable as the usual causal curves are non-rectifiable with respect to it \cite[Thm.\ 6]{noldus04b} \cite[Thm.\ 3]{muller19}.

In fact, in our work we do not really feel the need to restrict the class of causal curves with some kind of rectifiability property with respect to some metric. We are able to obtain standard limit curve theorems by parametrizing causal curves (or better isocausal curves) via {\em time functions constructed from the distance function}, as this type of function has within it information that behaves well under GH limits.

%In conclusion,

Our definition of Lorentzian length space is different from that by Kunzinger and S\"amann and more work will be required to establish if the two notions really coincide in some cases. We believe that many successes of the KS-framework could  be replicated in our approach (particularly if we assume the existence of convex neighborhoods, see Def.\ \ref{vmqp}) but our setting has some relevant mathematical and physical advantages. On the mathematical side we have a theory which is better behaved under Gromov-Hausdorff convergence, and is more closely related to the `smooth limit' of Lorentzian geometry, not needing the auxiliary metric $\rho$. On the related physical side, the independence from $\rho$ is very desirable as, while the Lorentzian distance in the smooth limit is indeed represented by a physical field - the Lorentzian metric $g$ -, the ingredient $\rho$ has no smooth counterpart, as we do not find in the fundamental physical fields of the Standard Model or of General Relativity a Riemannian metric $h$.

Finally, we mention that our work contains a (pre)compactness result for certain families of bounded Lorentzian metric spaces with uniform bounds on the timelike diameter and on the number of certain $\epsilon$-nets. Unfortunately, we could not establish a clear connection with Ricci bounds introduced via comparison geometry. It seems that the notion of $\epsilon$-net is too much adapted to the notion of distinction-metric ball, which is, unfortunately, a non-local concept. More work and new ideas will be needed to produce a precompactness theorem involving bounds on the Ricci curvature.

\subsection{Basic definitions}

The whole work will be devoted to the study of the following object:
\begin{definition}\label{D1}
A {\it bounded Lorentzian-metric space} $(X,d)$, is a set $X$ endowed with a
map $d: X\times X \to [0,\infty)$ such that
\begin{itemize}
\item[(i)] For every $x,y,z\in X$  with  $d(x,y)>0$, $d(y,z)>0$ we have
    \[
    d(x,z)\ge d(x,y) +d(y,z).
    \]
\item[(ii)] There is a topology ${T}$ on $X$ for which  $d$ is continuous in the product topology ${T}\times {T}$ and for every $\epsilon>0$ the sets
\[
\{(p,q): d(p,q)\ge \epsilon\}
\]
are compact with respect to the product topology ${T}\times {T}$.

\item[(iii)] $d$ distinguishes points, i.e.\ for every pair $x,y$, $x\ne y$ there is $z$ such that $d(x,z)\ne d(y,z)$ or $d(z,x) \ne d(z,y)$.
\end{itemize}
\end{definition}

 We let $\mathscr{T}$ denote the intersection of the family of the topologies in (ii).
We write $q\in I^+(p)$, $(p,q)\in I$, or $q\gg p$ if $d(p,q)>0$, and say that $q$ is in the {\it chronological future} of $p$, and dually in the past case. By the reverse triangle inequality (i),
$p\ll q$ and $q\ll r$ implies $p\ll r$.

\begin{remark}\label{R1} $\empty$
\begin{enumerate}
\item
It follows from the finiteness of $d$ and the reverse triangle inequality in (i) that $d(x,x)=0$ for all $x\in X$. This property
reads $p\not\ll p$, and we refer to it as {\it chronology}.
\item
By (iii) there is just one point, if it exists, denoted $i^0$ and called {\it spacelike boundary}, such that $d(i^0,p)=d(p,i^0)=0$ for every $p\in X$.
Any point $q$ different from $i^0$ has another point $z$ in its chronological past or future.
\item Condition (ii) is very reasonable as it demands two properties that are  natural for  a topology related to $d$. We  will prove below
 that all the topologies $T$ that satisfy  (ii)  induce the same topology on $X\backslash \{i^0\}$,  i.e. they can only differ on the neighborhood system of $i^0$. We shall prove that $\mathscr{T}$ satisfies the same properties in (ii) of $T$, but it has the smallest possible neighborhood system for $i^0$.  The topology $\mathscr{T}$
    will be referred as {\em the topology} of the bounded Lorentzian metric space. Indeed, we shall find a subbasis that is completely determined by $d$. Observe that on $T$ the condition ``$d$ is continuous'' implies  that there must be enough open sets in ${T}$, while the condition on the compactness of the sets $\{d\ge \epsilon\}$ implies that there are not too many open sets (as coarsening the topology makes it more easy to accomplish compactness of a certain set). The condition thus balances the family of open sets by means of the function $d$.
\item If $X$ is assigned a topology $T$ and on $X\times X$ is assigned the product topology $T\times T$, then for each $p\in X$ the maps  $x\mapsto (p,x)$ and $x \mapsto (x,p)$  are homeomorphisms onto their images. If, additionally, $d$ is continuous in the topology $T\times T$, then the compositions $d_p:=d(p, \cdot)$, $d^p=d(\cdot, p)$, are $T$-continuous.  Thus these properties hold for any $T$ in (ii). We conclude that $d_p$ and $d^p$ are continuous for $\mathscr{T}$. We
 shall prove only later that, more strongly, $d$ is continuous in the product topology  $\mathscr{T}\times \mathscr{T}$. The fact that the sets $\{d \ge \epsilon\}$ are compact in $\mathscr{T}\times \mathscr{T}$ is clear because  $\mathscr{T}\times \mathscr{T}$ is coarser than any $T\times T$.
\item By 1.\ if $\vert X\vert=1$ then the only point in $X$ is $i^0$. If $\vert X\vert>1$ then $X$ contains at least two points $p,q$ with $p\ll q$, ($X$ has at least three points if there is $i^0\in X$) thus the set $\{d\ge \epsilon\} \subset X\times X$, is not only compact, as by assumption (ii), but also non-empty, thus $d$ reaches a positive maximum on $X\times X$. In any case we have a bound on $d$ which justifies the adjective {\it bounded} in the definition.
\item We let the set consisting of only one point be included into the definition. The purpose is to have a well defined Gromov-Hausdorff limit for ``collapsing'' spaces.
\end{enumerate}
\end{remark}

Figure \ref{ksz} provides a simple example of bounded Lorentzian metric space. Observe that a strip of Minkowski 1+1 spacetime (with parallel spacelike geodesics as boundary) is not a bounded Lorentzian metric space.

\begin{figure}[ht]
\centering
\includegraphics[width=8cm]{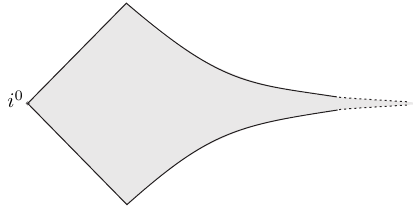}
\caption{This  subset of Minkowski 1+1 spacetime, with the induced Lorentzian distance, provides an example of non-compact bounded Lorentzian metric space $(X,d)$.  It is understood that the boundary of the gray region belongs to $X$, save for the  point $i^0$ that  might or might not belong to $X$.} \label{ksz}
\end{figure}

\begin{remark}[Necessity/uniqueness of the new structure]
When it comes to introducing new objects, it is natural to ask whether the new definitions are based on arbitrary choices or whether they are essentially unique. In other words: what are the chances that a better definition of `Lorentzian metric space' could be found? Inspection of our definition shows that properties (i) and (iii) are rather  weak and would certainly be shared by any other possible approach.
As for property (ii), this is also extremely weak. Observe that our goal is to describe low regularity versions of ``globally hyperbolic'' spacetimes, boundedness of the Lorentzian distance and global hyperbolicity being contained in the request that the sets $\{d \ge \epsilon\}$ are compact. The continuity and finiteness of  the Lorentzian distance are standard consequences of global hyperbolicity in the smooth setting \cite{beem96,minguzzi08e} and it is very desirable to preserve them under low regularity. For instance, in our study, the continuity of $d$ will be essential in order to construct (continuous) time functions, with which our causal curves are parametrized. It might be very difficult to define the parametrization of curves without a continuous function $d$.

In short, we expect that any other approach differing from our own should have the following features: (1) it is not a truly Lorentzian metric approach, as other ingredients are needed (an example is the theory of length spaces by Kunzinger-S\"amann,  where an auxiliary metric is needed), (2) it does not really describe a low regularity version of globally hyperbolic spacetime as, say, the Lorentzian distance $d$ is not continuous.

Last, the limitation to the bounded case is quite mild as the very nature of spacetime lies in its description over compact subsets, nonlocal issues being more of a mathematical nature. In this sense our formulation in terms
of a Lorentzian metric, without the need of auxiliary ``positive signature'' metrics or other unobserved fields, seems to be a key contribution to non regular spacetime geometry.

\end{remark}

\subsection{Topology}

In what follows we denote $\mathring{X}=X\backslash \{i^0\}$, and assigned a topology on $X$,  we denote with $\mathring{{T}}$ the topology induced on $\mathring{X}$. Of course, if $i^0\notin X$ then $\mathring{X}=X$ and $\mathring{{T}}={T}$.

We observed above that the map
\begin{align*}
d_p\colon X&\to \mathbb{R}\\
q&\mapsto d_p(q):=d(p,q)
\end{align*}
 is $T$-continuous. Similarly, $q\mapsto\, d^r(q):=d(q,r)$ is $T$-continuous. Thus the sets of the form
\begin{equation} \label{vos}
\{q: a<d(p,q)<b\}\cap  \{q: c<d(q,r)<e\}
\end{equation}
for any $p,r\in X$, and $a,b,c,e\in \mathbb{R}\cup \{-\infty,\infty\}$ are open in ${T}$ and hence in $\mathscr{T}$.
They form a subbasis for the topologies $\mathring{T}$, $\mathring{\mathscr{T}}$, that therefore coincide. The proof will require some work, cf.\  Corollary  \ref{vjw}.

\begin{proposition} \label{bjw}
Let $T$ be any topology that satisfies (ii).
$T$, $\mathscr{T}$, $\mathring{{T}}$ and $\mathring{\mathscr{T}}$   are  Hausdorff. $\mathring{{T}}$, $\mathring{\mathscr{T}}$ are  locally-compact and $\sigma$-compact.
\end{proposition}

\begin{proof}
{\it Hausdorff property:} We already know that $d_p$ and $d^r$ are continuous. Let $x\ne y$. Since $d$ distinguishes points we can find $z$ such that $d(z,x):=a\ne b:= d(z,y)$ (the other case being analogous).
Set $m=(a+b)/2$. Then only one among $x$ and $y$ belongs to the open set $\{r: d(z,r)< m\}=d_z^{-1}((-\infty,m))$, the other belonging to $\{r: d(z,r)> m\}=d_z^{-1}((m,\infty))$, thus $x$ and $y$
are separated by open sets.

{\it $\sigma$-compactness:} The sets $K_n=\{d\ge 1/n\}$ are compact in $X\times X$. Thus $\{\pi_1(K_n)\cup \pi_2(K_n)\}_n$ is a covering by compact sets as every point has a point in the future or past with distance larger than $1/n$ for sufficiently large $n$.  Here $\pi_{1,2}\colon X\times X\to X$ denote the canonical projections onto the first and second factor
respectively.

{\it Local compactness:} If $p$ is a point that admits a point $r$ in its chronological past, $\{q: d(r,q) \ge \epsilon\}=\pi_2(\{d\ge \epsilon \}\cap \pi_1^{-1}(r))$ (the topology is Hausdorff so $T_1$ hence points are closed sets)
is a compact neighborhood of it for some $\epsilon$ as it contains $\{q: d(r,q) > \epsilon\}$ which is open. The case in which $p$ admits a point in its
chronological future is analogous.
\end{proof}

\begin{proposition} \label{bix}
Suppose that $i^0\in X$. Let ${T}$ be a topology that satisfies (ii). Then there is a coarser topology $T'$ that satisfies (ii) such that $\mathring{T}'=\mathring{T}$ and such that  $(X,{T}')$ is compact. Moreover, a subbasis for the $T'$-neighborhood system of $i^0$ is provided by the complements of the compact sets of the form $\{d_p\ge b\}$, $\{d^p\ge b\}$ with $b>0$, $b\in \mathbb{R}$, $p\in X$. It is the smallest possible neighborhood system of $i^0$ for a topology that satisfies (ii).
\end{proposition}

We stress that compactness of $X$ does not imply that $i^0\in X$. Some causets will provide a counterexample, see Section \ref{cau}.
\begin{proof}
Let $T$ be a topology that satisfies (ii). By Eq.\ (\ref{vos}) sets of the form
\[
\{q: d(p,q)<b\}\cap  \{q: d(q,r)<e\}=X\backslash \{d_p\ge b\}\cap X\backslash \{d^r\ge e\}
\]
with $p,r\in X$, $b,e\in \mathbb{R}$, $b,e> 0$, belong to the  open neighborhood system of $i^0$ (note that $\{d_p\ge b\}=\pi_2(\pi_1^{-1}(p)\cap \{d\ge b\})$ is the projection of the intersection of a closed set and a compact set, hence compact, and similarly for $\{d^r\ge e\}$). The open neighborhood system of $i^0$ in $T$ contains at least the family $\mathcal{F}$ of finite intersections of sets of the above form.

If $T$ had other neighborhoods of $i^0$ that do not contain elements of $\mathcal{F}$ then it is possible to reduce the neighborhoods of $i^0$ to those of the above system while preserving continuity of $d$. This is obvious for any pair $(p,q)$ with $p,q\ne i^0$, since we are not altering the neighborhood system of  points different from $i^0$.

Let $\epsilon>0$. The compact set $B=\pi_1(\{d\ge \epsilon\})$ consists of points $p$ that admit some $r$ such that $d(p,r)\ge \epsilon$, thus $p$ is distinguished from $i^0$ and there is some $\tilde r\in X$ and $\delta>0$, such that $p\in \{d^{\tilde r} > \delta\}$ or $p\in  \{d_{\tilde r} > \delta\}$.
This means that $B$ can be covered with sets of this form. Pass to a finite subcovering $\{O_i\}$ where $O_i=\{d^{\tilde r_i} > \delta_i\}$ or $O_i=\{d_{\tilde r_i} > \delta_i\}$ and let $U=\cap_i X\backslash C_i \in \mathcal{F}$ where $C_i=\{d^{\tilde r_i} \ge \delta_i\}$ or $C_i=\{d_{\tilde r_i} \ge \delta_i\}$. We have  $U\cap B=\emptyset$ and hence, for each $x\in U$, $d(x,\cdot)<\epsilon$.

Similarly, we can find $V\in \mathcal{F}$ such that for every $x\in V$, $d(\cdot,x)<\epsilon$. This shows that the neighborhood system $\mathcal{F}$ of $i^0$ is sufficient to guarantee continuity of $d$ on $\pi_1^{-1}(i^0)\cup\pi_2^{-1}(i^0)$. We let $T'$ be the topology that has the same neighborhood system of $T$ for points in $\mathring{X}$, and $\mathcal{F}$ for $i^0$.

Let us prove that $(X,T')$ is compact. Indeed, the open covering must contain an open neighborhood $O$ which contains an element of $\mathcal{F}$. Let this element of $\mathcal{F}$ have the form $\cap_i X\backslash C_i=X\backslash \cup_i C_i$ where $C_i$ are compact sets, then $O$ and the finite family of open sets covering   $\cup_i C_i$ give the searched finite covering.
\end{proof}

\begin{proposition} \label{bis}
Let $\mathcal{A}$ be a family of open subsets for a Hausdorff, locally compact topology such that (a)  $\mathcal{A}$ separates points and (b) every point admits an open neighborhood in $\mathcal A$ contained in a compact set. Then $\mathcal{A}$ is a subbasis for the topology.
\end{proposition}

\begin{proof}
Let $p\in X$ and let $O'\ni p$ be an open set. By (b) there are an  open neighborhood in $R\in \mathcal A$, $p\in R$, and a compact set $C\supset R$.
Then $p\in O'\cap R\subset C$. We are going to show that there is an open set $A$, finite intersection of elements in $\mathcal A$, such that $p\in A\subset O:=O'\cap R$. Indeed, for every $q\in C\backslash O$ we  can find an open neighborhood $B(q)\in \mathcal A $ of $q$ and an open neighborhood $A(q) \in \mathcal A$ of $p$ such that $A(q)\cap B(q)=\emptyset$. The open set  $A(q)\cap R$ is contained in $C$.  Let $\{B(q_i)\}$ be a finite covering of the compact set $C\backslash O$. Then  $A =\cap_i A(q_i)\cap R $ is an open neighborhood of $p$ contained in $C\cap O\subset O'$.
\end{proof}

\begin{corollary} \label{vjw}
Let ${T}$ be a topology that satisfies (ii) and let $T'$ be the topology mentioned in  Proposition \ref{bix}. The topologies
$T'$ and $\mathring{{T}}=\mathring{{T}}'$ are the initial topologies of the family of functions $\{d_p, d^p: p\in X\}$. More precisely, the sets  of the form
\begin{equation}  \label{don}
\{q: a<d(p,q)<b\}\cap  \{q: c<d(q,r)<e\}
\end{equation}
with $p,r\in X$ and $a,b,c,e,\epsilon\in \mathbb{R}\cup \{-\infty,\infty\}$ form a subbasis for the topologies.

As a consequence, $\mathscr{T}=T'$, $\mathring{\mathscr{T}}=\mathring{{T}}'=\mathring{T}$, thus $d$ is continuous in $\mathscr{T}\times \mathscr{T}$, the sets $\{d\ge \epsilon\}$ for $\epsilon>0$ are $\mathscr{T}\times \mathscr{T}$-compact (namely $\mathscr{T}$ satisfies the properties of point Def.\ \ref{D1}(ii) and hence it is the coarsest topology that satisfies those properties) and $\mathscr{T}$ is uniquely determined by $d$,  having as subbasis the elements in (\ref{don}). Moreover, if $i^0\in X$ then $(X,\mathscr{T})$ is compact.
\end{corollary}

\begin{proof}

Let us first give the proof for $\mathring{{T}}$.
Let $\mathcal A$ be the family  of the form (\ref{don}). From the proof of the  Hausdorff property we  already know that they separate points.  As already observed if $o$ is a point that admits a point $p$ in its chronological past (the other case is analogous),
\[
\{q: d(p,q) \ge \epsilon\}=\pi_2(\{d\ge \epsilon \}\cap \pi_1^{-1}(p)),
\]
is a compact neighborhood of it for some rational $\epsilon>0$ as it contains $\{q\colon d(p,q) > \epsilon\}$ which is open. But the latter sets belong to $\mathcal A$, thus $\mathcal A$ satisfies both (a) and (b) in Proposition \ref{bis}, thus $\mathcal{A}$ gives a subbasis for the topology of $\mathring{{T}}$.

For ${T}'$ the proof is identical but easier since local compactness follows from compactness.

All the sets in (\ref{don}) are also open for $\mathscr{T}$, thus $T'=\mathscr{T}$, which implies that $d$ is also continuous with respect to $\mathscr{T}\times \mathscr{T}$.
\end{proof}

\begin{proposition}[One point compactification] \label{bir}
$\empty$
\begin{itemize}
\item[(a)] Let $(X,d)$  be a bounded Lorentzian metric space that contains $i^0$ and let $\mathscr{T}$ be its topology. The space $(\mathring{X}, \mathring{d})$ obtained by removing $i^0$ and considering the induced distance $\mathring{d}$ is a bounded Lorentzian metric space whose topology in the sense of point (ii) is $\mathring{\mathscr{T}}$ - the induced topology.
\item[(b)] Let $X$ be a bounded Lorentzian metric space that does not include $i^0$. Then adding an abstract element $i^0$ whose distance to  all the other elements in $X$ is defined to be zero gives a compact  bounded Lorentzian metric space $(\tilde X, \tilde d)$. Let $\mathscr{T}$, $\tilde{\mathscr{T}}$ be the respective topologies according to point (ii), then $(\tilde X,\tilde{\mathscr{T}})$ is the  standard one-point compactification of $(X,\mathscr{T})$.
\end{itemize}
\end{proposition}

\begin{proof}
For the first point, let $T$ be a topology that satisfies $(ii)$ for $(X,d)$. The induced topology $\mathring{T}$  satisfies (ii) for $(\mathring{X}, \mathring{d})$, thus condition (ii) is satisfied. As the other points are clear,  $(\mathring{X}, \mathring{d})$ is a bounded Lorentzian metric space.
Actually, we proved in Corollary \ref{vjw} that  $\mathring{T}=\mathring{\mathscr{T}}$, where $\mathscr{T}$ is the topology of $(X,d)$, thus $\mathring{\mathscr{T}}$ is the topology of $(\mathring{X}, \mathring{d})$ (remember that any topology that satisfies $(ii)$ is {\em the topology} for spaces that do not contain $i^0$).

As for the second point, there is certainly a topology $\tilde T$ that makes the extended distance continuous. As shown in  Prop.\ \ref{bix} it is sufficient to add to the topology $\mathscr{T}$ the found neighborhood system for $i^0$. This topology $T$ is sufficiently coarse that the sets of the form $\{d\ge \epsilon\}$ are compact in the product topology $\tilde T\times \tilde T$. Indeed, from a covering of $\tilde T\times \tilde T$ we pass to a covering of $\mathscr{T}\times \mathscr{T}$ (as  $\{d\ge \epsilon\}$ does not intersect $\pi_1^{-1}(i^0)\cup \pi^{-1}_2(i^0)$) from here to a finite covering of $\mathscr{T}\times \mathscr{T}$ and hence to a finite covering of $\tilde T\times \tilde T$.

For the final statement, let $K\subset X$ be a compact set for $(X,\mathscr{T})$. For every point $x\in K$ we can find some points $y\in X$ such that $d(x,y)>\epsilon>0$ or $d(y,x)>\epsilon>0$. We can pass to a finite covering that includes only  open sets of the form $d^{y_i}> \epsilon_i>0$ or $d_{y_i}> \epsilon_i>0$. This means that $K$ is included in a finite union of compact sets of the form $d^{y_i}\ge \epsilon_i/2>0$ or $d_{y_i}\ge \epsilon_i/2>0$, which proves that the neighborhood system for $i^0$ generated by sets of the form $X\backslash K$ is no larger than that of Prop.\ \ref{bix} (which is also obtained by taking intersections of complements of compact sets).
\end{proof}

In the next result we shall consider the space $C^0(X,\mathbb{R})\times C^0(X,\mathbb{R})$ endowed with the topology of uniform convergence, namely that induced by the norm
\[
 \Vert (f,g)\Vert=\max\{\sup_ z\vert  f(z) \vert, \sup_ z\vert  g(z) \vert\}.
\]
Consider the (Kuratowski-type)  map
\begin{align*}
I\colon X &\to C^0(X,\mathbb{R})\times C^0(X,\mathbb{R}) \\
x& \mapsto (d_x,d^x)
\end{align*}
Since $d$ is bounded, the image of $I$ actually lies in the subspace of bounded continuous functions.

By {\em topological embedding} we mean  homeomorphism on the image.

\begin{proposition}
Let $(X,d)$ be a bounded Lorentzian metric space.
The map $I$ is a topological embedding for the topology $\mathscr{T}$ on $X$.
\end{proposition}

\begin{proof}
By Definition \ref{D1}(iii) $I$ distinguishes points, i.e. $I$ is injective.

Let us prove the continuity of the map $I$ at $x$.
Let $\varepsilon>0$.
According to  Corollary \ref{vjw} for every $z\in X$ there exists $V^z_1,V^z_2\in \mathscr{T}$, neighbourhoods of $x$ and $z$ respectively such that for all $(v,w)\in
V^z_1\times V^z_2$ we have
\[
|d(x,z)-d(v,w)|<\frac{\varepsilon}{2}.
\]

The set $\pi_2(d\ge \varepsilon/2)=\{w\in X|\; \exists r\in X: d(r,w)\ge \varepsilon/2\}$
is $\mathscr{T}$-compact in $X$. Therefore we can choose a finite subcover $\{V^{z_i}_2\}_{1\le i\le k}$
of $\pi_2(d\ge \varepsilon/2)$ from among $\{V^z_2\}_{z\in X}$. Let $V_x:= \cap_{i=1}^k V^{z_i}_1$, then for  $y\in V_x$ and any $w\in X$
we have
\[
|d(x,w)-d(y,w)|<\varepsilon .
\]
Indeed, if $w\notin  \pi_2(d\ge \varepsilon/2)$, then $d(x,w),d(y,w)<\varepsilon/2$, while if $w\in  \pi_2(d\ge \varepsilon/2)$, we have $w\in V^{z_i}_2$ for some $i$, and hence
\[
\vert d_x(w)-d_y(w) \vert =|d(x,w)-d(y,w)|\le |d(x,w)-d(x,z_i)|+|d(x,z_i)-d(y,w)|<\varepsilon.
\]
An analogous argument for the functions $d^x$, proves that there is an open neighborhood $U_x$ of  $x$  such that  for every $y\in U_x$ and $w\in X$,
\[
|d^x(w)-d^y(w)|<\varepsilon
\]
This shows that for $y\in V_x\cap U_x$, $\Vert I(x)-I(y)\Vert\le \varepsilon$, which, given the arbitrariness of $\epsilon$, proves continuity of $I$ at $x$.

It remains to show that $I$ is a homeomorphism onto its image. In case $i^0\in X$ this follows from the compactness of $X$ (Proposition \ref{bix}) and the target being Hausdorff.

For $i^0\notin X$ we give two proofs, (a) and (b). Proof (a) uses Prop.\ \ref{bir} while (b) is more direct.

(a). We can one-point compactify the space to $(\tilde X, \tilde d)$, then $\tilde I: \tilde X\to C^0(\tilde X, \mathbb{R})\times C^0(\tilde X, \mathbb{R})$ is a homeomorphism on the image.
Actually each point $(f,g)$ in the image is such that $f(i^0)=g(i^0)=0$, thus we can replace $C^0(\tilde X, \mathbb{R})$ with $C^0_0(\tilde X, \mathbb{R})$, the continuous functions that vanish at $i^0$.
Let $C_0^0(X, \mathbb{R})$ be the space obtained by restricting the functions in $C^0_0(\tilde X, \mathbb{R})$. The topology in $C_0^0(X, \mathbb{R})$ coincides with that induced from $C^0(X, \mathbb{R})$.

The restriction of $\tilde I$ to the open set $X=\tilde X\backslash \{i^0\}$ is then again a homeomorphism onto the image.  Since the topology in $C_0^0(X, \mathbb{R})$ coincides with that induced from $C^0(X, \mathbb{R})$, $I:X\to C^0(X, \mathbb{R})\times C^0(X, \mathbb{R})$ is a homeomorphism on the image.

(b). In case $i^0\notin X$ note that for all $\varepsilon >0$ the restriction of $I$ to
\[
X_\varepsilon:=\{x\in X|\; \exists z\in X: d(x,z)\ge \varepsilon\}\cup \{y\in X|\; \exists z\in X: d(z,y)\ge \varepsilon\}
\]
is a
homeomorphism onto its image as the  set is compact and the target is Hausdorff. Now for $V\subset X$ open and $x\in V$ choose $\varepsilon >0$ with
$\|(d_x,d^x)\|>\varepsilon$. Since $I$ is continuous we know that $V_x:=I^{-1}(B_{\varepsilon/2}(I(x)))$ is open. For every $y\in V_x$ we have $\|(d_y,d^y)\|\ge \|(d_x,d^x)\|-\|(d_y-d_x,d^y-d^x)\|>\varepsilon /2$ thus $V_x\subset X_{\varepsilon/2}$,
 and
\[
I(V\cap V_x)=I(V\cap V_x\cap X_{\varepsilon/2})=B_{\varepsilon/2}(I(x))\cap I(X_{\varepsilon/2}\cap V).
\]
As shown above $I(X_{\varepsilon/2}\cap V)$ is open in $I(X_{\varepsilon/2})$.
By the definition of the induced topology there exists an open subset $W\subset C^0(X,\mathbb{R})\times C^0(X,\mathbb{R})$  with $I(V\cap V_x)= I(X_\varepsilon/2)\cap W$.
 As $I(V\cap V_x)=I(V\cap V_x)\cap B_{\varepsilon/2}(I(x))$ we can assume w.l.o.g.\ that $W\subset B_{\varepsilon/2}(I(x))$. As $B_{\varepsilon/2}(I(x))\subset I(X_{\varepsilon/2})$ we have $B_{\varepsilon/2}(I(x))\cap I(X_{\varepsilon/2})= B_{\varepsilon/2}(I(x))\cap I(X)$ we conclude $I(V\cap V_x)=W\cap I(X)$, i.e. it is open in $I(X)$.
\end{proof}

We recall that a Polish space is a separable completely metrizable topological space.
\begin{theorem}
Let $(X,d)$ be a bounded Lorentzian metric space.
 Then $(X,\mathscr{T})$ is a Polish space.
\end{theorem}

\begin{proof}
The topology in $C^0(X,\mathbb{R})\times C^0(X,\mathbb{R})$ is metric, and so is that induced on $I(X)$ and $\mathscr{T}$ (by the homeomorphism).
But $(X,\mathscr{T})$ is compact or $\sigma$-compact, and hence  a Lindel\"of metrizable space, thus second countable.
Every   second countable locally compact Hausdorff space (hence $\sigma$-compact) is actually Polish. Indeed, it has a one-point compactification that is second countable (of course there is no need to compactify if $X$ is already compact). From here the one-point compactification $X^*$ is second-countable, hence metrizable, hence completely metrizable  \cite[23C]{willard70}. But local compactness implies that $X$ is an open subset of $X^*$ which implies that $X$ is Polish \cite[Part II, Chap.\ IX, Sect.\ 6]{bourbaki66}.
\end{proof}

\begin{proposition}\label{prop(iv)}
Let $\mathcal{S}$ be a dense subset of a bounded  Lorentzian metric space $(X,d)$ (it exists as the metric space is  separable). Then $(X,d)$ satisfies the following property
\begin{itemize}
\item[(iv)] The distinguishing point $z$ in Def.\ref{D1}(iii) can be found in the countable subset $\mathcal{S}$.
\end{itemize}
\end{proposition}

It will be referred to as property (iv) of a bounded Lorentzian metric space, but it is understood that it follows from the other three.

Since $i^0$ does not help to distinguish any pair of points, if the objective is just finding a countable set $\mathcal{S}$ that distinguishes points (dropping the dense condition) we can remove $i^0$ from $\mathcal{S}$.

We mentioned that any point $q$ different from $i^0$ has another point $z$ in its chronological past or future. This result shows that $z$ can be found in $\mathcal{S}$.

\begin{proof}
With reference to (iii), suppose that the first case applies (the other case being analogous), i.e.\ $d(x,z)\ne d(y,z)$, with $z\in X$, then we can find some $a\in \mathbb{R}$ such that $d(x,z)< a < d(y,z)$ (or similarly with $>$ in place of $<$), which implies that $z\in  d_x^{-1}((-\infty,a))\cap d_y^{-1}((a,+\infty))$ where the right-hand side is open.  Thus we can find $s\in \mathcal{S}$, such that $s\in  d_r^{-1}((-\infty,a))\cap d_y((a,+\infty))$, which implies $d(x,s)< a < d(y,s)$ and hence $d(x,s)\ne  d(y,s)$.
\end{proof}

\begin{remark} \label{vkd}
In the proof of the Hausdorff property, Prop.\ \ref{bjw}, thanks to (iv),  we can assume $z\in \mathcal{S}$ thus there is a countable subset $\mathscr{Q}\subset \mathscr{T}$ that separates points (In the proof of Prop.\ \ref{bjw}  just take $m$ close to $(a+b)/2$ and belonging to the rationals).
\end{remark}

\begin{remark}  \label{bix-2}
It is also easy to prove, by arguing as in the last paragraph of Prop.\ \ref{bir} with $y\in \mathcal{S}$, that for $i^0\in X$ a subbasis for the neighborhood system of $i^0$ is provided by the complements of the compact sets of the form $\{d_p\ge b\}$, $\{d^p\ge b\}$ with $b>0$, $b\in \mathbb{Q}$, $p\in \mathcal{S}$.
\end{remark}

\begin{corollary} \label{vju}
Let $\mathcal{S}$ be a countable set that satisfies property (iv), e.g.\ a countable dense subset.
$\mathscr T$ and $\mathring{\mathscr{T}}$ are the initial topologies of the family of functions $\{d_p, d^p: p\in \mathcal{S}\}$. More precisely, the sets  of the form
\begin{equation}  \label{don2}
\{q: a<d(p,q)<b\}\cap  \{q: c<d(q,r)<e\}
\end{equation}
with $p,r\in \mathcal{S}$ and $a,b,c,e,\epsilon\in \mathbb{Q}\cup \{-\infty,\infty\}$ form a countable subbasis for the topologies.
\end{corollary}

Observe that if the subbasis is countable, then so is the basis obtained by the finite intersections  (a countable set has countably many finite subsets).

\begin{proof}
Let us first give the proof for $\mathring{\mathscr{T}}$.
Let $\mathcal A$ be the family  of the form (\ref{don2}). We just observed that they separate points. If $o$ is a point that admits a point $p\in \mathcal{S}$ in its chronological past (the other case is analogous),
\[
\{q: d(p,q) \ge \epsilon\}=\pi_2(\{d\ge \epsilon \}\cap \pi_1^{-1}(p)),
\]
is a compact neighborhood of it for some rational $\epsilon>0$ as it contains $\{q\colon d(p,q) > \epsilon\}$ which is open. But the latter sets belongs to $\mathcal A$, thus $\mathcal A$ satisfies both (a) and (b) in Proposition \ref{bis}, hence $\mathcal{A}$ is a subbasis for the topology.

For $\mathscr{T}$ the proof is identical but easier since local compactness follows from compactness.
\end{proof}

We can give a characterization of bounded Lorentzian-metric space which can also be used as a definition.

\begin{proposition} \label{vxg}
A space $(X,d)$ consisting of a set $X$, a map $d: X\times X \to [0,\infty)$, and for which there is a point $i^0$ at zero distance from any other point, is a bounded Lorentzian-metric space if and only if
\begin{itemize}
\item[(i)] For every $x,y,z\in X$ such that $d(x,y)>0$, $d(y,z)>0$ we have
    \[
    d(x,z)\ge d(x,y) +d(y,z).
    \]

\item[(ii')] There is a topology $T$ for which $d$ is continuous in $T\times T$ and  $X$ is compact.

\item[(iii)] $d$ distinguishes points, i.e.\ for every pair $x,y$, $x\ne y$ there is $z$ such that $d(x,z)\ne d(y,z)$ or $d(z,x) \ne d(z,y)$.
\end{itemize}
In this case there is a coarsest topology $\mathscr{T}$ with the property $(ii')$ (this is {\em the topology} of $(X,d)$).
\end{proposition}

\begin{proof}[Proof of the equivalence.]
We already proved, see  Corollary \ref{vjw}, that Definition  \ref{D1}(i)-(iii)  implies that there is a topology $T:=\mathscr{T}$ with the properties of (ii'). Any topology $T'$ that satisfies (ii') is such that $d_p$ and $d^p$ are continuous, thus the found subbasis (\ref{don}) for $\mathscr{T}$ consists of open sets for $T'$, hence $\mathscr{T}$ is the coarsest topology that satisfies (ii').

For the converse, the compactness of $X\times X$ in $T\times T$  implies that the closed subsets of the form $\{d\ge \epsilon\}$ are compact with respect to $T\times T$, thus
 Definition   \ref{D1}(ii) is satisfied.
\end{proof}

\begin{definition}
A map $f\colon X\to Y$ between bounded Lorentzian-metric spaces $(X,d_X)$ and $(Y,d_Y)$ is {\it distance preserving} if
$$d_Y(f(x),f(x'))=d_X(x,x')$$
for all $x,x'\in X$.

A bijective and distance preserving map is an {\it isometry}.
\end{definition}

Obviously $X$ includes $i^0_X$ iff $Y$ includes $i^0_Y$, in which case $f(i^0_X)=i^0_Y$.
We have

\begin{theorem} \label{vps}
Let $X$, $Y$ be bounded Lorentzian-metric spaces.
If $f:X\to Y$  is an isometry then it is a homeomorphism.
\end{theorem}

\begin{proof}
Indeed, a subbasis of $\mathscr{T}_Y$ is provided by the sets in (\ref{vos}), whether $i^0_Y\in Y$ or not, which implies that their preimage have the same form with $d_Y$ replaced by $d_X$ and so are elements of a subbasis for $\mathscr{T}_X$. This proves that $f$ is continuous (using the surjectivity of $f$). The proof that $f^{-1}$ is continuous is analogous (using the surjectivity of $f^{-1}$).
\end{proof}

\begin{corollary}
Let $f:X\to Y$  be an isometry. If $X$ and $Y$ do not contain the spacelike boundary then their one-point compactifications are isometric and hence homeomorphic.  Just extend $f$ as follows $f(i^0_X):=i^0_Y$. If $X$ and $Y$ do  contain the spacelike boundary then we get an isometry and hence an homeomorphism of $\mathring{X}$ and $\mathring{Y}$, through the restriction $f\vert_{\mathring{X}}$.
\end{corollary}

These results clarify that it is not restrictive to work with bounded Lorentzian-metric spaces that contain $i^0$, or with those that do not contain $i^0$.

\subsection{Distance quotient} \label{quo}
Let $(X,\mathscr{T})$ be a topological space endowed with a continuous function $d\colon X\times X\to [0,\infty)$.

We define an equivalence relation ``$\sim$'' on $X$ by
$p\sim q$ if $d(r,p)=d(r,q)$ and $d(p,r)=d(q,r)$ for every $r\in X$. It is easy to check that this is indeed an equivalence relation. Let $\tilde X=X/\sim $ be the
quotient,  let $\tilde{\mathscr{T}}=\mathscr{T}/\sim $ be the quotient topology, and let $\pi: X \to \tilde X$ be the quotient projection. The function $d$ passes
to the quotient. Indeed if $p,p'$ are two representatives of a class $[p]$ and $q,q'$ are two representative of a class $[q]$ we have
\[
d(p,q)=d(p',q)=d(p',q'),
\]
so that we can define a function $\tilde d([p],[q]):=d(p,q)$.  It is easy to check that $\tilde d$ is continuous.

Thus we have obtained a quotient topological space $(\tilde X,\tilde{\mathscr{T}})$  endowed with a continuous function $\tilde d\colon \tilde X\times \tilde X\to [0,\infty)$.

Let us observe how the properties for $d$ and $\mathscr{T}$ on $X$ are related to the analogous properties on $\tilde X$.

\begin{itemize}
\item $d$ satisfies the reverse triangle inequality, as in (i), iff $\tilde d$ does.
\item $\tilde d$ distinguishes points (i.e.\ satisfies (iii)) even if $d$ does not.
\item   There is a point of $X$ (possibly non unique) at vanishing distance from all the other points iff  $\tilde X$ contains the spacelike boundary point.

\end{itemize}

The next result is not obvious and requires a proof

\begin{proposition}
Suppose that $T$ is a topology of $X$  that satisfies (ii) then the quotient topology $\tilde T$ on $\tilde X$ satisfies (ii). If additionally, $X$ satisfies  (i) then $\tilde X$ is a bounded Lorentzian metric space.
\end{proposition}

\begin{proof}
Clearly every set of the form $d^{-1}((a,b))\subset X\times X$ with $a,b$ in the extended real line is projectable to $\tilde X\times \tilde X$.
Thus, $d$ is continuous for some product  topology $T\times T$ iff $\tilde d$ is continuous for the product of the quotient topology $\tilde T\times \tilde T$. Since the quotient
projection is continuous, the set $\{\tilde d\ge \epsilon\} $  is compact, as it is the $\pi\times \pi$-projection of a compact set.
\end{proof}

\section{Examples}

\subsection{Causets} \label{cau}

\begin{definition} \label{D2}
A finite set $S$ endowed with a function $d:S\times S\to [0,\infty)$ is called a {\it causet} if it satisfies points (i) and (iii) of Definition \ref{D1}.
\end{definition}

\begin{proposition} \label{bux}
Every causet is a bounded Lorentzian-metric space, where
 $\mathscr{T}$ for which $d$ is continuous is the discrete topology.
\end{proposition}

\begin{proof}
The discrete topology satisfies (ii), indeed $d$ is continuous in the product discrete topology and for every $\epsilon$  the set  $\{d\ge \epsilon\}$ is finite by the cardinality of  $S\times S$ hence compact. This proves that every causet is a bounded Lorentzian-metric space.

Now let us prove that any topology $T$ that satisfies (ii) is the discrete topology and hence that $\mathscr{T}$ is the discrete topology.

Let $p\in S$, we want to show that $\{p\}$ is open for $T$.
We know from Section \ref{bpo} that the functions $d_r, d^r$ are $T$-continuous. Let $r\in S$ and let $D_r=\{d(r,s), s\in S\}$ and
$D^r=\{d(s,r), s\in S\}$.
     For each $q$, since $D^q$ is finite, the number $d(p,q)$ is isolated from the other different numbers in $D^q$.  Thus we   can find $a_q,b_q\in \mathbb{R}$ such that $(a_q,b_q)$ contains only $d(p,q)$ and no other element of $D^q$. Similarly, we can find  $c_q,e_q\in \mathbb{R}$ such that $(c_q,e_q)$ contains only $d(q,p)$ among the numbers in $D_q$.
    Now the open set
    \[
    \bigcap_q \, (d^q)^{-1}((a_q,b_q))\,\bigcap_q\, (d_q)^{-1} ((c_q,e_q))
    \]
    contains $p$ by construction, but it cannot contain any other point $p'\ne p$, for otherwise $d$ would not distinguish $p$ from $p'$, contradicting point (iii). Thus $\{p\}$ is an open set.
\end{proof}

We recall that a {\it directed graph } is an ordered pair $G=(S,A)$ where $S$ is a set whose elements of $S$ are called {\it vertices} and $A$ is a set of ordered pairs of vertices, so $A\subset S\times S$, each pair being called an {\it arrow}. Given an arrow $(p,q)$, we call $p$ the {\it starting point} and $q$ the {\it ending point}.

Causets can be regarded as {\it weighted directed graphs}, where we say that $(p,q) \in A$ if $d(p,q)>0$. The distance $d(p,q)$ is then the {\it  weight} of the arrow $(p,q)$. The weight cannot be arbitrary as they have to satisfy the reverse triangle inequality.

We know that causets, being bounded Lorentzian-metric space, satisfy chronology, which means that the directed graph is acyclic. Observe that every point of the directed graph is connected to some other point by an arrow, save for one point $i^0$, if present, that would be isolated.

\begin{proposition}
The finite bounded Lorentzian-metric spaces are precisely the causets.
\end{proposition}

\begin{proof}
By Prop.\ \ref{bux} we know that every causet is a finite bounded Lorentzian-metric space. Conversely, a finite  bounded Lorentzian-metric space satisfies (i) and (iii), so it is a causet.
\end{proof}

It is convenient to define  causets via Definition \ref{D2} as it does not mention a topology.

\begin{proposition} \label{cna}
Let $S$ be a finite subset of a bounded Lorentzian-metric space endowed with the induced topology and distance. Then $S/\sim$ is a causet.
\end{proposition}

\begin{proof}
It is clear that $S/\sim$ inherits property (i) and that it satisfies (iii) as it is a distance quotient.
\end{proof}

\subsection{Subsets of globally hyperbolic spacetimes} \label{gob}

Let $\tilde X$ be a compact, causally convex subset of  a globally hyperbolic spacetime $(M,g)$. Denote with
\[
\p^0 \tilde X:=\{q\in \tilde X: \not\exists p\in \tilde X, p\ll q  \textrm{ and } \not\exists  r\in \tilde X, q\ll r \}.
\]
the {\it spacelike boundary} of $\tilde X$. Further, define $X:=\tilde{X}\setminus \partial^0 \tilde{X}$.

Let $d_X:=d\vert_{X\times X}$ where $d$ is the Lorentzian distance of $M$, and similarly let $d_{\tilde X}:=d\vert_{\tilde X\times \tilde X}$. By causal convexity $d_{\tilde X}$ is the supremum of the Lorentzian lengths of timelike connecting curves in $\tilde X$, but none of these curves can pass through $\p^0 \tilde X$, thus $d_X=d_{\tilde X}\vert_{X\times X}=d \vert_{X\times X}$ which is then the supremum of the Lorentzian lengths of the  causal connecting curves in $X$.

\begin{theorem} \label{vwm}
$(X,d_X)$ is a bounded Lorentzian-metric space that does not contain the spacelike boundary point $i^0$, and $\mathscr{T}$ is the topology induced on $X$ from the manifold topology.
\end{theorem}
A posteriori we can identify the abstract point $i^0$ with $\partial^0 \tilde{X}$.

\begin{proof}
Let us  verify the properties of bounded Lorentzian-metric space.

\begin{itemize}
\item[(i)] This is obvious by taking the restriction of analogous equations in $M$.
\item[(ii)]
We want to prove that  the induced topology $T$ on $X$ satisfies property (ii) of Definition \ref{D1}. Since $d$ is continuous we know that $d_X$ is continuous in the induced topology.  We know that $\{(p,q)\in \tilde X\times \tilde X: d(p,q)\ge \epsilon\}$ is compact in the product of the induced topologies because it is the intersection of a closed set with the compact set $\tilde X\times \tilde X$. But the previous set coincides with  $\{(p,q)\in  X\times  X: d(p,q)\ge \epsilon\}$, which therefore it is compact.

Now let us show that $\mathscr{T}$ is finer than $T$ and hence that $\mathscr{T}=T$. Let $O$ be an induced topology open set for a point $q\in X$,  so there is $O'$ open in the manifold topology of $M$ such that $O=X\cap O'$. Without loss of generality we can assume that $O'$ is a causally convex globally hyperbolic subset  for $(M,g)$ since these open sets generate the manifold topology for $M$.

 There will be $p\in X$, $d(p,q)>0$,   (or the other way). Any timelike curve from $p$ to $q$ is contained in $X$ (by causal convexity of $\tilde X$), so we can assume, without loss of generality, that $p\in O'$ and hence $p\in O$ (otherwise follow a timelike  curve from $p$ to $q$ and choose a point sufficiently close to $q$). Actually, in the same way we can find a second point $p'\in O$ sufficiently close to $q$, $p\ll p'\ll q$ and constants $a,b>0$ such that\footnote{The numbers $a$ and $b$ can be chosen so that there is a connected component of $Q$ inside $O'$ such that $q\in Q$. To show that there is no point of $Q$ outside $O'$, suppose that $r$ is such a point, then there is a causal curve connecting $p'$ to $r$ (because $d(p',r)>b$) that passes from the region of $O'$ where $d_p>a$, which would imply $d(p,r)>a$, a contradiction with the definition of $Q$.}
    $Q:=(d_X)_{p}^{-1}((-\infty, a))\cap (d_X)_{p'}^{-1}((b,\infty)) \subset X$
     is contained in $O'$ and contains $q$, thus $\mathscr{T}$ is finer than the induced topology (remember that for every $r$, $(d_X)_{r}$ and $(d_X)^{r}$ are   continuous in $\mathscr{T}$).

\begin{figure}[ht!]
\centering
 \includegraphics[width=9cm]{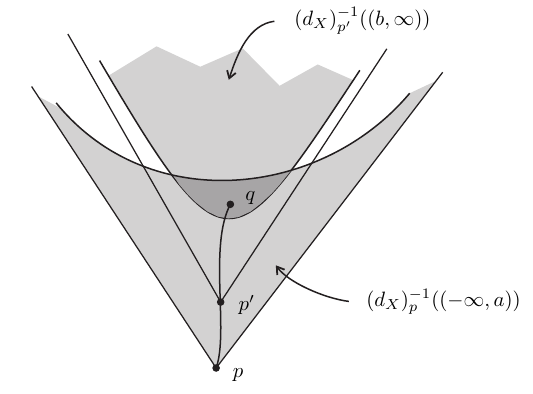}
 \caption{A step in the proof of Theorem \ref{vwm}.} \label{dyna}
\end{figure}

\item[(iii)]
   Let $x,y\in X$ with $x\ne y$. We can find $z\ll x$, $z\in X$ (or viceversa).
    By global hyperbolicity of $M$ either `there is  $p\in M$ such that $p\ll x$ but $p \not\ll y$' or $y\in J^+(x)$.
    In the latter case, $d(z,y)>d(z,x)$ because the maximizing causal curves connecting $z$ to $x$ and $x$ to $y$ have a corner at $x$. In the former case,
    we observe that $I^+(z)\cap I^{-}(x)\subset X$.
    Since $x\in I^+(p)$ and the last set is open in $M$, there will be some point $p'\in I^+(p)\cap [I^+(z)\cap I^{-}(x)]\subset X$. Now, $p'\not\ll y$, otherwise $p\ll y$, a contradiction. Thus $d(p', x)>0$ but $d(p',y)=0$. In both cases $x$ and $y$ are distinguished.
\end{itemize}
\end{proof}

\section{Distance preserving maps}

In this section we assume that the bounded Lorentzian-metric spaces do not include the spacelike boundary. Generalizations to the other cases are straightforward.

We recall that a bijection that preserves the distance is an {\it isometry}.

\begin{theorem} \label{ine}
Let $(X,d_X)$, $(Y,d_Y)$  be bounded Lorentzian-metric spaces, and let $f:X\to Y$ be distance preserving, then $f$ is injective.
\end{theorem}

\begin{proof}
Consider two distinct points $x,x'\in X$. Choose $z\in X$ such that $d_X(x,z)\ne d_X(x',z)$ (the other case being analogous). Then $d_Y(f(x),f(z))\ne d_Y(f(x'),f(z))$, which implies $f(x)\ne f(x')$.
\end{proof}

\begin{lemma} \label{azo0}
Let $(X,d)$ be a bounded Lorentzian-metric space
and let $f\colon X\to X$ be distance preserving, then $f$ is $\mathscr{T}$-continuous.
\end{lemma}

\begin{proof}
Since $\mathscr{T}$ is first-countable it is sufficient to prove the sequential continuity of $f$. Let $x\in X$ and let  $\{x_n\}$ be a sequence with $x_n\to x$. First note that the sequence $\{f(x_n)\}_n$
is contained in a compact subset of $X$. Indeed there exist $z_0\in X$ and $a>0$, with $d(z_0,x)> a>0$ (the case with $d(x,z_0)> a>0$ is treated similarly) so for any sufficiently large $n$,  $d(f(z_0),f(x_n))=d(z_0,x_n)> a>0$. Since $\pi_2\big(\pi_1^{-1}( f(z_0))\cap \{d\ge a\}\big)$ is compact and contains  $f(x_n)$ we get the claim.

Suppose that the claim is false, namely that the sequence of images $\{f(x_n)\}$ does not converge to $f(x)$. Then  we can assume, by passing to a subsequence, that $\{f(x_n)\}$
converges to a point $y\neq f(x)$. Choose $z\in X$ with $d(z, y)\neq d(z, f(x))$ or $d(y,z)\neq d(f(x),z)$. We consider just the former possibility, the treatment of the latter case being completely analogous.

Choose $\varepsilon >0$ such that $\vert d(z, f(x))- d(z, y)\vert > \varepsilon$, then we have  $\vert d(z, f(x))- d(z, f(x_n))\vert > \varepsilon$ for all $n$ sufficiently large.

It is convenient to observe that for each $q\in X$ we can find $p\in X$ such that $d(p,q)>\delta>0$, (or $r\in X$ such that $d(q,r)>\delta >0$ in which case we argue analogously), which implies that for any increasing sequence $k_m$, we have $d(f^{k_m}(p),f^{k_m}(q))>\delta>0$. By the compactness of $\{d \ge \delta\}$ we can conclude that  $f^{k_m}(q)$ admits a converging subsequence.

Now, for each such $n$ as in the last but one paragraph choose a strictly increasing sequence
$\{m^n_k\}_k\subset \mathbb{N}$ such that the sequences
\[
\{f^{m^n_k}(z)\}_k,\; \{f^{m^n_k+1}(x)\}_k,\text{ and }\{f^{m^n_k+1}(x_n)\}_k
\]
converge to points $a,b,c_n$ respectively.
Then by the continuity of $d$  we have
\begin{align*}
d(a, b)&=\lim_k d(f^{m^n_k}(z),f^{m^n_k+1}(x))= \lim_k d(z,f(x))=d(z,f(x)) ,\\
d(a,c_n)&=\lim_k d(f^{m^n_k}(z),f^{m^n_k+1}(x_n))=\lim_k d(z,f(x_n))=d(z,f(x_n)).
\end{align*}
Thus
\[
\vert  d(a, b)- d(a,c_n)\vert = \vert d(z, f(x))- d(z, f(x_n))\vert > \epsilon.
\]
Again by the continuity of $d$ these inequalities must hold in suitable neighborhoods of $a,b,c_n$ and hence for sufficiently large $k,k'$ we have
\[
\vert  d(f^{m^n_{k'}}(z), f^{m^n_k+1}(x))- d(f^{m^n_{k'}}(z), f^{m^n_k+1}(x_n)) \vert > \epsilon
\]
that is  for $k'>k$
\[
\vert d(f^{m^n_{k'}-m^n_k-1}(z), x)- d(f^{m^n_{k'}-m^n_k-1}(z), x_n) \vert > \varepsilon
\]
Choose for every $n$ an exponent $m(n):=m^n_{k'}-m^n_k-1\in\mathbb{N}$ where $k,k'$, $k'>k$, are chosen sufficiently large so as to satisfy
\[
\vert d(f^{m(n)}(z), x)- d(f^{m(n)}(z), x_n) \vert > \varepsilon.
\]
Notice that since $d\ge 0$, for each $n$ we have $d(f^{m(n)}(z),x)> \varepsilon$ or $d(f^{m(n)}(z),x_n)> \varepsilon$. If there are infinitely many former
cases,  the compactness of $\{d\ge  \varepsilon\}$ tells us that there is a subsequence of  $\{(f^{m(n)}(z), x)\}_n$ that converges and hence a subsequence of
$\{f^{m(n)}(z)\}_n$ that converges. Actually, the same conclusion can be drawn if there are infinitely many latter cases, for the compactness of
$\{d\ge \varepsilon\}$ tells us that there is a subsequence of  $\{(f^{m(n)}(z), x_n)\}_n$ that converges and hence  a subsequence of
$\{f^{m(n)}(z)\}_n$ that
converges. In either case, denoting with  $z_\infty$ the point to which a  subsequence of $\{f^{m(n)}(z)\}_n$ converges we get, taking the limit of the previous
equation in display
\[
0=\vert d(z_\infty,x)- d(z_\infty,x) \vert \ge  \varepsilon.
\]
The contradiction proves that $f(x_n)\to f(x)$.
\end{proof}

\begin{theorem} \label{iba0}
Let $(X,d)$ be a bounded Lorentzian-metric space
and let $f\colon X\to X$ be distance preserving. Then $f$ is surjective.
\end{theorem}

\begin{proof}
Suppose not.

We claim that under the assumption there exists a nonempty open set $O\subset X$ disjoint from $f(X)$. Let $p\in  X\backslash f(X)$.
If there exists a neighborhood of $p$ disjoint from $f(X)$ there is nothing to prove. Otherwise every neighborhood $U \ni p$ intersects $f(X)$, hence we can find $r_n\in X$, such that $f(r_n)\to p$. If there is a converging subsequence $r_{n_k}\to r\in X$, then as $f$ is continuous (by Lemma \ref{azo0})  $f(r_n)\to f(r)$ which implies $p=f(r)$, hence $p\in f(X)$, a contradiction. Thus no subsequence of $r_n$ converges.
By (iii) $I^+(p)$ or $I^-(p)$ is non-empty. We assume the former possibility, the latter case being analogous. The non-empty set $I^+(p)$ does not intersect $f(X)$, for if it had some point $f(q)\in I^+(p)$, $d(p,f(q))>\epsilon>0$, $q\in X$, then  $d(f(r_n), f(q))>\epsilon$ for sufficiently large $n$, which would imply $d(r_n, q)>\epsilon$ and hence by the compactness of $\{ d\ge \epsilon\}$ there would be some converging subsequence of $r_n$, a contradiction. As $I^+(p)$ is a non-empty open set not intersecting $f(X)$ we conclude that claim is true.

Let $O\subset X$ be an open set such that $f(X)$ is disjoint from $O$ and let $p\in O$. Choose $r_i,s^i\in X$ and $a_i,b_i,c^i,e^i\in \R\cup\{\pm\infty\}$ $(1\le i\le I$) such that
\begin{equation} \label{cqz}
p\in \bigcap_{i=1}^I (d_{r_i})^{-1}((a_i,b_i)) \cap (d^{s^i})^{-1}((c_i,e_i))
\subset O
\end{equation}
 and hence $\bigcap_{i=1}^I (d_{r_i})^{-1}((a_i,b_i)) \cap (d^{s^i})^{-1}((c_i,e_i)) \cap f(X)=\emptyset$. Since for all $k\ge 1$, $f^k(p)\in f(X)$, we have
 \begin{equation} \label{cxx}
 f^k(p)\notin  \bigcap_{i=1}^I (d_{r_i})^{-1}((a_i,b_i)) \cap (d^{s^i})^{-1}((c_i,e_i)).
 \end{equation}
There is an increasing sequence $k_n$ such that $f^{k_n}(r_i) \to \tilde r_i$, $f^{k_n}(p) \to \tilde p$, $f^{k_n}(s^i) \to \tilde s^i$,  for every $i$ (the proof was given in the 4th paragraph of the proof of Lemma \ref{azo0}).  By preservation of distance the sequence $\{ d( f^{k_n}( r_i), f^{k_n}(p))\}_n$ is constant
thus  $ d(  r_i, p)=\lim_n d( f^{k_n}( r_i), f^{k_n}(p))= d( \tilde r_i, \tilde p) $, where in the last step we used the continuity of $d$. Thus
$a_i<d( \tilde r_i, \tilde p) <b_i$. This implies that for $n$ big enough $a_i< d( f^{k_n}( r_i), \tilde p) <b_i$ (and similarly for the other inequalities involving $c_i$ and $e_i$), which implies that for $m>n$ big enough $a_i< d( f^{k_n}( r_i), f^{k_m} (p)) <b_i$, hence for  $m>n$ big enough, by the preservation of distance  $a_i< d(  r_i, f^{k_m-k_n} (p)) <b_i$.
Thus considering all inequalities we have that for $m>n$ big enough
\[
f^{k_m-k_n} (p) \in \bigcap_{i=1}^I (d_{r_i})^{-1}((a_i,b_i)) \cap (d^{s^i})^{-1}((c_i,e_i)),
\]
taking into account that $k_s$ is increasing, $k_m-k_n\ge 1$, which contradicts (\ref{cxx}).
\end{proof}

\begin{theorem} \label{cis}
Let $(X,d_X)$, $(Y,d_Y)$  be bounded Lorentzian-metric spaces, and let $f:X\to Y$ be distance preserving and surjective, then $f$ is continuous (with respect to the topologies $\mathscr{T}_X$ and $\mathscr{T}_Y$).
\end{theorem}

\begin{proof}
 Since the topologies are first countable  it suffices to prove sequential continuity of $f$.
Suppose that $f$ is  not continuous at $x$. We can find a sequence $x_n\to x$ such that  $\{f(x_n)\}_n$ does not converge to $f(x)$. We know that there exists a point in the chronological past or in the chronological future of $x$. Let us assume the
former possibility, the latter being analogous, and let $x'\ll x$. For sufficiently large $n$  we have $x'\ll x_n$ and
\[
d_Y(f(x'), f(x_n))=d_X(x',x_n)>d_X(x',x)/2>0.
\]
This proves that $f(x_n)$ belongs to a compact set. Without loss of generality we can assume that $f(x_n)\to y\in Y$, $y\ne f(x)$. However there exists $z\in Y$
such that $d_Y(y,z)\ne d_Y(f(x),z)$ (or the other analogous equation). Thus we can find $\epsilon >0$ such that $\vert d_Y(y,z)- d_Y(f(x),z) \vert
>\epsilon$, hence for sufficiently large $n$ we have $ \vert d_Y(f(x_n),z)- d_Y(f(x),z) \vert >\epsilon$.
By surjectivity there is $w\in X$ such that $z= f(w)$. Thus for sufficiently large $n$ we obtain $ \vert d_X(x_n,w)- d_X(x,w) \vert >\epsilon$, which contradicts the
continuity of $d_X(\cdot, w)$.
\end{proof}

\begin{theorem} \label{cia}
Let $(X,d_X)$, $(Y,d_Y)$  be bounded Lorentzian-metric spaces, and let $f:X\to Y$, $g: Y \to X$ be distance preserving.
Then $f$ and $g$ are
 continuous isometries (with continuous inverses).
\end{theorem}

\begin{proof}
By Theorem \ref{ine} $f$ and $g$ are injective.
By  Theorem \ref{iba0} $g \circ f: X\to X$ is surjective. Thus by the injectivity of $g$, $f$ is surjective. Similarly $g$ is surjective.
By Theorem \ref{cis} they are continuous bijections. The inverses are also surjective and distance preserving hence continuous.
\end{proof}

\begin{remark} \label{xof}
The statement of the previous theorem does not change if one (or both) $X$, $Y$, contain the spacelike boundary (as a consequence, one contains the spacelike boundary if so does the other). Indeed, the theorem applies to $\mathring{X}$ and $\mathring{Y}$ and the restrictions $f\vert_{\mathring{X}}$, $g\vert_{\mathring{X}}$ (note that by distance preservation they cannot have the spacelike boundary in their image) so they are indeed maps $f\vert_{\mathring{X}}: {\mathring{X}}\to {\mathring{Y}}$ and $g\vert_{\mathring{Y}}: {\mathring{Y}}\to {\mathring{X}}$ hence bijective by the previous theorem. But now $f$ is injective because if $f(i^0_X)=f(p)$ for some $p\ne i^0_X$ then $f(p)$ is in the chronological future or past of some point necessarily in $\mathring{Y}$ and hence of the form $f(q)$. By distance preservation we would get that $i^0_X$ is in the chronological future or past of $q$, a contradiction.  Thus necessarily $f(i^0_X)=i^0_Y$ since all the other points in $\mathring{Y}$ are already in the image of $f\vert_{\mathring{X}}$. From here it is immediate that $f$ is an isometry and hence a homeomorphism (Theorem \ref{vps}).
\end{remark}

\section{Lorentzian Gromov-Hausdorff distances}

\subsection{Abstract Gromov-Hausdorff semi-distance}

\begin{definition}
Let $X$ and $Y$ be sets.
A relation $R\subset X\times Y$ is a {\it correspondence} if
\begin{itemize}
\item[(i)]  for all $x\in X$ there exists $y\in Y$ with $(x,y)\in R$, and
\item[(ii)]  for all $y\in Y$ there exists $x\in X$ with $(x,y)\in R$.
\end{itemize}
\end{definition}

\begin{remark} \label{yws}
If $R\subset X\times Y$ is a correspondence so is
$$R^T:=\{(y,x):\; (x,y)\in R\}\subset Y\times X.$$
\end{remark}

A pair $(X,d_X)$ consisting of  a nonempty set $X$ and a bounded function $d_X\colon X\times X\to [0,\infty)$ is abbreviated as {\it bounded space}.

\begin{definition}
Let  $(X,d_X)$ and $(Y,d_Y)$ be bounded spaces and $R\subset X\times Y$ be a
correspondence.
The {\it distortion} of  $R$ is defined as:
\[
\textrm{dis}\, R:=\sup\{\vert d_X(x,x')-d_Y(y,y') \vert : (x,y), (x',y')\in R\}
\]
\end{definition}

\begin{remark}\label{R_symm}
We have $\textrm{dis}\, R^T=\textrm{dis}\, R$.
\end{remark}

\begin{proposition} \label{clo}
Let $R\subset X\times Y$ be a correspondence between two bounded Lorentzian metric spaces. Then $\bar R$ is a closed correspondence with the same distortion.
\end{proposition}

\begin{proof}
For every $x\in X$ we can find $y\in Y$, such that $(x,y)\in R\subset \bar R$, and analogously, given $y\in Y$ we can find $x\in X$ such that $(x,y)\in R\subset \bar R$.

Let $(x,y), (x',y')\in \bar R$, then we can find $x_n\to x$, $y_n\to y$, $(x_n,y_n)\in R$, and   $x'_n
\to x'$, $y'_n\to y'$, $(x'_n,y'_n)\in R$. Thus for every $n$,
\[
\vert d_X(x_n, x'_n)-d_Y(y_n, y'_n)\vert \le \textrm{dis} R
\]
thus taking the limit and using the continuity of $d_X$ and $d_Y$,
\[
\vert d_X(x, x')-d_Y(y, y')\vert \le \textrm{dis} R
\]
which implies that $\textrm{dis} \bar R\le \textrm{dis} R$, the other direction being obvious.
\end{proof}

\begin{definition}
Let $(X,d_X)$ and $(Y,d_Y)$ be bounded spaces. The {\it Gromov-Hausdorff semi-distance} between $(X,d_X)$ and $(Y,d_Y)$ is defined as
\begin{equation} \label{idp}
d_{GH}(X,Y)=\inf\nolimits_R \textrm{dis}\, R
\end{equation}
where the infimum is taken over all correspondences $R\subset X\times Y$.
\end{definition}

Notice that for every bounded space $X$, $d_{GH}(X,X)=0$, as the correspondence given by the diagonal $\Delta\subset X\times X$ has vanishing distortion.

\begin{example}
It is obvious from the definition that any bounded space $(X,d_X)$ with $\sup\{d_X(x,y)\}<\epsilon$ has Gromov-Hausdorff semi-distance
of less than $\epsilon$ from the trivial space $(Y,d_Y)$ consisting of $Y=\{pt\}$ and $d_Y\equiv 0$.
\end{example}

The composition of two relations $R_1\subset X\times Y$ and $R_2\subset Y\times Z$ is defined as follows
\begin{equation} \label{kdc}
R_2\circ R_1=\{(x,z)\vert \exists y\in Y: (x,y)\in R_1 \textrm{ and } (y,z)\in R_2\}.
\end{equation}

\begin{lemma}\label{lemmatriangle}
Let $(X,d_X)$, $(Y,d_Y)$ and $(Z,d_Z)$ be bounded spaces. Further, let $R_1\subset X\times Y$, $R_2\subset Y\times Z$ be two correspondences. Then
we have
\[
\textrm{dis}\,(R_2 \circ R_1)\le \textrm{dis}\, R_2+\textrm{dis}\, R_1 .
\]
\end{lemma}

\begin{proof}
By definition we have
\begin{align*}
\textrm{dis}(R_2 \circ R_1)=\sup\{\vert d_Z(z,z')-d_X(x,x') \vert:(x,z) (x',z')\in  R_2 \circ R_1\}.
\end{align*}
Note that for every $y,y'\in Y$ and in particular for $y,y'$ such that $(x,y), (x',y')\in R_1$, $(y,z), (y',z')\in R_2$ follows
\[
\vert d_Z(z,z')-d_X(x,x') \vert\le \vert d_Z(z,z')-d_Y(y,y') \vert +\vert d_Y(y,y')-d_X(x,x') \vert
\]
and thus
\begin{align*}
\textrm{dis}\,(R_2 \circ R_1) &\le \sup\{\vert d_Z(z,z')-d_Y(y,y') \vert +\vert d_Y(y,y')-d_X(x,x') \vert:\\
&\qquad (x,y), (x',y')\in R_1,(y,z), (y',z')\in R_2 \}\\
&\le \sup\{\vert d_Z(z,z')-d_Y(y,y') \vert:(y,z) (y',z')\in  R_2 \} \\
&\quad +\sup\{\vert d_Y(y,y')-d_X(x,x') \vert:(x,y) (x',y')\in  R_1\}\\
&\le \textrm{dis}\, R_2 + \textrm{dis}\, R_1.
\end{align*}
\end{proof}

\begin{proposition} \label{qkz}
The Gromov-Hausdorff semi-distance satisfies the triangle inequality. Namely, for any $(X,d_X)$, $(Y,d_Y)$ and $(Z,d_Z)$ bounded spaces, we have
\[
d_{GH}(X,Z)\le d_{GH}(X,Y)+d_{GH}(Y,Z).
\]
\end{proposition}

\begin{proof}
The family of correspondences $R_{X,Z}$ between $X$ and $Z$ contains the family of those correspondences that are  obtained through composition of  correspondences
$R_{X,Y}$ and $R_{Y,Z}$. With Lemma \ref{lemmatriangle} follows
\begin{align*}
d_{GH}(X,Z)&=\inf_{R_{X,Z}} \textrm{dis}\, R_{X,Z}\le  \inf_{R_{Y,Z}, R_{X,Y} } \textrm{dis}\, (R_{Y,Z}\circ R_{X,Y})\\
&\le  \inf_{R_{Y,Z}, R_{X,Y} } (\textrm{dis}\, R_{X,Y}+\textrm{dis}\, R_{Y,Z})\\
&\le  \inf_{R_{X,Y}} \textrm{dis}\, R_{X,Y}+ \inf_{R_{Y,Z}} \textrm{dis}\, R_{Y,Z} =d_{GH}(X,Y)+d_{GH}(Y,Z).
\end{align*}
\end{proof}

In summary we have

\begin{proposition}
The Gromov-Hausdorff semi-distance between bounded spaces has the following properties:
\begin{itemize}
\item[(a)] $d_{GH}(X,Y)\ge 0$
\item[(b)] $d_{GH}(X,Y)=d_{GH}(Y,X)$
\item[(c)] $d_{GH}(X,Z)\le d_{GH}(X,Y)+d_{GH}(Y,Z)$
\end{itemize}
\end{proposition}

\begin{proof}
Point (a) is obvious from the definition. Point (b) follows from Remark \ref{R_symm}. Point (c) is Proposition \ref{qkz}.
\end{proof}

\begin{remark} \label{buy}
The addition of the spacelike boundary to two bounded Lorentzian-metric spaces not including it does not alter their  Gromov-Hausdorff semi-distance. The reason is that any correspondence can be enlarged including the element $(i^0_X, i^0_Y)$ without altering its distortion. Similarly, if $X$, $Y$ include the spacelike boundary then we can remove it from both
without altering their  Gromov-Hausdorff semi-distance (because any correspondence can see its distortion reduced letting $i^0_X$ correspond just to $i^0_Y$ and conversely, and so both can be really removed without altering the distortion). This does not work if the spacelike boundary is removed on just one of the two spaces.
\end{remark}

The next result clarifies the behavior of the Gromov-Hausdorff semi-distance with reference to the distance quotient (see Section \ref{quo}).

\begin{proposition} \label{vkx}
Let $(X,\mathscr{T})$ be a topological space endowed with a continuous and bounded function $d\colon X\times X\to [0,\infty)$, hence a bounded space.
Let $S$ be a  subset of $X$, endowed with the induced  distance. Then $d_{GH}(S,S/\!\sim)=0$. Moreover, there is a subset $\check S\subset S$, isometric to $S/\sim$, such that the distinguishing property of Definition \ref{D1}(iii) holds for $\check S$, and $d_{GH}(S,\check S)=0$.
\end{proposition}

By Proposition \ref{cna} if $X$ is  a bounded Lorentzian-metric space and $S$ is a finite subset then, by first taking the distance quotient and then the representatives, we can find $\check S \subset S$, $d_{GH}(\check S, S)=0$, which is a causet.

\begin{proof}
Consider the correspondence $R\subset S\times S/\!\sim$, $R=\{(s, [s]), s\in S\}$. Then for $(s_1, [s_1]), (s_2, [s_2]) \in R$,
\[
\vert \tilde d([s_1],[s_2])- d(s_1,s_2)\vert=0,
\]
by the definition of quotient distance $\tilde d$. Thus $\textrm{dis}\, R=0$ and hence
\[
d_{GH}(S,S/\sim)=0.
\]

Let us define $\check S$ by picking an element in each class $[s]\in S/\sim$. This establishes a bijection between the two sets (hence a natural correspondence). By the same calculation of the previous paragraph, $\check S$ and $S/\sim$ are isometric, thus $d_{GH}(S/\sim,\check S)=0$. The desired equality  $d_{GH}(S,\check S)=0$ follows from the triangle inequality for $d_{GH}$. Further, we know that $S/\sim$ satisfies the distinguishing property of Definition \ref{D1}(iii) so, by the isometry, the same holds for $\check S$.
\end{proof}

\subsection{Lorentzian Gromov-Hausdorff distance}

The following result  will be proved through several propositions in what follows.

\begin{theorem}\label{vmw}
The Gromov-Hausdorff semi-distance between bounded Lorentzian-metric spaces that contain $i^0$ has the following properties:
\begin{itemize}
\item[(a)] The Gromov-Hausdorff semi-distance is non-negative. Further, the
 spaces $(X,d_X)$ and $(Y,d_Y)$ are isometric and homeomorphic, that is, there exists an isometry (which then is also a homeomorphism) $f\colon X\to Y$, if and only if $d_{GH}(X,Y)=0$.
\item[(b)] $d_{GH}(X,Y)=d_{GH}(Y,X)$
\item[(c)] $d_{GH}(X,Z)\le d_{GH}(X,Y)+d_{GH}(Y,Z)$
\end{itemize}
The same result holds with ``contain" replaced by ``do not contain''.
\end{theorem}

The next result is then an immediate consequence of the previous one \cite[Prop.\ 1.1.5]{burago01}.

\begin{corollary} \label{bof}
The Gromov-Hausdorff semi-distance $d_{GH}$ descends to the set of isometry classes of bounded Lorentzian-metric spaces that contain $i^0$ (equivalently, that do not contain $i^0$). The resulting space
is a metric space.
\end{corollary}

\begin{definition}
The {\it distortion of a  map} $f: X\to Y$ is
\[
\textrm{dis} f:= \sup\{\vert d_Y(f(x),f(x'))-d_X(x,x')\vert : x,x'\in X\vert\}.
\]
\end{definition}

\begin{definition} \label{vid}
A map $f: X \to Y$ is an $\epsilon$-{\it isometry on the image} or an $\epsilon$-{\it approximation on the image} if
 $\textrm{dis} f\le \epsilon$.
\end{definition}

\begin{proposition} \label{boa}
Let $(X,d_X)$ and $(Y,d_Y)$ be bounded Lorentzian-metric spaces, and let $R$ be a correspondence. If
$d_{GH}(X,Y)<\epsilon$, then there exists an   $\epsilon$-isometry on the image  $f: X\to Y$.
\end{proposition}

\begin{proof}
Let $R \subset X\times Y$ be a correspondence such that $\textrm{dis} R<\epsilon$. For each $x\in X$ let us choose some $f(x)\in Y$ such that $(x,f(x))\in R$ so as to get a map $f: X\to Y$. We  claim  that $\textrm{dis} f <\epsilon$. Indeed for $x,x'\in X$ we have
\begin{align*}
\vert d_Y(f(x),f(x'))- d_X(x,x')\vert & \le \sup_{\substack{y,y'\in R:\\
(x,y), (x',y')\in R}}  \vert d_Y(y,y')-d_X(x,x') \vert  \\
& \le \textrm{dis} R\;  \le d_{GH}(X,Y)  <\epsilon.
\end{align*}
\end{proof}

\begin{proposition} \label{inf}
Let $(X,d_X)$ and $(Y,d_Y)$ be bounded Lorentzian-metric spaces that either both contain the spacelike boundary or that do not. If $d_{GH}(X,Y)=0$ then the spaces $(X,d_X)$ and $(Y,d_Y)$ are isometric (and the isometry is a homeomorphism).
\end{proposition}

The fact that the isometry is a homeomorphism follows from Theorem \ref{vps} (see also Theorem \ref{cia} and Remark \ref{xof}). Observe that removing $i^0$ from a bounded Lorentzian-metric space gives a bounded Lorentzian-metric space which, however, is not in general  at zero Gromov-Hausdorff distance from the original space.

\begin{proof}
We give the proof for $X,Y$ not containing the spacelike boundary.  The other case follows from Remark \ref{buy}.
By Proposition \ref{boa} we know that there are maps $f_n: X\to Y$ which are $\epsilon_n$-isometries on the image for any sequence $\epsilon_n\to 0$. We observe that for every $x\in X$ we can find a subsequence $\{f_{n_s}(x)\}_s$
such that $f_{n_s}(x)$ converges in $Y$. Indeed, we can find $z\ll x$ (or the other inequality), hence for $n$  sufficiently large such
that $\epsilon_n<d(z,x)/2$ we have
\[
d(f_{n}(z),f_{n}(x)) \ge d(z,x)/2>0
\]
which proves that the sequence $f_{n}(x)$ is contained in a compact set for $n$ large enough.

Let $S=\{x_k\}$ be a countable dense subset of $X$.
For every $k$ we can find a subsequence $\{f_{n^k_m}(x_k)\}_m$ that converges to some point of $Y$ which we denote $f(x_k)$. Notice that
the sequence $\{n^{k+1}_m\}_m$ can be chosen to be a subsequence of $\{n^k_m\}_m$. Then by the Cantor diagonal trick,
$\{f_{n^m_m}(x)\}_m$ is a subsequence that converges to $f(x)$ for every $x\in S$. Thus, without loss of generality, we can assume that $f_n\vert_S$ converges to some function $f:S\to Y$. Since $f_n$ is an $\epsilon_n$-isometry on its image we have for $x,x'\in S$
\[
0\le \vert d(f(x),f(x'))-d(x,x')\vert=\lim_n \vert d(f_n(x),f_n(x'))-d(x,x')\vert\le \lim \epsilon_n=0,
\]
thus $f$ preserves $d$. Now we extend $f$ to $X$ as follows. If $x\in X$ choose a sequence $x_n\to x$, $x_n\in S$, such that $\lim_n f(x_n)$  exists  (this is possible by the usual compactness argument. Namely take $z\ll x$, or the other inequality, and use that $d(f(z),f(x_n))=d(z,x_n)  \ge d(z,x)/2 >0$ for large $n$) and set $f(x):=\lim_n f(x_n)$. Let $x,x'\in X$, and let $x_n, x_n'$ the just mentioned sequences defining $f(x)$ and $f(x')$. We have
\[
\vert d(f(x),f(x'))-d(x,x')\vert= \lim_n \vert d(f(x_n),f(x'_n))-d(x_n,x'_n)\vert=0,
\]
where in the last step we used the fact that $f$ preserves $d$ on $S$. We conclude that $f$ preserves $d$ on $X$.
We can construct an analogous   distance preserving map $g:  Y \to X$.
By Theorem \ref{cia} $f$ and $g$ are isometries.
\end{proof}

We are ready to prove Theorem \ref{vmw}
\begin{proof}
The non-trivial direction of (a) is Proposition \ref{inf}. For the trivial direction, the correspondence $R=\{(x,f(x)): x\in X\}$ where $f$ is the isometry, has vanishing  distortion and so $d_{GH}(X,Y)=0$. Point (b)
 follows from Remark \ref{R_symm}. Point (c) is Proposition \ref{qkz}.
 \end{proof}

\subsection{The distinction metric}

Let $(X,d)$ be a bounded Lorentzian-metric space. We assume that
$i^0\in X$, i.e. $(X,\mathscr{T})$ is compact by Corollary \ref{vjw}.

We define the following {\it distinction metric} on $X$
\[
\gamma(x,y)=\max\left(\sup_{z\in X} \vert d(x,z)-d(y,z) \vert, \ \sup_{z\in X} \vert d(z,x)-d(z,y) \vert  \right).
\]
which measures how much a point $x\in X$ can be distinguished from a point $y\in X$ by using the function $d$. The property  that $\gamma(x,y)=0$ implies $x=y$ follows from Definition \ref{D1}(iii). The triangle inequality is  straightforward. so it is a metric is the classical sense. Note that for each $z\in X$ the functions $d_z:=d(z,\cdot)$ and $d^z=d(\cdot,z)$, are
$1$-Lipschitz continuous with respect to $\gamma$.

The {\em Noldus  metric} is defined as
\[
D(x,y)=\sup_{z\in X} \vert d(z,x)+d(x,z)-d(z,y)-d(y,z)\vert
\]
 It was originally introduced in \cite[Definition 3]{noldus04} for smooth manifolds and called {\em strong metric}. It was also very much related to an earlier metric construction by Meyer \cite{meyer86}. For an observation related to the next result see also \cite[Remark on p.\ 852]{noldus04b}.
\begin{proposition}
The distinction metric coincides with  Noldus' strong metric.
\end{proposition}

\begin{proof}
We consider the case $x\ne y$ otherwise the identity $D(x,y)=0=\gamma(x,y)$ is clear.
Since $d$ is continuous and $X$ is compact, there is some $w\in X$ such that
\[
D(x,y)=d(w,x)+d(x,w)-d(w,y)-d(y,w)
\]
or
\[
D(x,y)=d(w,y)+d(y,w)-d(w,x)-d(x,w).
\]
Let us assume the former case, the latter being analogous (one is obtained from the other exchanging $x\leftrightarrow y$). Only one among $d(w,x)$ and $d(x,w)$ can be positive, and similarly only one among $d(w,y)$ and $d(y,w)$ can be positive.
 If $d(w,x)>0$ and $d(y,w)>0$ we have $y\ll w\ll x$ and therefore
\[
D(x,y)=d(w,x)-d(y,w)<d(y,x)-d(y,y),
\]
contradicting the choice of $w$. Similarly, the case $d(x,w)>0$ and $d(w,y)>0$ leads to a contradiction.

As a consequence  we have
\[
D(x,y)=d(x,w)-d(y,w)\text{ or }D(x,y)=d(w,x)-d(w,y),
\]
i.e.
\begin{align*}
D(x,y)& \le \sup_z \max\{d(z,x)-d(z,y), d(x,z)-d(y,z)\}\\
&\le \sup_x \max\{\vert d(z,x)-d(z,y)\vert, \vert d(x,z)-d(y,z)\vert\}=\gamma(x,y)
\end{align*}
The same conclusion is reached in the latter case (observe that the last expression is invariant under exchanges $x\leftrightarrow y$).

For the converse, since $d$ is continuous and $X$ is compact, there is some $w\in X$ such that (a) $\gamma(x,y)=d(x,w)-d(y,w)$, or (b) $\gamma(x,y)=d(y,w)-d(x,w)$, or (c) $\gamma(x,y)=d(w,x)-d(w,y)$, or (d) $\gamma(x,y)=d(w,y)-d(w,x)$. It is sufficient to consider case (a), the other cases following by time duality  ($d(u,v)\to d(v,u)$)  or by the symmetry $x\leftrightarrow y$.

We know that $\gamma(x,y)>0$ thus $d(x,w)>0$ and hence $d(w,x)=0$ by chronology.
If $d(y,w)=0$ we must have $d(w,y)=0$ for if $d(w,y)>0$ then $x \ll w\ll y$, and then the choice $z=y$ would contradict the fact that $w$ gives the maximal value.  If $d(y,w)>0$ then again $d(w,y)=0$ by chronology.

This shows that in case (a)
\begin{align*}
\gamma(x,y)&=d(w,x)+d(x,w)-d(w,y)-d(y,w)
 \\
&=\vert d(w,x)+d(x,w)-d(w,y)-d(y,w)\vert \le D(x,y).
\end{align*}
Since the penultimate expression is invariant by time duality ($d(u,v)\to d(v,u)$) and by exchanges $x\leftrightarrow y$, we conclude that it also holds in cases (b), (c) and (d).
\end{proof}

\begin{proposition} \label{nwq}
The Lorentzian distance $d$ is 1-Lipschitz with respect to $\gamma$.
\end{proposition}

\begin{proof}

\begin{align*}
\vert d(u,v)-d(x,y) \vert\le  \vert d(u,v)-d(x,v)\vert +\vert d(x,v)-d(x,y) \vert \le \gamma(u,x)+\gamma(v,y).
\end{align*}
\end{proof}

The topology induced by $\gamma$ is called {\em $\gamma$-topology}. For Lorentzian manifolds the next result is stated in \cite[Theorem\ 8 (c)]{noldus04} without proof.
\begin{proposition}
Let $(X,d)$ be a bounded Lorentzian-metric space. The $\gamma$-topology coincides with the topology $\mathscr{T}$ of $(X,d)$. Thus $(X,\gamma)$ is a compact metric space.
\end{proposition}

As a consequence, $\gamma$ is $\mathscr{T}$-continuous, as every metric is continuous in the topology it induces.

\begin{proof}
By Prop.\ \ref{nwq} $d$ is continuous in the $\gamma$-topology and so are the functions $d_p, d^p$ for $p\in X$. By Cor.\ \ref{vjw} the $\gamma$-topology is finer than $\mathscr{T}$.

For the other direction, suppose by contradiction that there is  a $\gamma$-ball $\{z: \gamma(x,z)<r\}$, $x\in X$, $r>0$, such that no $\mathscr{T}$-open neighborhood of $x$ is contained in it. As $\mathscr{T}$ is first-countable we can find $x_n \to x$ (convergence in the $\mathscr{T}$ topology) such that $\gamma(x,x_n)\ge \epsilon$. By the $\mathscr{T}$-compactness of $X$
there is $w_n$ such that $\gamma(x,x_n)=\vert d(w_n, x)-d(w_n, x_n)\vert$ or  $\gamma(x,x_n)=\vert d(x, w_n)-d( x_n, w_n) \vert$. We can pass to a subsequence so that only one of the equalities holds, say the former, and $w_n\to w$, for some $w\in X$ (convergence in the $\mathscr{T}$ topology). By taking the limit,  $0<\epsilon\le  \vert d(w, x)-d(w, x)\vert = 0$, a contradiction.
\end{proof}

\begin{remark} \label{voz}
If $Y \subset X$ then $\gamma_Y\le \gamma_X\vert_{Y\times Y}$.
\end{remark}

\begin{proposition} \label{eqd}
Let $(X,d)$ be a bounded Lorentzian metric space. The diameter $\textrm{diam}^\gamma X$ for the distinction metric coincides with the diameter $\textrm{diam} X$ for the Lorentzian metric.
\end{proposition}

Elsewhere the Lorentzian diameter of a bounded Lorentzian metric space $X$ is also denoted $\textrm{diam} X$. The below proof is independent of whether $i^0\in X$ or not.

\begin{proof}
By Remark \ref{R1} there are $x,y\in X$ such that $d(x,y)=\textrm{diam} X$. Now, by the expression for $\gamma(x,y)$ we have for every $z\in X$
\[
\gamma(x,y)\ge \max\{\vert d(x,z)-d(y,z)\vert, \vert d(z,x)-d(z,y)\vert \} ,
\]
thus choosing  $z=y$ we get $\gamma(x,y) \ge d(x,y)-d(y,y)=d(x,y)=\textrm{diam} X$
which proves $\textrm{diam}^\gamma X\ge \textrm{diam} X$.

For the other inequality, observe that for every $x,y,z\in X$, $\vert   d(x,z)-d(y,z)\vert \le \textrm{diam} X$ and $\vert d(z,x)-d(z,x)\vert\le \textrm{diam} X$, thus $\gamma(x,y) \le \textrm{diam} X$ which implies $\textrm{diam}^\gamma X\le \textrm{diam} X$.
\end{proof}

\subsection{Non-rectifiability via the distinction metric}

The purpose of this section is to provide some  explicit examples of the distinction metric. We shall also confirm that, already in Lorentzian manifolds, it cannot be used to parametrize causal curves as they are non-$\gamma$-rectifiable (for the result framed in the context of the Noldus metric, see \cite[Thm.\ 6]{noldus04b}  and \cite[Thm.\ 3]{muller19}).

Note that by the definition of bounded Lorentzian metric space there exists for all $x,y\in X$ a point $z_{x,y}\in X$ with
\[
\gamma(x,y)=|d(z_{x,y},x)-d(z_{x,y},y)|\text{ or }|d(x,z_{x,y})-d(y,z_{x,y})|.
\]

For Lorentzian metric spaces induced from smooth spacetimes (and more generally Lorentzian length spaces, see Section \ref{sec_lor_length_spaces}) the following observation from
\cite[Section 4]{bombelli04} provides a first idea where one has to look for the point $z_{x,y}$ for given $x,y\in X$.

\begin{proposition}[\cite{bombelli04}]\label{prop_gamma}
Let $(X,d)$ be a bounded Lorentzian metric space induced by a causally convex subset of a smooth globally hyperbolic spacetime. Let $x,y\in X$ be given. If
\[
\gamma(x,y)=|d(z_{x,y},x)-d(z_{x,y},y)|
\]
we have $z_{x,y}\in I^-(x)\triangle I^-(y)$ and $I^-(z_{x,y})\subset I^-(x)\cap I^-(y)$. If
\[
\gamma(x,y)=|d(x,z_{x,y})-d(y,z_{x,y})|
\]
we have $z_{x,y}\in I^+(x)\triangle I^+(y)$ and $I^-(z_{x,y})\subset I^+(x)\cap I^+(y)$.
\end{proposition}
Here $\triangle$ denotes the symmetric difference.
In order to provide some intuition about the behavior of this metric we discuss a simple example derived from a causal diamond in $1+1$-dimensional Minkowski space.

\begin{example}
Consider the subset $\tilde{X}:=[0,1]\times [0,1]\subset \mathbb{R}^2$ with the Lorentzian metric $\eta((v_1,v_2),(w_1,w_2)):=\frac{1}{2}(v_1 w_2+ v_2 w_1)$. The induced Lorentzian
distance on $\tilde{X}$ is ($\tilde{x}=(x_1,x_2), \tilde{y}=(y_1,y_2)$):
\[
\tilde{d}(\tilde{x},\tilde{y})=\begin{cases}
\sqrt{(y_1-x_1)(y_2-x_2)},&\text{ if }x_1\le y_1,\; x_2\le y_2 \\
0,&\text{ otherwise}
\end{cases}
\]
Identifying $(1,0)$ with $(0,1)$, and denoting it $i^0$, yields a bounded Lorentzian metric space $(X,d)$ with the topology on $ X\backslash \{i^0\}$ being that induced from  $\mathbb{R}^2$.

For points $x=(x_1,x_2)$ and $y=(y_1,y_2)$ in $X$ define
\begin{align*}
z^1:=(\min\{x_1,y_1\},0),\;& z^2:=(1,\max\{x_2,y_2\}),\\
z^3:=(\max\{x_1,y_1\},1),\;& z^4:=(0,\min\{x_2,y_2\})
\end{align*}

We claim that the distinction distance of $x$ and $y$ is given by
\[
\gamma(x,y)=\max_{1\le i\le 4} \{d(x,z^i),d(y,z^i), d(z^i,x), d(z^i,y)\}.
\]
A similar result for strips in Minkowski space has been obtained in \cite{muller19}.

According to Proposition \ref{prop_gamma} for $z\in X$ with
\[
\gamma(x,y)=|d(x,z)-d(y,z)|
\]
we have $z\in I^+(x)\triangle I^+(y)$ and in case
\[
\gamma(x,y)=|d(z,x)-d(z,y)|
\]
we have $z\in I^-(x)\triangle I^-(y)$.

Since $\gamma$ is symmetric we can assume $x_1\le y_1$. First assume that $x_2\le y_2$, i.e.\ $y\in J^+(x)$. Further we have $z^1=(x_1,0), z^2=(1,y_2), z^3=(y_1,1)$, and
$z^4=(0,x_2)$. For $z=(z_1,z_2)\in I^-(x)\triangle I^-(y)=I^-(y)\setminus I^-(x)$ we have
\begin{align*}
d(z,y)=\sqrt{(y_1-z_1)(y_2-z_2)}& \le \max\{\sqrt{y_1(y_2-z_2)},\sqrt{(y_1-z_1)y_2}\}\\
& \le  \max\{\sqrt{y_1(y_2-x_2)},\sqrt{(y_1-x_1)y_2}\} \\
& =  \max\{d(z^1,y),d(z^4,y)\}
\end{align*}
Analogously we get
\[
d(x,z')\le \max\{d(x,z^2),d(x,z^3)\}
\]
for $z'\in I^+(x)\triangle I^+(y)=I^+(y)\setminus I^+(x)$, i.e.
\begin{align*}
\gamma(x,y)& =\max_{1\le i\le 4}\{d(z^i,y),d(x,z^i)\}\\
&= \max\{\sqrt{y_1(y_2-x_2)},\sqrt{(y_1-x_1)y_2}, \sqrt{(1-x_1)(y_2-x_2)},\sqrt{(y_1-x_1)(1-x_2)}
\}
\end{align*}
In the other case $x_2>y_2$ we have $z^1=(x_1,0), z^2=(1,x_2), z^3=(y_1,1)$, and $z^4=(0,y_2)$. It follows for $z\in I^-(x)\triangle I^-(y)$
\begin{align*}
d(z,y)=\sqrt{(y_1-z_1)(y_2-z_2)}& \le \max\{\sqrt{(y_1-z_1)y_2},\sqrt{y_1(y_2-z_2)}\} \\
& \le \sqrt{(y_1-x_1)y_2} \\
& = d(z^1,y)
\end{align*}
and
\begin{align*}
d(z,x)=\sqrt{(x_1-z_1)(x_2-z_2)}& \le \max\{\sqrt{(x_1-z_1)x_2},\sqrt{x_1(x_2-z_2)}\} \\
& \le \sqrt{x_1(x_2-y_2)} \\
& = d(z^4,y).
\end{align*}
Analogously we obtain for $z'\in I^+(x)\triangle I^+(y)$
\[
d(x,z')\le d(x,z^3)\text{ and }d(y,z')\le d(y,z^2).
\]
\end{example}

It worth noting that the distinction metric in the example behaves like a square root of the Euclidean distance for the distance approaching zero. This non-Lipschitz behaviour
is not particular to this example as the next proposition shows.

\begin{proposition}\label{prop_estim_dist}
Let $(X,d)$ be a bounded Lorentzian metric space induced by a precompact and causally convex set of a smooth globally hyperbolic spacetime $(M,g)$ (see Sec.\ \ref{gob}) . Then for every Riemannian metric $h$ on $M$
with distance function $\dist$ and every $\varepsilon >0$ there exists $C_{h,\varepsilon}>0$ such that
\[
\gamma(x,y)\ge C_{h,\varepsilon}(\dist(x,y))^{\frac{3}{4}}
\]
for all $x,y\in X$ with $\dist(x,\partial X),\dist(y,\partial X)>\varepsilon$.
\end{proposition}

We have the immediate corollary.

\begin{corollary}\label{cor_non-rectif}
Let $(X,d)$ be a bounded Lorentzian-metric space induced by a precompact and causally convex subset of a smooth globally hyperbolic spacetime. Then every curve in the
interior of $X$ that is rectifiable w.r.t.\ to the distinction metric of $(X,d)$ is constant.
\end{corollary}

\begin{proof}[Proof of Corollary \ref{cor_non-rectif}]
Let $h$ be a Riemannian metric on $M$. We proceed by proving the contraposition, i.e.\ a non-constant continuous curve in the interior of $X$ is not rectifiable w.r.t.\ $\gamma$. Let $\eta\colon I\to X\setminus \partial X$ be
continuous and non-constant. We can assume that $I$ is a compact interval $[a,b]$ as rectifiability is passed on to subarcs. Let $\varepsilon>0$ be a lower bound for the distance
of $\eta$ to $\partial X$.

Let $R>0$ be given. For a partition $a=t_0<t_1<\ldots t_{n-1}<t_n=b$ with $\dist(\eta(t_k),\eta(t_{k+1}))\le R^{-4}$ we obtain
with Proposition \ref{prop_estim_dist}:
\begin{align*}
L^\gamma(\eta)\ge \sum_{k=0}^{n-1} \gamma(\eta(t_{k+1}),\eta(t_k))&\ge C_{h,\varepsilon}\sum_{k=0}^{n-1} (\dist(\eta(t_{k+1}),\eta(t_k)))^{\frac{3}{4}} \\
& \ge C_{h,\varepsilon}R\sum_{k=0}^{n-1} \dist(\eta(t_{k+1}),\eta(t_k))
\end{align*}
The supremum of the last sum taken over all possible partitions of $[a,b]$ is the length of $\eta$ relative to $h$, which is positive or infinity since $\eta$ is non-constant. Since $R$ can
be chosen arbitrarily large the supremum for the first term is always infinity. This shows that $\eta$ is not rectifiable w.r.t.\ the distinction metric $\gamma$.
\end{proof}

Before giving the formal proof of Proposition \ref{prop_estim_dist} we explain the argument for the case of Minkowski space $(M,g)=(\mathbb{R}^{m},\langle.,.\rangle_1)$. Here the
construction is purely of affine geometric nature. For curved spacetimes one has to account for changes of the Lorentzian metric in coordinates. The optimal choice here will be normal
coordinates since their deviation from the a constant inner product is minimal.

Let $X\subset \mathbb{R}^m$ be a precompact and causally convex subset. Fix the canonical euclidean inner product $\langle.,.\rangle_0$ on $\mathbb{R}^m$.
Assume that $x,y\in X$ are sufficiently close such that the ball of radius $\sqrt{|x-y|}$ centered in $y$ in contained in $X$. Choose a plane $E\subset \mathbb{R}^m$ with $e_0,x-y\in E$ and a base $\{e_0,u\}$
of $E$ orthonormal w.r.t. both $\langle.,.\rangle_1$ and the $\langle.,.\rangle_0$. We discuss only the case $x-y=ae_0+bu$ with $a\ge 0$ and $b\le 0$. The other cases being analogous
up to replacing $e_0$ and/or $u$ with $-e_0$ and/or $-u$. For $r:=|x-y|$ we have $x\in J^+(y-ru)$ and
\[
\left\langle \sqrt{r} e_0+\left(\sqrt{r}+\frac{r}{2}\right)u, \sqrt{r} e_0+\left(\sqrt{r}+\frac{r}{2}\right)u\right\rangle_1=r^{\frac{3}{4}}+\frac{r^2}{4}>0.
\]
From this we conclude $z:=y+\sqrt{r}e_0+\left(\sqrt{r}-\frac{r}{2}\right)u\notin J^+(y-ru)$ and therefore $z\notin J^+(x)$.
On the other hand we have
\[
\left\langle \sqrt{r} e_0+\left(\sqrt{r}-\frac{r}{2}\right)u, \sqrt{r} e_0+\left(\sqrt{r}-\frac{r}{2}\right)u\right\rangle_1=-r^{\frac{3}{4}}+\frac{r^2}{4}<0,
\]
i.e. $z\in J^+(y)$. It follows that
\[
\gamma(x,y)\ge  |z-y|_1-|z-x|_1=\sqrt{r^{\frac{3}{4}}-\frac{r^2}{4}}\ge \frac{1}{2} r^{\frac{3}{4}}
\]
for $r>0$ sufficiently small. This establishes the proposition for bounded Lorentzian metric spaces which appear as subsets of Minkowski space.

\begin{proof}[Proof of Proposition \ref{prop_estim_dist}]
Choose a timelike $g$-unit vector field $T\in \Gamma(TM)$. Let a Riemannian metric $h$ on $M$ and $\varepsilon>0$ be given. Using the local equivalence of
Riemannian metrics and the precompactness of $X$ in $M$ we can assume that $h$ is given by a Wick rotation of $g$ about $T$, i.e.\ $h=g+2T^*\otimes T^*$, where $T^*$
denotes $g$-dual of $T$. Choose a finite covering
 $\{U_i\}$ of $X\setminus \{i^0\}\subset M$ by $g$-convex neighbourhoods and a Lebesgue number $\delta>0$ relative to
$\dist$ for this covering. It suffices to show the estimate for points with distance less than $\delta^2$ by decreasing the constant $C_{h,\varepsilon}$.

Let $x,y\in X$ with $\dist(x,y)<\delta^2$ and $\dist(x,\partial X),\dist(y,\partial X)>\varepsilon$ be given. Let $\exp$ denote the $g$-exponential map based at $y$. In $g$-normal
coordinates around $y$ we have
\begin{equation}\label{eq_normal_coord}
(\exp^* g_z)_v= g_y+ \frac{1}{3}R(.,v,v,.) + O(|v|^3),
\end{equation}
where $v=\exp^{-1}(z)$ and $R$ denotes the curvature tensor of $g$ at $y$. From
\eqref{eq_normal_coord} it follows that there exists a constant $C<\infty$, independent of $x,y$, such that for all $\rho\in (0,\delta)$ the Lorentzian metric
\[
G_\rho:=g_y-C\rho^2\;T_y^*\otimes T_y^*
\]
has wider lightcones than $\exp^* g$ on $\exp^{-1}(B_\rho(y))$.

For convenience set $w:=\exp^{-1}(x)$. Choose a $2$-dimensional linear subspace $E$ in $TM_y$ containing $T_y$ and $w$.
Next choose a vector $U_y\in E$ such that $\{T_y,U_y\}$ form a orthonormal base of $E$. Note that this is true for
both $g$ and $h$ by our previous assumption that $h$ is given by a Wick rotation about $T$. For the coefficients $a,b\in \mathbb{R}$ of $w$ relative to the base $\{T_y,U_y\}$
we can assume w.l.o.g.\ that $a\ge 0$ and $b\le 0$.
The other cases follow by obvious modifications of the subsequent argument. Thus for $r:=|w|$ we know that $w$ lies in the $G_{\sqrt{r}}$-future of $-r U_y$ ($T_y$ serving
as time-orientation here). We conclude
\begin{align*}
G_{\sqrt{r}}&(\sqrt{r}T_y+(\sqrt{r}+r/2)U_y,\sqrt{r}T_y+(\sqrt{r}+r/2)U_y)\\
&=-(1+Cr)r+(\sqrt{r}+r/2)^2= 2\sqrt{r}r+\left(\frac{1}{4}-C\right)r^2>0
\end{align*}
for $r$ sufficiently small, i.e. $-rU_y$ and $\sqrt{r}T_y+(\sqrt{r}-\frac{r}{2})U_y$ are not causally related relative to $G_{\sqrt{r}}$. Thus $w$ and $\sqrt{r}T_y+(\sqrt{r}-\frac{r}{2})U_y$
are not causally related relative to $G_{\sqrt{r}}$. Since $g_y$ has smaller lightcones than $G_{\sqrt{r}}$ and
\[
g_y(\sqrt{r}T_y+(\sqrt{r}-\frac{r}{2})U_y,\sqrt{r}T_y+(\sqrt{r}-\frac{r}{2})U_y)=-\sqrt{r}r+\frac{r^2}{4}<0
\]
we obtain for $z:= \exp_y(\sqrt{r}T_y+(\sqrt{r}-\frac{r}{2})U_y)$:
\[
\gamma(x,y)\ge |d(y,z)-d(x,z)|=\sqrt{\sqrt{r}r-\frac{r^2}{4}}\ge \frac{1}{2} |\exp^{-1}(x)|^{\frac{3}{4}}.
\]
Using the precompactness of $X$ in $M$ and the resulting bi-Lipschitzness of the exponential map near the zero section we conclude the claim.
\end{proof}

\subsection{Kuratowski-type embeddings and $\epsilon$-nets} \label{kur}

Let $\mathcal{B}:=\ell^\infty$, namely  the Banach space of bounded real sequences
\[
\mathcal{B}:=\{f\colon \mathbb{N}\to \mathbb{R}:\; \Vert f\Vert_\infty:=\sup\nolimits_m |f(m)|<\infty\}.
\]
The associated metric is
\[
\dist\nolimits_\infty(f,g)=\Vert f-g\Vert_\infty=\sup_m |f(m)-g(m)|.
\]

The product $\mathcal{B}\times \mathcal{B}$ can be regarded as a Banach space with norm $ \Vert (b_1,b_2) \Vert_\infty=\max(\Vert b_1\Vert_\infty,\Vert b_2\Vert_\infty)$. Of course, it is also a metric space once given the associated metric
\[
\dist\nolimits_\infty((f,f'),(g,g')):=\max\{\dist\nolimits_\infty (f,g),\dist\nolimits_\infty (f',g')\}.
\]
The completeness of this metric space, which is inherited from that of the Banach space $\mathcal{B}$,  will be particularly important for us.

Let $X$ be a bounded Lorentzian metric space containing $i^0$. Let $\mathcal{S} \subset X$ be a countable $\mathscr{T}$-dense set and $\mathscr{S}\colon \mathbb{N}\to \mathcal{S}$,
$m\mapsto s_m$ be a surjective map. For $x\in X$ define the maps $e_x,e^x\in \mathcal{B}$ as the concatenations $e_x:=d_x\circ \mathscr{S}$ and $e^x:=d^x\circ \mathscr{S}$.

\begin{theorem}\label{Theo_kurat}
The map
\[
I\colon X\to \mathcal{B}\times \mathcal{B},\quad x\mapsto E_x:=(e_x,e^x)
\]
is an embedding, i.e.\ $I$ is a homeomorphism of $X$ with its (compact, hence closed) image $I(X)$.

Moreover, $I$ is an isometry between $(X,\gamma)$ and  $(I(X), \dist\nolimits_\infty)$.
\end{theorem}

In the following we will call $I$ the {\it Kuratowski embedding} of $(X,d)$ with respect to $\mathscr{S}$. Observe that $\gamma$ does not depend on $\mathscr{S}$.

\begin{proof}
We know that $X$ is compact while $(\mathcal{B}\times \mathcal{B}, \dist\nolimits_\infty)$, being a metric space, is Hausdorff. It is well known that a continuous injective map from a compact space $X$ to a Hausdorff space is actually a  homeomorphism of $X$ with its image \cite[Theorem 17.14]{willard70} (the image is necessarily closed as the continuous image of a compact set is compact, hence closed). Thus it is sufficient to prove that $I$ is injective and continuous.

If $x,y\in X$ are mapped to the same element of $\mathcal{B}\times \mathcal{B}$ then for every $s\in \mathcal{S}$, $d(x,s)=d(y,s)$ and $d(s,x)=d(s,y)$ which implies $x=y$ by the defining
property (iii) of bounded Lorentzian-metric space and Proposition \ref{prop(iv)}. Thus $I$ is injective.

Let us prove continuity of $I$. Let $y_n\in X$ converge to $y\in X$. We will show that $d_{y_n}$ converges uniformly to $d_y$. The other assertion that $d^{y_n}$
converges uniformly to $d^y$ follows analogously. The claim then follows readily from the construction of $I$.

Assume to the contrary that $d_{y_n}$ does not converge uniformly to $d_y$. Then, passing to a subsequence if necessary, we can assume that there exists $\epsilon>0$ such that for every $n\in\mathbb{N}$ there is $r_n\in X$
with $|d_{y_n}(r_n)-d_y(r_n)|\ge \varepsilon$. We thus have
\begin{equation} \label{jbu}
|d(y_n,r_n)-d(y,r_n)|\ge \varepsilon
\end{equation}
for all $n$. Since $(X,\mathscr{T})$ is compact we can choose a converging subsequence $\{r_{n_k}\}_k$ converging to $r\in X$.
Then $d(y_{n_k},r_{n_k})$ and $d(y,r_{n_k})$ both converge to $d(y,r)$, which contradicts \eqref{jbu}.

We know that $I$ is a bijection between $X$ and $I(X)$. Moreover, we have from the definitions that
\[
(I^*\dist\nolimits_\infty)(x,y)=\max\left(\sup_{z\in \mathcal{S}} \vert d(x,z)-d(y,z) \vert, \sup_{z\in \mathcal{S}} \vert d(z,x)-d(z,y) \vert  \right)
\]
Since $\mathcal{S}$ is dense and $d$ is continuous we have that the right-hand side coincides with $\gamma(x,y)$, namely $I$ is an isometry.
\end{proof}

\begin{definition}
We say that a sequence of bounded Lorentzian-metric spaces $\{(X_m,d_{X_m})\}_{m\in\mathbb{N}}$ {\it converges to} a bounded
Lorentzian-metric space $(X,d_X)$ if
\[
d_{GH}(X_m,X)\to 0 \quad \textrm{for} \ m\to\infty.
\]
\end{definition}

\begin{definition}
Let $(X,d)$ be a bounded Lorentzian-metric space that includes $i^0$.
A set $\mathcal{N}\subseteq X$ is an {\it $\varepsilon$-net} if for every $x\in X$ there exists some $\eta\in \mathcal{N}$
with $\gamma(x,\eta)\le \varepsilon$
\end{definition}

\begin{remark}\label{remark2}$\empty$
\begin{enumerate}
\item The set $\mathcal{N}\subseteq X$ is an $\varepsilon$-net iff it is so in the sense of
\cite[Definition 1.6.1]{burago01} for the metric space $(X,\gamma)$.
\item  \label{bgf}
If  $\mathcal{N}\subseteq X$ is an $\varepsilon$-net and $Y\subseteq X$ is a bounded Lorentzian metric space such that $\mathcal{N}\subseteq Y $, then $\mathcal{N}$ is an $\varepsilon$-net of $Y$. This follows from Remark \ref{voz}.
\item For every $\varepsilon>0$ every bounded Lorentzian-metric space contains a finite $\varepsilon$-net. This follows directly
from the compactness of $X$. Just cover $X$ with $\gamma$-open balls of radius $\epsilon$ and pass to a finite covering $\{B_i\}$.
The set $\mathcal{N}$ consisting of the centers of the balls provides a finite $\varepsilon$-net.
\item  \label{bjd}
For an $\varepsilon$-net $\mathcal{N}\subseteq X$ we have $d_{GH}(\mathcal{N},X)\le 2\varepsilon$. This can be seen as follows:
For every $x\in X$ choose $\eta_x\in \mathcal{N}$ with $\gamma(x,\eta_x)\le \epsilon$
with the additional convention $\eta_\nu=\nu$ for $\nu\in\mathcal{N}$. Define the correspondence
$R\subset X\times \mathcal{N}$ as
\[
R:=\{(x,\eta_x):\; x\in X\}.
\]

For all $x,y\in X$ we have
\begin{align*}
|d(x,y)-d(\eta_x,\eta_y)|&\le |d(x,y)-d(\eta_x,y)|+|d(\eta_x,y)-d(\eta_x,\eta_y)|\\
&\le \gamma(x,\eta_x)+\gamma(y,\eta_y)  \le 2\epsilon.
\end{align*}

It follows that the distortion of $R$ is bounded by $2\epsilon$.

\item Every $\varepsilon$-net contains a finite $2\varepsilon$-net. This follows from the compactness of $(X,\mathscr{T})$.
\item \label{bkd}
 Every finite $\varepsilon$-net $\mathcal{N}\subseteq X$ contains a $3\varepsilon$-net $\check{\mathcal{N}}$ which is additionally a causet and such that $d_{GH}(\mathcal{N},\check{\mathcal{N}})=0$.
We are going to construct $\check{\mathcal{N}}$ as in Proposition \ref{vkx}, via the quotient distance and then by choice of representatives, thus the  fact that $\check{\mathcal{N}}$ is a causet and $d_{GH}(\mathcal{N},\check{\mathcal{N}})=0$ follows from that proposition.

To see that $\check{\mathcal{N}}$ is a $3\varepsilon$-net consider $\mu,\nu\in \check{\mathcal{N}}$ with $[\mu]=[\nu]\in \mathcal{N}/\sim$, i.e.
$(d_\mu,d^\mu)|_{\mathcal{N}}\equiv (d_\nu,d^\nu)|_{\mathcal{N}}$.
For every $y\in X$ we can find $\xi(y)\in \mathcal{N}$ such that $\gamma(y,\xi(y))\le \epsilon$, thus
\begin{align*}
|d(\mu,y)-d(\nu,y)|&\le |d(\mu,y)-d(\mu,\xi)|+|d(\nu,\xi)-d(\nu,y)|\\
&\le 2\gamma(y,\xi)\le 2\epsilon
\end{align*}
and
\begin{align*}
|d(y,\mu)-d(y,\nu)|&\le |d(y,\mu)-d(\xi,\mu)|+|d(\xi,\nu)-d(y,\nu)|\\
&\le 2\gamma(y,\xi)\le 2\epsilon.
\end{align*}
By the arbitrariness of $y$, we have $\gamma(\mu,\nu)\le 2 \epsilon$.
Now, for every $x\in X$ there is $\xi(x)\in \mathcal{N}$ such that $\gamma(x,\xi(x))\le \epsilon$, and we can find a representative $\nu\in \check{\mathcal{N}}$, $[\nu]=[\xi(x)]$, so that $\gamma(\xi(x),\nu)\le 2 \epsilon$. Thus, by the triangle inequality, $\gamma(x,\nu)\le 3 \epsilon$.
\end{enumerate}
\end{remark}

A consequence of the previous observations is that every bounded Lorentzian-metric space contains a causet $\mathcal{N}$ which is an $\epsilon$-net and such that $d_{GH}(\mathcal{N},X)\le \epsilon$.

\begin{corollary} \label{vpa}
Every bounded Lorentzian-metric space is the Gromov-Hausdorff limit of causets.
\end{corollary}

\begin{proposition}\label{prop_nets}
A sequence of bounded Lorentzian-metric spaces $(X_m,d_m)$ containing the spacelike boundary converges to a bounded Lorentzian-metric space $(X,d)$ containing the spacelike boundary  iff for all $\varepsilon>0$ there exist finite
$\varepsilon$-nets $\mathcal{N}\subseteq X$ and $\mathcal{N}_m\subseteq X_m$ such that $\mathcal{N}_m \xrightarrow{{\small GH}} \mathcal{N}$ for $m\to\infty$.
Additionally, if convergence does apply, the finite $\varepsilon$-nets $\mathcal{N}$, $\mathcal{N}_m$ can be chosen to be causets such that $d_{GH}(\mathcal{N}, X)\le \epsilon$, $d_{GH}(\mathcal{N}_m, X_m)\le \epsilon$.
\end{proposition}

\begin{proof}
The last statement follows from Remark \ref{remark2}, \ref{bjd} \& \ref{bkd}. Just apply the first statement to a sequence of $\epsilon/6$-nets.

We start by assuming that $X_m \xrightarrow{GH} X$.  Let $\epsilon>0$ be given. Consider $m$ sufficiently
large such that $d_{GH}(X_m,X)< \epsilon/3$. Choose for every such $m$ a correspondence $R_m\subset X_m\times X$ with $\delta_m:=
\textrm{dis} R_m<\epsilon/3$ and $\delta_m\to 0$ for $m\to \infty$.

Next choose an $\epsilon/3$-net $\mathcal{N}\subset X$. For $m\in \mathbb{N}$ and $\eta\in \mathcal{N}$ select $\check r_m(\eta)\in X_m$
with $(r_m(\eta),\eta)\in R_m$ and define
\[
\mathcal{N}_m:=\{r_m(\eta):\, \eta\in \mathcal{N}\}.
\]

We claim that $\mathcal{N}_m$ is an $\epsilon$-net in $X_m$.
Let $x_m,y_m\in X_m$. Choose $x,y\in X$ with $(x_m,x),(y_m,y)\in R_m$.
We can find $\eta\in \mathcal{N}$ such that $\gamma(\eta,x)\le  \frac{\epsilon}{3}$, thus we have
\begin{align*}
|d_m(r_m(\eta),y_m)-d_m(x_m,y_m)|&\le |d(\eta,y)-d(x,y)|+2\textrm{dis} R_m \\
&\le \gamma(\eta,x)+2\textrm{dis} R_m \le \epsilon
\end{align*}
and
\begin{align*}
|d_m(y_m,r_m(\eta))-d_m(y_m,x_m)|&\le |d(y,\eta)-d(y,x)|+2\textrm{dis} R_m \\
&\le \gamma(\eta,x)+2\textrm{dis} R_m \le \epsilon
\end{align*}
By the arbitrariness of $y_m$, $\gamma_m(r_m(\eta),x)\le \epsilon$, that is, $\mathcal{N}_m$ is an $\epsilon$-net.

 Further the restriction of $R_m$ to $\mathcal{N}_m\times \mathcal{N}$ defines a
correspondence with distortion bounded by $\delta_m\le \epsilon/3$, hence $d_{GH}(\mathcal{N}_m, \mathcal{N})\to 0$.

For the converse implication let us assume that for every $\delta>0$ we have finite $\delta$-nets $\mathcal{N}_m\subset X_m$, $\mathcal{N}\subset X$,
such that $\mathcal{N}_m \xrightarrow{GH} \mathcal{N}$.

For given $\epsilon>0$ we can choose finite $\epsilon/6$-nets  $\mathcal{N}_m\subset X_m$, $\mathcal{N}\subset X$, such that $\mathcal{N}_m \xrightarrow{GH}
\mathcal{N}$ for $m\to\infty$. Thus for $m$ sufficiently large we have $d_{GH}(\mathcal{N}_m,\mathcal{N})\le \epsilon/3$ and therefore, by Remark \ref{remark2}(4),
\[
d_{GH}(X_m,X)\le d_{GH}(X_m,\mathcal{N}_m)+d_{GH}(\mathcal{N}_m,\mathcal{N})+d_{GH}(\mathcal{N},X)\le \epsilon.
\]
This proves that $X_m \xrightarrow{GH} X$.
\end{proof}

Finally, we relate the convergence of bounded Lorentzian metric spaces $(X,d)$ to the Gromov-Hausdorff convergence of the associated metric $(X,\gamma)$.

\begin{theorem}
Let $(X_n,d_n) \xrightarrow{{\small GH}} (X,d)$, then $(X_n,\gamma_n) \xrightarrow{{\small GH}} (X,\gamma)$.
\end{theorem}

\begin{proof}
By assumption we can find a correspondence $R_n\subset X\times X_n$ such that $\epsilon_n:=\textrm{dis}_d R_n\to 0$. We want to show that $\delta_n:=\textrm{dis}_\gamma R_n\to 0$. Let $x,y \in X$, and let $x_n,y_n\in X_n$, be any pair such that $(x,x_n)\in R_n$, $(y,y_n)\in R_n$.
We want to prove that $\vert \gamma(x,y)-\gamma_n(x_n,y_n)\vert \le 3 \epsilon_n$, from which  $\textrm{dis}_\gamma R_n\le 3 \epsilon_n$ follows and hence the thesis.

From the definition of $\gamma(x,y)$ we can find $w\in X$ such that
\[
\vert d(x,w)-d(y,w)\vert \ge \gamma(x,y)-\epsilon_n
\]
or the analogous inequality $\vert d(w,x)-d(w,y)\vert \ge \gamma(x,y)-\epsilon_n$. We work in the former case, the latter being analogous. Moreover, for every $z\in X$,
\[
\vert d(x,z)-d(y,z)\vert \le \gamma(x,y), \quad \textrm{and} \quad \vert d(z,x)-d(z,y)\vert \le \gamma(x,y).
\]
Let $w_n\in X_n$ be such that $(w,w_n)\in R_n$, then
\[
\vert d_n(x_n,w_n)-d_n(y_n,w_n)\vert \ge \gamma(x,y)-3\epsilon_n
\]
which implies
\[
\gamma_n(x_n,y_n)\ge \gamma(x,y)-3\epsilon_n.
\]
Observe that from the definition of $\gamma_n$ we get that  there is $z_n\in X_n$ such that $\vert d_n(x_n,z_n)-d_n(y_n,z_n)\vert \ge \gamma_n(x_n,y_n)-\epsilon_n$ (or the other analogous inequality with $z_n$ in the first entry) so we can find $z\in X$, $(z,z_n)\in R_n$ so that
\[
\gamma(x,y)\ge \vert d(x,z)-d(y,z)\vert \ge \vert d_n(x_n,z_n)-d_n(y_n,z_n)\vert-2 \epsilon_n\ge \gamma_n(x_n,y_n)-3\epsilon_n.
\]
hence $\vert \gamma(x,y)-\gamma_n(x_n,y_n)\vert \le 3\epsilon_n $ as desired.
\end{proof}

\begin{theorem}
Let $X,Y$ be bounded Lorentzian metric spaces. We have the inequality
\[
d_{GH}(X,Y)\le 2 \textrm{inf}_{X,Y\hookrightarrow Z } \, d^\gamma_H(X,Y)\le \max\{\textrm{diam} X, \textrm{diam} Y \} ,
\]
where the infimum is over all possible isometric injections of $X,Y$ in a bounded Lorentzian metric space $Z$. Here  $d^\gamma_H$ is the Hausdorff distance for $(Z,\gamma)$, with $\gamma$ the distinction metric on $Z$.
\end{theorem}

\begin{proof}
Suppose we one such injection.
Let $r>d_H^\gamma(X,Y)$, we need to find a correspondence $R\subset X\times Y$, with $\textrm{dis} R\le 2r$, from which the desired inequality would follow.
Let $R$ be given by all those pairs $(x,y)$ for which $\gamma(x,y)<r$. It is easy to check that this is a correspondence by the definition of the Hausdorff distance.

By Prop.\ \ref{nwq} if $(x_1, y_1)\in R$ and $(x_2, y_2)\in R$, as all these points can be seen as belonging to $Z$
\[
\vert d(x_1,x_2)-d(y_1,y_2) \vert \le \gamma(x_1,y_1)+\gamma(x_2,y_2)\vert<2r,
\]
which implies  that $\textrm{dis} R\le 2r$.

Now, let $Z= X \sqcup Y$ be the disjoint union (but identify the $i^0$s if both spaces  have), and define $d$ so that $d \vert_{X\times X}=d_X$, $d \vert_{Y\times Y}=d_Y$, and $d \vert_{X\times Y}=d \vert_{Y\times X}=0$. Then $Z$ is a bounded Lorentzian-metric space. Observe that by Prop.\ \ref{eqd} $d^\gamma_H(X,Y)\le \textrm{diam}^\gamma Z= \textrm{diam} Z=\max\{\textrm{diam} X, \textrm{diam} Y \}$.
\end{proof}

This result is important as it allows us to construct several examples of Lorentzian Gromov-Hausdorff convergence. Consider for instance a causal diamond in Minkowski spacetime, so as to get a bounded Lorentzian metric space. Two timelike hypersurfaces on it converging to each other would converge in a Gromov-Hausdorff sense, as the Hausdorff distance for the distinction metric goes to zero (a precise determination of the distinction distance is not required for this conclusion).

\section{ Causal relations \& Lorentzian length spaces}

\subsection{The (extended) causal relation}

\begin{definition}
Let $(X,d)$ be a bounded Lorentzian metric space. The relation $J\subset X\times X$ defined by
\begin{equation} \label{bmb}
J=\{(x,y)|\; \forall p\in X: d(p,y)\ge d(p,x) \ \& \ d(x,p)\ge d(y,p)\}
\end{equation}
is called the {\em (extended) causal relation}.  We write $x\le y$ for $(x,y)\in J$ and $x<y$ for $x\le y$ and $x\ne y$. We also write $J^+(x):=\{y: x\le y\}$, $J^-(x):=\{y: y\le x\}$.
\end{definition}

\begin{remark}
This causal relation will be very useful though it is intuitively a bit too large, particularly at the boundary points.
Let $X^+\subset X$ be the subset of points that do not admit a point in the chronological future, and analogously for $X^-$, so that $i^0=X^+\cap X^-$ provided $i^0\in X$.
Observe that if $(x,y)\in X^-\times X^+$
then $(x,y)\in J$.  Therefore, as examples obtained cutting out subsets of Minkowski spacetime show, the relation $J$ is in general strictly larger than the closure of $I$, and it is also different from the smallest closed, reflexive and transitive relation containing $I$.

On the other hand, if $y\notin X^+$ and assuming every neighborhood of $y$ intersects $I^+(y)$, every pair $(x,y)\in J$ lies in the closure of $I$. The reason for this is the reverse triangle
inequality proven in Theorem \ref{tuh} below. The time reversed case holds as well.\footnote{The expression (\ref{bmb}) makes sense also in Lorentzian manifolds. In that framework it is easy to prove that in causally simple spacetimes the relation coincides with the usual causal relation, see the arguments in Theorem \ref{tuh}.}

One should not expect the  just mentioned inclusion to hold without any assumption. Indeed, this is a ``no bubbling'' property  which, already for proper cone structures makes use of the `proper' condition \cite[Theorem 2.8]{minguzzi17}.  The assumption on the non-isolatedness of $y$ in terms of chronology  is a kind of metric replacement for the proper condition \cite[Theorem 2.2]{minguzzi17}. Observe that proper cone structures can collapse to closed cone structures, and in the process we expect that the no bubbling property can be lost. In other
words, we do not  expect the mentioned inclusion to be preserved under Gromov-Hausdorff limits. Examples of spaces with this property will be given
below in the context of chronally connected spaces, see Definition \ref{def_chr_conn}.
\end{remark}

\begin{remark}
If $(x,z)$ are such that $d(x,z)=0$ and $y\ll z$ then $(x,y) \notin J$ (and dually). In particular, in a causet this shows that $J$ is not much bigger than $I$, and that, roughly speaking, away from the boundary $J$ `wraps' $I$.
\end{remark}

The following result is stronger than \cite[Thm.\ 8(d)]{noldus04}
\begin{proposition}\label{conve}
For each $x\in X$ and $r>0$ the balls $\{z: \gamma(x,z)\le r\}$, $\{z: \gamma(x,z) < r\}$
 are causally convex.
\end{proposition}

\begin{proof}
The inclusion of $i^0$ in $X$ does not alter  $\gamma$ or $J$ on $\mathring{X}\times \mathring{X}$. Thus we can assume that $X$ contains $i^0$ and hence is compact.
We prove the version with the lesser-equal sign. Suppose that $\gamma(x,p)\le r$, $\gamma(x,q)\le r$, and let $y\in J^+(p)\cap J^-(q)$. By compactness the supremum in the definition of $\gamma(x,y)$ is realized by a point $w$, and we have four possible cases.
\begin{itemize}
\item[(1)] $\gamma(x,y)=d(w,y)-d(w,x)\le d(w,q)-d(w,x)\le \gamma(x,q)\le r$.
\item[(2)]    $\gamma(x,y)=d(w,x)-d(w,y)\le d(w,x)-d(w,p)\le \gamma(x,p)\le r$.
\item[(3)]  $\gamma(x,y)=d(y,w)-d(x,w)\le d(p,w)-d(x,w)\le \gamma(x,p)\le r$
\item[(4)]     $\gamma(x,y)=d(x,w)-d(y,w)\le d(x,w)-d(q,w)\le \gamma(x,q)\le r$.
\end{itemize}
The strict inequality sign case follows from the previous case, as there is $r'<r$ such that $p,q\in \{z: \gamma(x,z) \le  r'\}$.
\end{proof}

By considering balls of radius $\gamma(x,y)$ centered in $x$ and $y$ we get

\begin{corollary}\label{causconv}
For any pair of points $x,y\in X$ we have
\[
\gamma(x,z)\text{ and }\gamma(z,y)\le \gamma(x,y)
\]
for all $z\in J^+(x)\cap J^-(y)$.
\end{corollary}

We recall that for two relations $A,B\subset X\times X$ the composition $A\circ B$ is defined as in Eq.\ (\ref{kdc}).

\begin{definition}
Let $(X,d)$ be a bounded Lorentzian metric space. A continuous function $\tau \colon X\to \mathbb{R}$ with $\tau(x)<\tau(y)$ whenever $x<y$ is called a {\it time function}.
\end{definition}

\begin{theorem} \label{tuh}
The causal relation $J$ is closed, reflexive and transitive. If $(x,y)$, $(y,z)\in J$ then $d(x,y)+d(y,z)\le d(x,z)$. Moreover, $I\subset J$ and $I\circ J\cup J\circ I \subset I$.
Finally, there is a bounded time function $\tau:X\mapsto [-M,M]$, $M<\infty$,
which is 1-Lipschitz with respect to the distinction distance. In particular, $J$ is antisymmetric.
\end{theorem}

It is well known that if $J$ is closed then $J^+(p)$ and $J^-(p)$ are closed for every $p\in X$ (a closed preordered space is a semiclosed preordered space, see \cite{nachbin65}). If we assume that $i^0\in X$ then by the compactness of $X$, $J^+(p)\cap J^-(q)$ is compact for every $p,q\in X$.

\begin{proof}
Reflexivity is immediate. Closedness follows from the continuity of $d$. As for transitivity, let $(x,y), (y,z)\in J$, and let $p\in X$, then
\[
d(p,y)-d(p,x)\ge 0\text{ and }d(p,z)-d(p,y)\ge 0,
\]
which summed give $d(p,z)-d(p,x)\ge 0$, and similarly for the other type of inequality, hence $(x,z)\in J$.

Let us prove the inclusion $I\subset J$. Let $(x,y)\in I$ and let $p\in X$. If $p$ is such that $d(p,x)>0$ then $d(p,y)-d(p,x)\ge d(x,y)> 0$ follows from the reverse triangle inequality for $I$. If $d(p,x)=0$, then necessarily $d(p,y)-d(p,x)\ge 0$ as $d$ is non-negative. This shows the validity of one type of inequality, the other type being similarly proved.

Let $(x,y),(y,z)\in J$. If $d(x,y)=d(y,z)=0$ then as $d(x,z)\ge 0$, the reverse triangle inequality is satisfied. If the two distances are both positive, the reverse triangle inequality follows from that for $I$. If one distance is zero and the other positive, say $d(x,y)=0$, $d(y,z)>0$, then we need only show $d(x,z)- d(y,z)\ge 0$ which follows from $(x,y)\in J$. The other case is analogous. The formula $I\circ J\cup J\circ I \subset I$ is a consequence of the reverse triangle inequality.

Let us come to the last statement. We know that the Lorentzian distance $d$ is bounded by some constant $M<\infty$ (cf.\ Remark \ref{R1}).
Recall that by  Proposition \ref{prop(iv)} $(X,\mathscr{T})$ admits a dense denumerable subset
$\mathcal{S}=\{s_n\}_{n\in\mathbb{N}}$ that distinguishes points.
Let us consider the following function
\begin{equation} \label{for}
\tau :X \to [-M,M], \quad \tau (x):=\frac{1}{2} \left[\sum_{n=1} \frac{1}{2^n} d(s_n, x)-\sum_{n=1} \frac{1}{2^n} d( x,s_n)\right].
\end{equation}
It is continuous in the topology $\mathscr{T}$ as it is a uniform limit of continuous functions.
Let  two points $x,y\in X$, $(x,y)\in J$ be given.  Then we have $d(s_n, x)\le d(s_n,y)$ and $d(x,s_n) \ge d(y,s_n)$ for each $n\in\mathbb{N}$.
 As a consequence, $\tau(x)\le \tau(y)$. But $x\ne y$ implies that  they are distinguished by some $s_k$, which means that some of the previous inequalities were strict, which implies $\tau(x)<\tau(y)$.

For the Lipschitz property  note that
\begin{align*}
\vert \tau(y)-\tau(x)\vert&= \frac{1}{2} \vert \sum_n \frac{1}{2^n} [d(s_n, y)-d(s_n, x)]-\sum_n \frac{1}{2^n} [d( y,s_n)-d(x,s_n)]\vert \\
&\le \frac{1}{2}  \sum_n \frac{1}{2^n} \vert d(s_n, y)-d(s_n, x)\vert +\frac{1}{2} \sum_n \frac{1}{2^n} \vert d( y,s_n)-d(x,s_n)\vert\\
&\le  \gamma(x,y).
\end{align*}
\end{proof}

\subsection{Lorentzian length spaces}\label{sec_lor_length_spaces}

In a bounded Lorentzian metric space $(X,d)$ we call a continuous curve $\sigma\colon [a,b]\to X$ {\it isocausal} if $\sigma(t) <  \sigma(t')$ for all $t,t'\in [a,b]$ with $t<t'$.
Further, an isocausal curve $\eta\colon [a,b]\to X$ is {\it isochronal} if $\eta(t)\ll \eta(t')$ for all $t,t'\in [a,b]$, with $t<t'$.

We remark that isocausal curves are continuous by definition thus we shall often omit this adjective.

\begin{proposition}
Let $(X,d)$ be a bounded Lorentzian metric space and $\tau\colon X\to\mathbb{R}$ be a time function on $X$. Then every isocausal curve can be reparameterized to be $\tau$-uniform,
i.e.\ for $\sigma\colon I\to X$ there exist $a<b\in \mathbb{R}$ and a monotone increasing homeomorphism $\varphi\colon [a,b]\to I$ such that $\tau(\sigma\circ\varphi(s))=s$.
\end{proposition}

\begin{proof}
Let $\sigma\colon I\to X$ be isocausal. Then the map $\tau\circ\sigma \colon I\to \mathbb{R}$ is continuous and strictly monotone increasing, i.e. a homeomorphism onto its image.
Denote with $[a,b]$ the  image of the map. Then $\varphi:= (\tau\circ\sigma)^{-1}\colon [a,b]\to I$ is the claimed homeomorphism.
\end{proof}

\begin{definition}
An isocausal curve $\sigma\colon [a,b]\to X$ is {\it maximal} (or {\it maximizing}) if
\begin{equation} \label{jji}
d(\sigma(t),\sigma(t'))+d(\sigma(t'),\sigma(t''))=d(\sigma(t),\sigma(t'')).
\end{equation}
for every $t,t',t''\in [a,b]$, $t< t'< t''$.
\end{definition}
Clearly, it is equivalent to the same property for $t\le t'\le t''$.

Note that it might be the case that  $d(\sigma(t),\sigma(t'))=0$ for some $t<t'\in[a,b]$ though $d(\sigma(a),\sigma(b))>0$.
Any restriction  $\sigma|_{[c,e]}$, $[c,e]\subset [a,b]$, $c < e$, is also maximal
by the fact that any restriction of an isocausal curve is isocausal.

In order to formulate the limit curve theorem, it is convenient to extend the interval of definition of the curves (note that the extension is not isocausal) mostly due to slight complications caused by the fact that each curve has its own domain of definition. Alternatively, we could proceed using the notion of uniform convergence introduced in
\cite[Definition 2.1]{minguzzi07c}.

For an isocausal curve
$\sigma \colon [a,b]\to X$ we set

\begin{equation} \label{gid}
\hat{\sigma}\colon \mathbb{R}\to X,\quad s\mapsto \begin{cases}
\sigma(a),& \text{ if }s< a\\
\sigma(s),& \text{ if }s\in[a,b]\\
\sigma(b),& \text{ if }s>b.
\end{cases}
\end{equation}

\begin{definition}
Let $\sigma_n:[a_n,b_n] \to X$, $\sigma: [a,b]\to X$ be  isocausal curves.
We say that $\sigma_n$ {\em converges uniformly} to $\sigma$ with respect to the metric (distance) $h$ if $a_n \to a$, $b_n\to b$, and $\hat \sigma_n$  converges uniformly to $\hat \sigma$ with respect to the metric $h$. In particular, $\sigma_n$ converges to $\sigma$ pointwisely.
\end{definition}

\begin{remark}
In a bounded Lorentzian metric space that contains $i^0$, $(X,\mathscr{T})$ is a  compact metrizable space (the topology $\mathscr{T}$ being induced by $\gamma$) and hence admits a unique uniform structure \cite[Theorem 36.19]{willard70}. Thus, although we shall prove the following theorem by using the uniformity induced by $\gamma$, there is really no need to mention it in the statement of the theorem.
\end{remark}

\begin{theorem}[Limit Curve Theorem]\label{LCT}
Let $(X,d)$ be a bounded Lorentzian metric space that includes $i^0$. Let  $\sigma_n: [a_n,b_n]\to  X$ be  a sequence of
isocausal curves parametrized with  respect to a given time function $\tau$, $\tau(\sigma(t))=t$. Suppose that $\{a_n\}$ has no accumulation point in common with $\{b_n\}$.
Then there  exists  a  isocausal curve $\sigma:[a,b]\to X$
 and a subsequence $\{\sigma_{n_k}\}_k$ that converges uniformly to $\sigma$.
If the curves $\sigma_n$ are maximizing (i.e.\ Equation \eqref{jji} holds) then so is $\sigma$.
\end{theorem}
The last result is really just a consequence of pointwise convergence.

\begin{proof}
 Set   $x_n:=\sigma_n(a_n)$, $y_n:=(\sigma_n(b_n))$. We can pass to a subsequence, denoted in the same way, so that they converge, $x_n\to x$, $y_n\to y$. But
 $\tau(x_n)=\tau(\sigma_n(a_n))=a_n$ thus in the limit $a_n\to a:=\tau(x)$, and similarly $b_n\to b:=\tau(y)$. Set $\sigma(a):=x$ and $\sigma(b):=y$.
 Observe that $a< b$ due to $a_n\le b_n$ and the accumulation condition on $\{a_n\}$ and $\{b_n\}$.

For $q\in \mathbb{Q}$  consider the sequence  $\{\hat{\sigma}_n(q)\}_n$ with $n$ sufficiently large. Since $(X,\mathscr{T})$ is compact  the sequence admits a converging subsequence. Via a Cantor diagonal argument we can pass to a subsequence so that  the sequence $\{\hat{\sigma}_n(q)\}_n$ converges to some point that we denote $\hat{\sigma}(q)$  for every $q\in \mathbb{Q}$.

Observe that $\tau(\hat{\sigma}_n(q))=q$ for $q\in (a,b)\cap\mathbb{Q}$ and $n$ sufficiently large gives  $\tau(\hat{\sigma}(q))=q$ in the limit. Moreover,
 $q<q'$, $q,q'\in (a,b)\cap \mathbb{Q}$ implies $\sigma_n(q)<\sigma_n(q')$  for $n$ sufficiently large, and since $J$ is closed we obtain $\hat{\sigma}(q)\le \hat{\sigma}(q')$
 in the limit. But $\tau$ takes different values on these two points thus $\hat{\sigma}(q)<\hat{\sigma}(q')$.

We shall prove in a moment that, by a similar reasoning, $\hat{\sigma}\colon \mathbb{Q}\to X$ is  {\it Cauchy continuous}, i.e. it maps Cauchy sequences to Cauchy sequences. Then, by a standard result in topology (recall that by Theorem \ref{Theo_kurat}, $(X,\gamma)$ is a compact metric space, hence complete),
$\hat{\sigma}$ has a unique continuous extension to $\mathbb{R}$, that we denote in the same way. Observe that by continuity of the extension $\hat \sigma$, $\tau(\hat{\sigma}(r))=r$ for every $r\in [a,b]$. Similarly, by continuity of $\hat \sigma$ and the closure of $J$, $\hat{\sigma}(r)\le \hat{\sigma}(r')$, for every $r,r'\in [a,b]$. If $r,r'\in [a,b]$ are such that $r<r'$ then as $\tau$ takes different values on $\hat{\sigma}(r)$ and $\hat{\sigma}(r')$, we get $\hat{\sigma}(r)< \hat{\sigma}(r')$, namely $\sigma:= \hat \sigma\vert_{[a,b]}$ is a  isocausal curve.

Observe that for $q\in \mathbb{Q}$, $q< a$, we have for sufficiently large $n$, $q<a_n$, hence $\hat \sigma_n(q)=x_n$, which implies taking the limit, $\hat \sigma(q)=x$. Any continuous extension $\hat \sigma$ satisfies $\hat \sigma(r)=x$ for every $r\in \mathbb{R}$, $r\le a$, and analogously for $r\ge b$. This proves that $\hat \sigma$ is indeed the extension of $\sigma$ in the sense of Eq.\ (\ref{gid}).

In order to prove the Cauchy continuity  of $\hat{\sigma}\colon \mathbb{Q}\to X$,  let us proceed by contradiction.
Suppose that there exists a Cauchy sequence $\{q_n\}\subset \mathbb{Q}$ that does not map to a $\gamma$-Cauchy sequence. Let $r\in \mathbb{R}$ be the limit of $q_n$.  There is $\varepsilon >0$
 and two subsequences $\{q_{m_k}\}_k,\{q_{n_k}\}_k\subset \{q_n\}$ with
\[
\gamma(\hat{\sigma}(q_{m_k}),\hat{\sigma}(q_{n_k}))\ge \varepsilon
\]
for all $k\in\mathbb{N}$.
Without loss of generality we can assume that $q_{m_k}<q_{n_k}$ for infinitely many $k$ and that both sequences converge with $\hat{\sigma}(q_{m_k})\to w$
and $\hat{\sigma}(q_{n_k})\to z$, respectively. We have $\gamma(w,z)\ge \varepsilon$.  Observe that $\hat{\sigma}(q_{m_k}) \le \hat{\sigma}(q_{n_k})$, hence $w\le z$, and finally $w<z$.

For $r\in (a,b)$ we have for sufficiently large $k$, $q_{m_k}, q_{n_k}\in (a,b)$.
 Since $\tau(\hat{\sigma}(q_{m_k}))=q_{m_k}\to r$, $\tau(\hat{\sigma}(q_{n_k}))=q_{n_k}\to r$ we obtain
$\tau(w)=\tau(z)=r$.
This contradicts the properties of the time function $\tau$.

If $r\le a$, we can prove that for all but a finite number of $k$, it must be $q_{n_k}>  a_n $ for sufficiently large $n$. Indeed, otherwise there are infinite values of $k$ for which $q_{m_k},q_{n_k} \le  a_n $ for infinitely many $n$, which implies  $\hat \sigma_n(q_{m_k})=\hat \sigma_n(q_{m_k})=x_n$,
and  for $n\to \infty$,  $\hat{\sigma}(q_{m_k})=\hat{\sigma}(q_{n_k})=x$, and finally for $k\to \infty$, $w=z=x$, $\gamma(x,x)\ge \epsilon>0$, a contradiction.
 However, as $k$ is such that  $q_{n_k}>  a_n $ for every sufficiently  large $n$, we have $\tau(\hat{\sigma}_n (q_{n_k}))=q_{n_k}>a_n$, and in the limit $n\to \infty$,  $\tau(\hat{\sigma}(q_{n_k}))=q_{n_k}\ge a$ and hence in the limit $k \to \infty$, $\tau(z)=r=a$ (as we are in the case $r\le a$). Now there are two cases dependent on the sequence $q_{m_k}$. If  for all but a finite number of $k$, we have  $q_{m_k}\le a_n $ for sufficiently large $n$, then $\hat \sigma_n(q_{m_k})= x_n$ and taking the limit  $n\to \infty$, $\hat \sigma(q_{m_k})=x$, and taking the limit $k\to \infty$, $w=x$,  which implies $\tau(w)=a$. If, instead, there are infinite $k$ for which $q_{m_k}> a_n$ for infinitely many $n$, then as proved above for $q_{n_k}$ we get $\tau(w)=a$.
 In any  case  $\tau(w)=\tau(z)=a$  which contradicts the properties of the time function $\tau$. The case $r\ge b$ is analogous.

Next we show that $\{\hat{\sigma}_n\}$ admits a subsequence converging uniformly to $\hat{\sigma}$.
Choose $n\in\mathbb{N}$ sufficiently large such that $a_n>a-1$ and $b_n<b+1$. Let
$\varepsilon >0$ be given. Choose $m\in \mathbb{N}$ and rational parameters $s_0<a-1<s_1<\ldots <b+1<s_m$ such that
\[
\gamma(\hat{\sigma}(s_i),\hat{\sigma}(s_{i+1}))<\varepsilon.
\]
For the subsequence constructed above we can choose $N$ so large that
\[
\gamma(\hat{\sigma}_n(s_i),\hat{\sigma}(s_i))<\varepsilon
\]
for all $n>N$ and $0\le i\le m-1$. Now for $s\in [a-1,b+1]$ and $n> N$ choose $i$ with $s\in [s_i,s_{i+1}]$. It follows with Corollary \ref{causconv}
\begin{align*}
\gamma(\hat{\sigma}_n(s),\hat{\sigma}(s))&\le \gamma(\hat{\sigma}_n(s),\hat{\sigma}_n(s_i))+\gamma(\hat{\sigma}_n(s_i),\hat{\sigma}(s_i))+
\gamma(\hat{\sigma}(s_i),\hat{\sigma}(s))\\
&\le \gamma(\hat{\sigma}_n(s_{i}),\hat{\sigma}_n(s_{i+1}))+\gamma(\hat{\sigma}_n(s_i),\hat{\sigma}(s_i))+
\gamma(\hat{\sigma}(s_i),\hat{\sigma}(s_{i+1}))\\
&
\le [\gamma(\hat{\sigma}_n(s_{i}),\hat{\sigma}(s_{i}))+\gamma(\hat{\sigma}(s_{i}),\hat{\sigma}(s_{i+1}))
+\gamma(\hat{\sigma}(s_{i+1}),\hat{\sigma}_n(s_{i+1}))]+2\varepsilon
 \\
&\le 3\varepsilon +2\varepsilon=5\varepsilon,
\end{align*}
i.e. the chosen subsequence of $\{\hat{\sigma}_n\}_n$ converges uniformly to $\hat{\sigma}$ on $[a-1,b+1]$. Since for all $n$ sufficiently large all curves are constant
on the complement of $[a-1,b+1]$ we see that the convergence is uniform on all of $\mathbb{R}$.

The equations
\[
d(\sigma_n(t),\sigma_n(t'))+d(\sigma_n(t'),\sigma_n(t''))=d(\sigma_n(t),\sigma_n(t''))
\]
clearly pass in the limit to Equation \eqref{jji} for $\sigma$.
\end{proof}

\begin{definition}\label{def_chr_conn}
A bounded Lorentzian metric space $(X,d)$ is a {\it (bounded) Lorentzian prelength space} if for every pair $x,y\in X$ with $x\ll y$ there exists an isocausal curve $\sigma: [0,1]\to X$
from $x$ to $y$.
\end{definition}

We refer to the above property as {\em chronal connecteness by isocausal curves}.
\begin{proposition} \label{pby}
Let $(X,d)$ be a bounded Lorentzian prelength space that includes $i^0$. Then for every $(x,y)\in \bar I$, $x \ne y$,  there exists a  isocausal curve
$\sigma: [a,b]\to X$ connecting $x$ to $y$.
\end{proposition}

\begin{proof}
Let $\tau$ be a time function and let $(x_n,y_n)\to (x,y)$, $(x_n,y_n)\in I$.
Since $J$ is closed, $\bar I\subset J$, thus $x<y$, which implies $\tau(x)<\tau(y)$. Let $\sigma_n:[a_n,b_n]\to X$ be  isocausal curves connecting $x_n$ to $y_n$ and parametrized with the time function.
The assumptions of the limit curve theorem is satisfied as $a_n\to \tau(x)$, $b_n\to \tau(y)$, which differ.
\end{proof}

\begin{definition}
Let $(X,d)$ be a bounded Lorentzian prelength space. The {\it restricted causal relation} $\check J\subset J$ is defined as the set of pairs of points connected by an isocausal curve
or coincident.
\end{definition}

\begin{corollary}
Let $(X,d)$ be a bounded Lorentzian prelength space. The restricted causal relation $\check J$  is reflexive, transitive, antisymmetric and closed.  Moreover, $I\subset \check{J}$ and $I\circ \check{J}\cup \check{J}\circ I \subset I$.
\end{corollary}

The closure of $\check J$ follows from the same limit curve argument  used in the proof of Proposition\ \ref{pby}.

The relation $\check J$ is meant to cure the fact that, as previously observed, $J$ is too large, particularly at boundary points.

\begin{definition}\label{def_max}
A bounded Lorentzian prelength space $(X,d)$ is a {\em bounded Lorentzian length space} if for every pair $x,y\in X$ with $x\ll y$ there exists a maximal  isocausal curve connecting $x$ and $y$.
\end{definition}

We refer to the above property as {\em maximal chronal connectedness by isocausal curves}.

Observe that if an isocausal curve $\sigma\colon [0,1]\to X$ connects two points $x,y$ such that $d(x,y)=0$, then  $d(\sigma(t),\sigma(t'))=0$
for every $t,t'\in [0,1]$. In particular it is maximizing in the sense of  Equation \eqref{jji}. Thus any two distinct $\check J$-related points in a bounded Lorentzian length space are connected by a maximizing isocausal curve.

\begin{theorem} \label{bjq}
Let $\{(X_n,d_n)\}_n$ be a sequence of bounded Lorentzian length spaces containing $i^0$ converging in the Gromov-Hausdorff sense to a bounded Lorentzian metric space $(X,d)$
containing $i^0$. Then $(X,d)$ is a  bounded Lorentzian length space. The same statement holds for {\em prelength} replacing {\em length}.
\end{theorem}

\begin{proof}
We proceed by proving the former statement, the latter being obtained dropping the maximization condition on the  curves $\sigma_n$ introduced below.

Let $x,y\in X$ with $x\ll y$ be given. In order to construct a maximal isocausal curve $\zeta\colon [0,1]\to X$ between $x$ and $y$ we first define its values for rational parameter values
$a\in [0,1]$ and extend this later to all real numbers $\tau \in [0,1]$. The resulting map will be shown to be maximizing in the sense of Definition \ref{def_max}.

Set $\epsilon_n:=d_{GH}(X,X_n)$ and fix a correspondence $R_n\subset X\times X_n$  with $\textrm{dis}R_n \le 2\epsilon_n$ (we can assume that $\epsilon_n>0$ otherwise
the conclusion follows by isometry).
Choose $x_n,y_n\in X_n$ such that $(x,x_n), (y,y_n)\in R_n$. Since $d(x,y)>0$ we can assume that $d_n(x_n,y_n)>0$, namely $x_n\ll y_n$
(otherwise pass to a subsequence).

Let $\mathcal{S}=\{s_k\}_{k\in \mathbb{N}}$ be a dense subset of  $(X,d)$.
It distinguishes points  in  $(X,\mathscr{T})$  by  Proposition \ref{prop(iv)}.
Choose $s^n_k\in X_n$ with $(s_k,s^n_k)\in R_n$ and set $\mathcal{S}_n:=\{s^n_k\}_{k\in\mathbb{N}}$.
Notice that in general $\mathcal{S}_n$ might not distinguish points in $X_n$.

Consider the time function (the proof is the same as that provided for the function \ref{for})
\begin{equation} \label{for2}
\tau \colon X\to \R,\quad \tau (z):=\alpha\left[\sum_{k=1} \frac{1}{2^k} d(s_k, z)-\sum_{k=1} \frac{1}{2^k} d( z,s_k)\right] +\beta,
\end{equation}
where constants $\alpha>0$ and $\beta$ are chosen such that $\tau (x)=0$, $\tau (y)=1$.  By replacing $d_{s_k}$ with $(d_n)_{s^n_k}$, and $d^{s_k}$ with $(d_n)^{s^n_k}$ we obtain
a continuous function $\tau_n:X_n\to \mathbb{R}$ such that $\tau_n(p)\le \tau_n(q)$ for every $(p,q)\in J_n$. Note, however, that it is not necessarily the case that $p_n,q_n\in X_n$, $p_n<_n q_n$, implies
$\tau_n(p_n)<\tau_n(q_n)$. For every $p\in X$, $p_n\in X_n$ with $(p,p_n)\in R_n$ we have
\[
\vert \tau(p)-\tau_n(p_n) \vert \le 4 \alpha \epsilon_n,
\]
which implies $\tau_n(x_n)\to 0$, $\tau_n(y_n)\to 1$ as a special case.

Let $\sigma_n: [0,1] \to X_n$ be a maximizing isocausal  curves connecting $x_n$ to $y_n$. For each rational number $a\in (0,1)$, choose a point $z_n(a)$ on
$\sigma_n$ with
\[
\tau_n(z_n(a))=a.
\]
The choice for $z_n(a)$ might not be unique. Set $z_n(0)=x_n$, $z_n(1)=y_n$. Notice that, due to $\tau_n(x_n)\to 0$ and $\tau_n(y_n)\to 1$, for fixed $a' \in (0,1)$, the curve $\sigma_n$ intersects the slice $\tau_n^{-1}(a')$ for all $n$ sufficiently large. For  a rational number $a\in (0,1)$ choose $\zeta_n(a) \in X$ such that $(\zeta_n(a), z_n(a))\in R_n$. We set $\zeta_n(0)=x$ and $\zeta_n(1)=y$. Via a Cantor type argument, we can pass to a subsequence so that for every $a\in [0,1]$, $\zeta_n(a)$ converges to a point $\zeta(a)\in X$.  Observe that $\zeta(0)=x$ and $\zeta(1)=y$. From  $(\zeta_n(a), z_n(a))\in R_n$ we get
\[
\vert \tau(\zeta_n(a))-a   \vert=\vert \tau(\zeta_n(a))-\tau_n(z_n(a) ) \vert \le 4 \alpha \epsilon_n
\]
which, by the continuity of $\tau$, implies $\tau(\zeta(a))=a$.

Let rational numbers $a,b\in [0,1]$ with $a<b$ be given.
Both $z_n(a)$ and $z_n(b)$ belong to the image of $\sigma_n$, but it cannot be $z_n(b) \le z_n(a)$, otherwise, as $\tau_n$ is non-decreasing over isocausal curves, $b=\tau_n(z_n(b))\le \tau_n(z_n(a))=a$, which gives a contradiction. Thus $a<b$ implies  $z_n(a) < z_n(b)$.

Let $a,b,c\in [0,1]$, be rational numbers such that $a<b<c$.
The property
\[
d_n(z_n(a), z_n(b))+d_n(z_n(b),z_n(c))=d_n(z_n(a),z_n(c))
\]
implies
\[
\vert d(\zeta_n(a), \zeta_n(b))+d(\zeta_n(b),\zeta_n(c)) - d(\zeta_n(a),\zeta_n(c))  \vert \le 6 \epsilon_n.
\]
Taking the limit we obtain
\begin{equation} \label{kkj}
d(\zeta(a), \zeta(b))+d(\zeta(b),\zeta(c)) = d(\zeta(a),\zeta(c)).
\end{equation}

Let $a,b\in [0,1]$ be rational numbers with $a<b$. Next we prove that $(\zeta(a),\zeta(b))\in J$. Let $w\in X$ be given and choose $w_n\in X_n$ such that $(w,w_n)\in R_n$. Then as $(z_n(a),z_n(b))\in  J_n$ we have
\[
d_n(w_n, z_n(a)) \le d_n(w_n, z_n(b)), \qquad \textrm{and} \qquad d_n(z_n(a),w_n)\ge d_n(z_n(b), w_n)
\]
which imply
\[
d(w, \zeta_n(a))\le d(w,\zeta_n(b))+4 \epsilon_n,  \qquad \textrm{and} \qquad d(\zeta_n(a),w)\ge d(\zeta_n(b), w) -4 \epsilon_n.
\]
By taking the limit we obtain
\[
d(w, \zeta(a))\le d(w,\zeta(b)),  \qquad \textrm{and} \qquad d(\zeta(a),w)\ge d(\zeta(b), w) .
\]
which, by the arbitrariness of $w$, implies $(\zeta(a),\zeta(b)) \in J$.
Since $\tau $ is a time function and $\tau(\zeta(a))=a$ as well as $\tau(\zeta(b))=b$, we conclude $\zeta(a)<\zeta(b)$.

Next we extend $\zeta$  continuously to the whole interval $[0,1]$ as an isocausal curve.

Let us proceed by contradiction as in the proof of Theorem \ref{LCT}.
Suppose that there exists a Cauchy sequence $\{q_n\}\subset \mathbb{Q}\cap [0,1]$ that does not map to a $\gamma$-Cauchy sequence. Let $r\in \mathbb{R}$ be the limit of $q_n$.  There is $\varepsilon >0$
 and two subsequences $\{q_{m_k}\}_k,\{q_{n_k}\}_k\subset \{q_n\}$ with
\[
\gamma({\zeta}(q_{m_k}),{\zeta}(q_{n_k}))\ge \varepsilon
\]
for all $k\in\mathbb{N}$.
Without loss of generality we can assume that $q_{m_k}<q_{n_k}$ for infinitely many $k$ and that both sequences converge with ${\zeta}(q_{m_k})\to w$
and ${\zeta}(q_{n_k})\to z$, respectively. We have $\gamma(w,z)\ge \varepsilon$.  Observe that ${\zeta}(q_{m_k}) \le {\zeta}(q_{n_k})$, hence $w\le z$, and finally $w<z$.
Since $\tau({\zeta}(q_{m_k}))=q_{m_k}\to r$, $\tau({\zeta}(q_{n_k}))=q_{n_k}\to r$ we obtain
$\tau(w)=\tau(z)=r$.
This contradicts the properties of the time function $\tau$.

 All the properties of the function $\zeta$ such as $\tau(\zeta(a))=a$, for $a\in [0,1]$, its isocausality  or its maximality are readily obtained by continuity from the same properties on the restriction $\zeta\vert_{\mathbb{Q}}$.
 \end{proof}

\begin{example}
We want to exhibit an example of a bounded Lorentzian length space $(X,d)$ which contains two points $x\ll y$ such that every maximal isocausal curve $\gamma\colon I\to X$ connecting $x$ and $y$ admits no isochronal subarc, i.e. for all $s<t\in I$ there exists $s\le s'<t'\le t$ with $\gamma(s')\not\ll \gamma(t')$.

Let us consider the Minkowski spacetime $(\mathbb{R}^3,g)$, with
\[
g=-(dx+dy)(dx+2dy)+dz^2 ,
\]
where $(x,y,z)$ are linear coordinates. We define the sequence
\[
C_n:=\left\{t\in [0,1]\left|\; \exists c_1,\ldots,c_n\in \{0,2\}: \sum_{k=1}^n \frac{c_k}{3^k}\le t \le \sum_{k=1}^n \frac{c_k}{3^k}+\frac{1}{3^n}\right.\right\}
\]
and the sequences of functions $f_n,h_n\colon [0,1]\to\mathbb{R}$
\[
f_n(t):=\begin{cases}
\left(\frac{3}{2}\right)^n&,\; t\in C_n\\ 0&,\; t\notin C_n,
\end{cases}
\]
as well as
\[
h_n(t):=1-\left(\frac{2}{3}\right)^n f_n(t)=\chi_{[0,1]\setminus C_n}.
\]
Let $F_n(t):=\int_0^t f_n(\tau)d\tau$, $H_n(t):=\int_0^t h_n(\tau)d\tau$. Note that $\bigcap_n C_n$ is the $\frac{1}{3}$-Cantor set and $F_n$ converges uniformly to the Cantor function,
also known as the {\it Devil's staircase}.

Next, let us consider the curves
\[
\gamma_n\colon [0,1]\to\R^3,\quad \gamma_n(t):=(t,F_n(t),H_n(t)).
\]
The differential of $\gamma_n$ exists for all $t\in [0,1]$ except at finitely many points and the tangent vectors are $g$-causal, whenever they exist. Therefore $\{\gamma_n\}_{n\in\mathbb{N}}$ is a sequence
of casual curves between $\gamma_n(0)=(0,0,0)\in \mathbb{R}^3$ and
\[
\gamma_n(1)= \left(1,1,1-\left(\frac{2}{3}\right)^n\right).
\]
The length of $\gamma_n$ is given by
\begin{align*}
L^g(\gamma_n)
&=\int_0^1\sqrt{|g(\dot{\gamma}_n(\tau),\dot{\gamma}_n(\tau)|}d\tau\\
&=\frac{2^n}{3^n} \sqrt{\left(1+\left(\frac{3}{2}\right)^n\right)\left(1+2\left(\frac{3}{2}\right)^n\right)}\\
&=\sqrt{\left(\left(\frac{2}{3}\right)^n+1\right)\left(\left(\frac{2}{3}\right)^n+2\right)}.
\end{align*}
The lengths converge to $\sqrt{2}$ for $n\to \infty$. Note that the sequences $\{\gamma_n(t)\}_n$ are eventually constant in $n$ for an open and dense set of $t\in[0,1]$, more precisely
$\{\gamma_n(t)\}_n$ is eventually constant for parameters $t$ in complement of the $\frac{1}{3}$-Cantor set.

Therefore, the $\gamma_n(t)$ converge uniformly to a continuous curve $\gamma(t)$ connecting $(0,0,0)$ to
$(1,1,1)$.  This curve $\gamma$ and all the $\gamma_n$ are endowed with the intrinsic  Lorentzian distance (i.e.\ that obtained from the integral of the Lorentzian length on subintervals).

%\xout{from the Minkowski spacetime}} (note that $\gamma$ is not absolutely continuous{\color{red}, even though it admits a Lipschitz reparameterization}, thus it is not a causal curve in the sense of low regularity Lorentzian geometry as developed, e.g.\, in \cite[Sec.\ 2.1]{minguzzi17}, but will be soon shown to be an isocausal curve in our sense).

The images $\gamma_n([0,1])$ are not bounded Lorentzian metric spaces, since points on the null segments of $\gamma_n|_{\supp h_n}$ cannot be distinguished inside the set. In order to remedy this defect we extend
$\gamma_n([0,1])$ by adding to it the set
\[
S_n:=\bigcup_{t\in \supp h_n} \{(t,F_n(t),s)|\; s\in [0,1]\}.
\]
After suitably deleting/identifying some corners of the $S_n$'s we see that the union
\[
X_n:=\gamma_n([0,1])\cup S_n
\]
with the induced Lorentzian distance and topology from $(\mathbb{R}^3,g)$ is a bounded Lorentzian length space (while all the points in $\gamma_n$ are now distinguished, some pairs of points in $S_n$,  might be non-distinguished, but this can be remedied by passing to the usual quotient).

Since $S_n\subset S_{n+1}$ for all $n\in\mathbb{N}$ and $d_{GH}(S_m,S_n)\le \diam(S_m\setminus S_n)\le \frac{1}{3^n}$ for $m\ge n$ it is not difficult to see that the sequence $\{(X_n,d_n)\}_n$
 GH-converges to the bounded Lorentzian length space $(X,d)$, where $X:=\gamma([0,1])\cup \bigcup_n S_n$ and where $d$ is the  intrinsic  Lorentzian distance. As a consequence, $(X,d)$ is a  Lorentzian length space. Up to reparameterization the only isocausal curve connecting $(0,0,0)$ and $(1,1,1)$ in $X$ is the limit curve $\gamma$, which is maximal but has no isochronal subarc.
\end{example}

\section{Length and Curvature}

\subsection{The length functional}
This section can be skipped on first reading. Its goal is to introduce the length functional though most of the theory can be developed without this concept.
\begin{definition}
Let $(X,d)$ be a bounded Lorentzian metric space, and let $\sigma: [0,1]\to X$ be a  isocausal curve. Its {\em Lorentzian length} is
\[
L(\sigma):=\textrm{inf} \sum_{i=0}^{k-1} d(\sigma(t_i),\sigma(t_{i+1}))
\]
where the infimum is over the set of all partitions $\{t_0,t_1,t_2, \cdots, t_k\}$, $k\in \mathbb{N}$, $t_i\in [0,1]$, $t_i<t_{i+1}$, $t_0=0$, $t_k=1$.
\end{definition}

Clearly, by the reverse triangle inequality, if $x=\sigma(0)$, $y=\sigma(1)$,
\[
L(\sigma)\le d(x,y).
\]

\begin{proposition} \label{ffr}
Let $\sigma: [0,1]\to X$ be a isocausal curve with  endpoints $x$ and $y$.
We have $L(\sigma)=d(x,y)$ iff $\sigma$ is maximal.
\end{proposition}

\begin{proof}
 Let $\sigma\colon [0,1]\to X$ be a maximal isocausal curve. For a given partition $0=t_0<t_1<\ldots<t_{n}<t_{n}=1$ we conclude inductively
\begin{align*}
d(\sigma(0),\sigma(1))&=d(\sigma(0),\sigma(t_1))+d(\sigma(t_1),\sigma(1))\\
& \vdots \\
&=\sum_{k=0}^{n-1} d(\sigma(t_k),\sigma(t_{k+1}))
\end{align*}
Therefore $L(\sigma)=d(\sigma(0),\sigma(1))$.

Now let $\sigma\colon [0,1]\to X$ be an isocausal curve with $L(\sigma)=d(\sigma(0),\sigma(1))$.

 Let $t,t',t''\in [0,1]$, $t\le t'\le t''$.  We have
\begin{equation}
d(\sigma(t),\sigma(t'))+d(\sigma(t'),\sigma(t''))\le d(\sigma(t),\sigma(t'')).
\end{equation}
then
\begin{align*}
L(\sigma)\le& d(x,\sigma(t))+d(\sigma(t),\sigma(t'))+d(\sigma(t'),\sigma(t''))+d(\sigma(t''),y)\\
\le &d(x,\sigma(t))+d(\sigma(t),\sigma(t''))+d(\sigma(t''),y)\le d(x,y)=L(\sigma)
\end{align*}
which implies that all inequalities are equalities and hence Equation \eqref{jji} holds true.
\end{proof}

We observe that the following result does not depend on the existence of convex neighborhoods (for a similar observation in the smooth manifold context see \cite{minguzzi17})

\begin{theorem}[Upper semi-continuity of the length functional] $\empty$ \\
Let $\sigma_n: [0,1]\to X$ and $\sigma: [0,1]\to X$ be  isocausal curves and suppose that $\sigma_n\to \sigma$ pointwisely. Then
\[
\limsup L(\sigma_n)\le L(\sigma).
\]
\end{theorem}

\begin{proof}
We need to show that for every $\epsilon>0$, we have for sufficiently large $n$, $L(\sigma_n) \le L(\sigma)+\epsilon$.

Let $\epsilon>0$. By definition of $L(\sigma)$ there is a partition  $\{t_i, i=0,\cdots,k\}$ of $[0,1]$ such that  $\sum_i d(\sigma(t_i),\sigma(t_{i+1}))\le L(\sigma)+\frac{\epsilon}{2}$. By the continuity of $d$, as $\sigma_n(t_i)\to \sigma(t_i)$, we have for sufficiently large $n$ and for every $i$, $d(\sigma_n(t_i),\sigma_n(t_{i+1})) \le d(\sigma(t_i),\sigma(t_{i+1}))+\frac{\epsilon}{2 k}$, which implies
\begin{align*}
L(\sigma_n)&=\sum_i L(\sigma_n\vert_{[t_i,t_{i+1}]}) \le \sum_i d(\sigma_n(t_i),\sigma_n(t_{i+1})) \\ &\le \sum_i  d(\sigma(t_i),\sigma(t_{i+1})) +\frac{\epsilon}{2}\le L(\sigma)+ \epsilon.
\end{align*}
\end{proof}

We recall that the concept of Lorentzian prelength spaces is quite interesting as it is invariant under GH-limits. Let $(X,d)$ be
 a Lorentzian prelength space.  For each pair $x\ll y$ it makes sense to define the following object
\[
\check d(x,y):= \textrm{sup}\, L(\sigma)
\]
where the supremum is taken over all the isocausal curves $\sigma: [0,1]\to X$, such that $\sigma(0)=x$, $\sigma(1)=y$. Observe that $\check d(x,y)>0$ implies $x\ll y$. Moreover, since for every $\sigma$, $L(\sigma)\le d(x,y)$,  we have $\check d\le d$, in particular it is finite and bounded.

\begin{theorem}
Let $(X,d)$ be a bounded Lorentzian prelength space. The following conditions are equivalent
\begin{itemize}
\item[(i)] $ \check d=d$,
\item[(ii)] $(X,d)$ is a length space.
\end{itemize}
\end{theorem}

\begin{proof}
(ii) $\Rightarrow$ (i). Let $x\ll y$  be given. Since $(X,d)$ is a bounded Lorentzian length space there is a maximal  isocausal curve $\gamma$ such that $\gamma(0)=x$, $\gamma(1)=y$.  By Proposition \ref{ffr} $d(x,y)=L(\gamma)$ which implies $\check d \ge d$, and hence $\check d=d$.

(i) $\Rightarrow$ (ii). Let $x\ll y$  be given. Let $\sigma_n:[0,1] \to X$ be a sequence of  isocausal curves such that $\sigma_n(0)=x$, $\sigma_n(1)=y$, $L(\sigma_n) \to \check d(x,y)$.
By  Theorem \ref{LCT}  there is a curve $\sigma: [0,1]\to X$ such that $\sigma_n\to \sigma$ pointwisely, and by the upper semi-continuity of the length functional
\[
d(x,y)=\check d(x,y)=\lim_n L(\sigma_n)=\limsup_n L(\sigma_n)\le L(\sigma)
\]
which implies $L(\sigma)=d(x,y)$. By  Proposition \ref{ffr} $\sigma$ is maximal.
\end{proof}

\subsection{Triangle comparison}

\begin{definition}
 Let $(X,d)$ be a bounded Lorentzian length space. A {\em timelike triangle} is given by a triple  $(x,y,z)\in X\times X\times X$, with $x\ll y\ll z$.
\end{definition}

We denote $a=d(x,y)$, $b=d(y,z)$, $c=d(x,z)$, thus, $a,b>0$ and by the reverse triangle inequality, $c \ge a+b>0$.

A point $p$ is said to be on the side $xy$, if $d(x,p)+d(p,y)=d(x,y)$, in which case we define $\alpha\in [0,1]$ such that $d(x,p)=\alpha d(x,y)$. Similarly, $p$ is said to be on the side $yz$, if $d(y,p)+d(p,z)=d(y,z)$, in which case we define $\beta\in [0,1]$ such that $d(y,p)=\beta d(y,z)$. Finally,  point $p$ is said to be on the side $xz$, if $d(x,p)+d(p,z)=d(x,z)$, in which case we define $\gamma\in [0,1]$ such that $d(x,p)=\gamma d(x,z)$. A point that belongs to one of the sides is said to belong to the perimenter of the timelike triangle.

Let us consider an abstract triangle $\Delta$ of vertices $A,B,C$ on a 2-dimensional affine space and parametrize each of the sides $AB, BC, AC$, with the parameters $\alpha,\beta,\gamma\in [0,1]$. Each choice of point on the triangle $\Delta$ determines the values of one  parameter, and hence, (recall that, by the length space property, $(X,d)$ is maximally chronally connected and $d$ is continuous) there is some point on the perimeter of the timelike triangle   in $X$, with a corresponding value of parameter (e.g.\ if it belongs to the side $xy$, $d(x,p)=\alpha d(x,y)$). Note that there could be more than one such point.

We recall that a set valued map $x\mapsto F(x)\subset \mathbb{R}^l$ defined on some open set  $O\subset\mathbb{R}^k$ is said to be  {\em upper semi-continuous} if for every $x\in O$ and for every neighborhood  $U\supset F(x)$ we can find a neighborhood $N \ni x$ such that $F(N):=\cup_{x\in N} F(x)\subset U$, cf.\ \cite{aubin84}.

\begin{definition}
Let $O\subset \mathbb{R}^3$ be an open set and let   $F: O \to \Delta^2 \times \mathbb{R}$ be an upper semi-continuous map such that $F(x)$ is a closed set for every $x$.
We say that the sectional curvature is {\em $(O,F)$-bounded} if for every timelike triangle such that $(a,b,c)\in O$ and for every pair $(d_1,d_2)\in \Delta^2$ we can find $p,q$ in the perimeter of the timelike triangle corresponding to $d_1,d_2$ respectively, such that $(d_1,d_2, d(p,q))\in F(a,b,c)$.
\end{definition}

More specific choices for $(O,F)$ will be considered below. The open set $O$ plays the role of {\em triangle size bound} while $F$ that of the bound on sectional curvature. Their form is  not important for the following result.

 Theorem \ref{bjq} tell us that bounded Lorentzian length spaces are preserved under Gromov-Hausdorff limits. Further, we have

\begin{theorem}
Let $(X_n,d_n)$, $(X,d)$, be bounded Lorentzian length spaces with $X_n \xrightarrow{{\small GH}} X$. If $(X_n,d_n)$ have $(O,F)$-bounded sectional curvature, then so has $(X,d)$.
\end{theorem}

\begin{proof}
Let us set $\epsilon_n:= 2 d_{GH}(X,X_n)>0$ (if one of them is zero the conclusion is immediate by isometry), so that $\epsilon_n\to 0$, and let us fix a correspondence $R_n\subset X\times X_n$  with $\textrm{dis}R_n \le \epsilon_n$.

Let $(x,y,z)$ be a timelike triangle in $X$ such that $(a,b,c)\in O$, and let $(d_1,d_2)\in \Delta$.

Let $x_n,y_n,z_n\in X_n$ be such that $(x,x_n), (y,y_n), (z,z_n) \in R_n $. Observe that  $a_n:=d_n(x_n,y_n)\to d(x,y)$, $b_n:=d_n(y_n,z_n)\to d(y,z)$, $c_n:=d_n(x_n,z_n)\to d(x,z)$, so that for sufficiently large $n$ the three distances are positive and  $(a_n,b_n,c_n)\in O$.

Let $p_n,q_n\in X_n$ be in the perimeter of the timelike triangle of vertices $(x_n,y_n,z_n)$, chosen so as to correspond to the parameters determined by the choice $d_1,d_2\in \Delta$.
By assumption there is one such choice such that
\[
(d_1,d_2, d_n(p,q))\in F(a_n,b_n,c_n).
\]
Let $\tilde p_n,\tilde q_n\in X$ be chosen so that $(\tilde p_n, p_n)\in R_n$, and $(\tilde q_n,q_n)\in R_n$, and let us pass to  a subsequence (denoted in the same way) such that $\tilde p_n\to p$, $\tilde q_n\to q$. We want to show that $p,q$ correspond to $d_1,d_2\in \Delta$, respectively. For instance, if $d_1$ belongs to the side parametrized by $\alpha$, we have

\[
\vert d(x,\tilde p_n)-\alpha d_n(x_n,y_n)\vert=\vert d(x,\tilde p_n)-d_n(x_n, p_n)\vert \le  \epsilon_n
\]
which implies $d(x,p)=\alpha d(x,y)$. Additionally,  from $d_n(x_n,p_n)+d_n(p_n,y_n)=d_n(x_n,y_n)$
\[
\vert d(x,\tilde p_n)+d(\tilde p_n,y)-d(x,y) \vert\le 3 \epsilon_n
\]
which in the limit gives $d(x,p)+d(p,y)=d(x,y)$. The other cases are treated similarly.

Finally, $(d_1,d_2, d_n(p_n,q_n)) \to (d_1,d_2, d(p,q))$, and since $F$ is upper semi-continuous \cite[Prop.\ 2, p.\ 41]{aubin84},  we have $(d_1,d_2, d(p,q))\in F(a,b,c)$. This shows that for each choice $d_1,d_2$ we can find $(p,q)$ as desired.
\end{proof}

We define the sectional curvature on a smooth spacetime $(M,g)$, where $g$ has signature $(-,+, \dots, +)$ and $R(X,Y)Z=\nabla_X\nabla_Y Z-\nabla_Y\nabla_X Z -\nabla_{[X,Y]} Z$, through the formula
\[
K(X\wedge Y):= - \frac{g(R(X,Y)Y, X)}{g(X,X)g(Y,Y)-g(X,Y)^2}.
\]
This definition differs by a sign  with that adopted in \cite{harris82,alexander08}. It is called timelike (spacelike) sectional curvature whenever $\textrm{span}(X,Y)$ is a timelike (resp. spacelike) plane.
It should be observed that if there is $k\in \mathbb{R}$ such that  $K\ge k$,  over every timelike plane containing $X$, then $\textrm{Ric}(X)\ge n k g(X,X)$ where $n+1$  is the spacetime dimension. Thus a lower bound on the timelike sectional curvature implies a lower bound for Ricci in the timelike direction (this is a type of strong energy condition which appears in singularity theorems), and similarly for upper bounds (with the convention in \cite{harris82,alexander08} a lower bounds leads to an upper bound and viceversa).

We say that the sectional curvature is lower bounded, and write $K\ge k$, if the following inequality holds
\[
g(R(X,Y)Y, X) \ge k [g(X,Y)^2-g(X,X)g(Y,Y)]
\]
on any 2-plane regardless  of whether $\textrm{span}(X,Y)$  is timelike or spacelike. This condition implies that the timelike sectional curvature is lower bounded by $k$.

The 2-dimensional Lorentzian manifold of constant sectional curvature $k\in \mathbb{R}$ will be
denoted $M_k$. More precisely (again, positive and negative constant sectional curvature cases are switched with respect to \cite{harris82,alexander08}), $M_0$  is the Minkowski
$1+1$ spacetime $(M^{1,1},\eta)$,  $M_{-q^2}$ is the timelike submanifold
\[
\{x: \ \eta(x,x)=1/q^2\}\subset (M^{2,1}, \eta)
\]
endowed with the induced metric (i.e.\ the 1+1 de Sitter spacetime), while $M_{q^2}$ is the submanifold
\[
\{x: \ x^2-y^2-z^2=-1/q^2\}\subset (\mathbb{R}^{3}, \dd x^2-\dd y^2-\dd z^2)
\]
endowed with the induced metric  (i.e.\ the 1+1 anti-de Sitter spacetime) \cite{harris82,hawking73}.

We recall that a timelike triangle of given sides $(a,b,c)$ is realized as a timelike triangle in a model space accordingly to the next realizability lemma \cite{harris82} \cite[Lemma 2.1]{alexander08} \cite[Lemma 4.6]{kunzinger18}

\begin{proposition}
For $k\in \mathbb{R}$, let $O_k\subset \mathbb{R}_+^3$ be the open subset determined by the inequalities
\[
a+b<c< \frac{\pi}{\sqrt{k}},
\]
(where it is understood that $1/\sqrt{k}=\infty$ for $k\le 0$). Then for every $(a,b,c)\in O_k$ there is a timelike triangle in $M_k$ with sides $(a,b,c)$ (i.e. $x,y,z\in M_k$, $x\ll y\ll z$, and  $a=d(x,y)$, $b=d(y,z)$, $c=d(x,z)$).
\end{proposition}

Under the conditions of the theorem the realizing triangle on $M_k$ is actually unique up to isometries and the vertices are connected by unique maximizing timelike geodesics.

Let us consider a realizing triangle of sides $a,b,c$ on $M_k$.
Let $d_1,d_2\in \Delta$ and let $\bar p,\bar q$ be corresponding points on the realizing triangle.
On $M_k$ there is a continuous function $H_k:\mathbb{R}^3\times \Delta^2 \to \mathbb{R}$,  such that $d_{M_k}(\bar p,\bar q)=H_k(a,b,c,d_1,d_2)$ for every triangle. The expression will not be relevant here.

\begin{definition}
Let $k\in \mathbb{R}$. We say that the bounded Lorentzian length space $(X,d)$ has  sectional curvature  bounded from below by $k$, if it is $(O_k,F_k)$-bounded where
\[
F_k(a,b,c) =\{(d_1,d_2, r): \quad r\le  H_k(a,b,c,d_1,d_2)\}.
\]
or, equivalently, for any timelike triangle and for any choice of parameters $d_1,d_2\in \Delta$ we can find two points $p,q$ on the timelike triangle with such parameters and such that   $d(p,q)\le d_{M_k}(\bar p,\bar q)$ where $\bar p,\bar q$ are corresponding points on the comparison triangle on $M_k$.

The upper bounded version is obtained replacing $\le$ with $\ge$.
\end{definition}

The upper semi-continuity of $F_k$ follows from the (upper-semi)continuity of $H_k$.

If we know that the vertices of the timelike triangle are connected by unique maximal isocausal curves, then after "equivalently" we can replace: for any timelike triangle and for any two points $p,q$ on the timelike triangle we have   $d(p,q)\le d_{M_k}(\bar p,\bar q)$ where $\bar p,\bar q$ are corresponding points on the comparison triangle on $M_k$. This will be possible within convex neighborhoods (see below).

Intuitively, under a lower bound on the sectional curvature the timelike triangle should be slender than the comparison triangle, while under an upper bound it should be fatter than the comparison triangle.

Of course, the bound on sectional curvature introduced here is meant to be the low regularity version of the bound on sectional curvature of smooth spacetimes, while at the same time be preserved under $GH$ limits.

The fact that a lower bound on the sectional curvature leads indeed to the type of inequality used in the above definition  was proved, at least for triangles contained in normal neighborhoods, by Alexander and Bishop \cite{alexander08} (actually they prove a stronger result because they used a signed distance that captures information on non-causally related events), see also Harris \cite{harris82}.

Our definition has similarities with  that used in \cite{kunzinger18}, but our geodesic triangle is just a triple of points, not of curves and we demand the existence of some pair $p,q$ with specific properties while other pair choices with equal parameters (roughly, same distances from the vertices) might not satisfy those properties.

\begin{definition} \label{vmqp}
We say that the bounded Lorentzian length space $(X,d)$ {\em admits convex neighborhoods}, if every point $p\in X$ admits some  neighborhood $U$ such that $(U, d\vert_U)$ is a bounded Lorentzian length space and
  for every pair of  chronologically related points in $U$ there is just one maximal curve which, furthermore, is isochronal and entirely contained in $U$.
\end{definition}

\begin{proposition}
In a bounded Lorentzian length space that admits convex neighborhoods  a maximal isocausal curve connecting two chronologically related points is actually isochronal (i.e.\ maximal curves have definite causal character).
\end{proposition}

\begin{proof}
By contradiction suppose that the maximal curve $\sigma:[0,1]\to X$, admits two points $\sigma(t_1)$, $\sigma(t_2)$, $t_1<t_2$, such that $d(\sigma(t_1),\sigma(t_2))=0$. By continuity of $d$ there is a largest non-empty closed interval $[a,b]\subset [0,1]$ such that $d(\sigma(a),\sigma(b))=0$. Since the inclusion  $[a,b]\subset [0,1]$ is proper, we might assume, without loss of generality $b<1$. Let $q=\sigma(b)$, and let $U_q$ be a convex neighborhood. We can find $s<b$ sufficiently close to $b$ such that $p:=\sigma(s)\in U_q$, and $t>b$ sufficiently close to $b$ such that $r:=\sigma(t)\in U_q$. By the maximality of $\sigma$, $d(p,q)=0$ and $d(p,r)=d(q,r)$ is the length of $\sigma\vert_{[s,t]}$. However, this curve is not isochronal, thus, by the property of convex neighborhoods, it is not actually maximal, which implies that $d(p,q)>L(\sigma\vert_{[s,t]})=d(p,q)$, a contradiction.
\end{proof}

\begin{example}
There exists bounded Lorentzian length spaces for which maximal isocausal curves do not have definite causal character.
Let us consider a compact coordinate box of Minkowski 2+1 spacetime and a timelike plane  $\Sigma$ on it. Let $A$ and $B$ be two timelike segments in the box transverse to $\Sigma$ where the final point of $A$, $p$, belongs to $\Sigma$, and the starting points of $B$, $q$, belongs to $\Sigma$. Moreover, let $q \in E_\Sigma^+(p)$ accordingly to the causality of $\Sigma$. Let $M=\Sigma\cup A\cup B$, and let its Lorentzian distance $d$ be given as the supremum of the Lorentzian lengths of the piecewise smooth causal connecting curves in $M$ (so its restriction to $\Sigma$ is the usual Lorentzian distance for the Lorentzian submanifold $\Sigma$). This is a Lorentzian prelength space but for $a\in A\backslash\{p\}$, $b\in B\backslash\{q\}$, we have $d(a,b)>0$, and the maximal isocausal curve passes through $p$ and $q$ where $d(p,q)=0$.
\end{example}

\begin{example} The existence of convex neighborhoods is not preserved under GH limits.
Let us consider $M_r$, $r>0$, a subset of $1+1$ Minkowski spacetime contained in the set $t^{-1}([-1,1])\cap x^{-1}([-1,1])$, where we additionally remove the sets $\{p: t(p)\ge 0, x(p)<0\}$, and $\{p: t(p)\le 0, t(p)> r x(p)\}$.  Let the Lorentzian distance $d_r$ be  the supremum of the Lorentzian lengths of the piecewise smooth causal connecting curves in $M_r$.

We have that $(M_r,d_r)$  is a bounded Lorentzian length space for $k\ge 1$ which admits convex neighborhoods iff $r>1$. For $r=1$ maximal isocausal curves do not have definite causal character. Furthermore, $M_{1+1/k}$, $k\to \infty$, Gromov-Hausdorff converges to $M$, which shows that the existence of convex neighborhoods is not preserved under GH limits.
\end{example}

\begin{definition}
We say that  on $(X,d)$ maximal isocausal curves {\em branch in the future} if we can find two maximal isocausal curves $\sigma:[0,u]\to X$ and $\gamma:[0,v]\to X$ whose images coincide over two non-empty closed intervals, i.e.\ $\sigma([0,b])=\gamma([0,c])$, $0<b<u$, $0<c<v$, that cannot be extended to the right while preserving the same property. Then $\sigma(b)$ $(\gamma(c)$) is a {\em  future branching point} for $\sigma$ (resp.\ $\gamma$).  We say that $(X,d)$ is {\em non-branching to the future } if no maximal isocausal curve branches to the future. Similar definitions hold in the past case.
\end{definition}
Observe that if $\sigma([0,b'])=\gamma([0,c'])$, $\sigma(u)\ne \sigma(v)$, then there will largest closed sets $[0,b]$, $[0,c]$, $u<b$, $c<v$, for which this identity holds and that cannot further extended to the right.

Note that our non-branching property is somewhat stronger than \cite[Def. 4.10]{kunzinger18}. Therein  non-branching spaces might admit curves that after being coincident on an interval are again coincident on a sequence of points approaching the edge of the interval.

The proof of the next result follows ideas in \cite[Thm. 4.12]{kunzinger18}, but adapted to our notions.
\begin{proposition}
Let $(X,d)$ be a bounded Lorentzian length space that admits convex neighborhoods and  suppose that the sectional curvature is bounded from below by $k\in \mathbb{R}$.
Maximal isochronal curves do not branch.
\end{proposition}

One could just impose some form of local sectional curvature boundedness, the proof argument taking place  within a convex neighborhood.
Observe that these local conditions likely do not pass to $GH$-limits as the existence of convex neighborhoods does not.

\begin{proof}
We prove that future branching is impossible, the other case being analogous. By contradiction, suppose that $\sigma$ branches to the future at the parameter $b$, and let us keep the notation introduced in the above definition.
Let $q=\sigma(b)$ and let $U_q$ be a convex neighborhood. We can find $s<b$ sufficiently close to $b$ that $x:=\sigma(s)\in U_q$ and $\sigma([s,b])\subset U_q$. We can find $t>b$ sufficiently close to $b$ that $z:=\sigma(t)\in U_q$, $\sigma([b,t])\subset U_q$, and $\sigma(t)\notin \gamma([0,1])$. Since $\sigma(b)\ll \sigma(t)$, and since for some $c$, $\gamma(c)=\sigma(b)$, we have $\gamma(c)\ll \sigma(t)$.
Let $e> c$ be such that $y:=\gamma(e)\ll \sigma(t)$, $\gamma([c,e])\subset U_q$ and $\gamma(e)\notin \sigma([0,1])$. We consider the triangle $(x,y,z)$. Since $\sigma$ is maximal $L(\sigma\vert_{[s,t]})=d(x,z)$. Similarly, since $\gamma$ is maximal $L(\gamma\vert_{[f,e]})=d(x,y)$, where $\gamma(f)=\sigma(s)$. It cannot be $d(y,z)=d(x,z)-d(x,y)$ otherwise the isochronal curve obtained concatenating $\gamma\vert_{[f,e]}$ with the maximal isochronal curve connecting $y$ to $z$ would give a maximal isochronal curve connecting $x$ to $z$ but different from $ \sigma\vert_{[s,t]}$ as passing from $y$, contradicting the uniqueness of maximal isochronal curves in convex neighborhoods. Thus the triangle is associated with a triangle inequality with the strict sign, $d(x,z)>d(x,y)+d(y,z)$. The realizing triangle in $M_k$ is therefore non-degenerate (sides are not aligned). Observe that the point $\sigma(b)$ is the only point with distances $d(x,\sigma(b))$, $d(\sigma(b),z)$ from $x$ and $z$ (by the uniqueness of maximal isochronal curves implied by the convex neighborhood). Let $\bar \sigma(b)$ the corresponding point on the side $\bar x \bar z$ on the realizing triangle on $M_k$. Similarly, let $\bar \gamma(c)$ be the point that corresponds to $\gamma(c)$ on the side $\bar x \bar y$. Note that
$d_{M_k}(\bar \gamma(c), \bar z) <d_{M_k}(\bar \sigma(b),\bar z)$ otherwise one could go from $\bar x$ to $\bar z$ with a curve of length
\[
d_{M_k}(\bar x, \bar \gamma(c))+d_{M_k}(\bar \gamma(c), \bar z)\ge d_{M_k}(\bar x, \bar \sigma(b))+d_{M_k}(\bar \sigma(b), \bar z)=d_{M_k}(\bar x, \bar z),
\]
passing through $\bar\gamma(c)$ and, due to the corner at $\bar \gamma(c)$, (due to the fact that the realizing triangle in non-degenerate) we would actually have that this timelike curve could be deformed to a longer timelike curve,  a contradiction.

Finally, since the timelike curvature is bounded from below by $k$ (used in the penultimate step)
\[
d(\sigma(b),z)=d_{M_k}(\bar \sigma(b), \bar z) > d_{M_k}(\bar \gamma(c), \bar z)\ge d(\gamma(c), z)=d(\sigma(b),z),
\]
a contradiction.
\end{proof}

\begin{remark}
It can be observed that in the above proof we are comparing $d(w,z)$ with $d_{M_k}(\bar w, \bar z)$ where $w$ is a point on the side $xy$ and $\bar w$ the corresponding point on the side $\bar x\bar y$. In the smooth case this type of comparison result follows just from the boundedness on the timelike sectional curvature due to Harris' results \cite{harris82}, with no need for conditions on spacelike plane as in Alexander and Bishop \cite{alexander08}. This means that we could have weakened the conditions in the definition of sectional curvature bound for Lorentzian length spaces imposing the comparison distance inequality for pairs $(w,z)$ placed as above, while still being able to obtain the above non-branching result.
\end{remark}

\section{Compactness}

\subsection{Uniformly totally bounded families}

As in the case of metric spaces, we will give criteria for sets of  bounded Lorentzian-metric spaces to be precompact with respect to the
Gromov-Hausdorff semi-distance.

\begin{remark}
In the next proofs we shall use the semi-distance $d_{GH}$ between bounded spaces $(X,d)$  whose distance $d$  satisfies the reverse triangle inequality as in
Definition \ref{D1}.
\end{remark}

For a bounded Lorentzian-metric space we define   $\textrm{diam} \, X:=\max_{X\times X} d$. We already observed  that it positive and  finite.

\begin{definition} \label{osj}
Let $D>0$, $\alpha:=\{\alpha_k\}_{k\in\mathbb{N}}\subset (0,\infty)$ be a decreasing sequence  with $\alpha_k\to0$ and $\beta:=\{\beta_k\}_{k\in\mathbb{N}}\subset (0,\infty)$
be an increasing sequence with $\beta_k\to \infty$.
The class $\mathfrak{X}=\mathfrak{X}(D, \alpha,\beta)$ of bounded Lorentzian-metric spaces consists of all those $X$ such that
\begin{enumerate}
\item $\textrm{diam} \, X\le D$ for all $X\in \mathfrak{X}$,
\item every $X\in \mathfrak{X}$ contains for each $k$ an $\alpha_k$-net consisting of no more than $\beta_k$ points,
\item $i^0\in X$.
\end{enumerate}
We say that $\mathfrak{X}$ is {\it uniformly totally bounded} with respect to $(D, \alpha,\beta)$.
\end{definition}

Observe that if $\alpha'_i=\alpha_{k_i}$, $\beta'_i=\beta_{k_i}$, are subsequences, then $\mathfrak{X}(D, \alpha,\beta)\subset\mathfrak{X}(D, \alpha',\beta')$.

\begin{remark}
If $\mathfrak{X}$ is uniformly totally bounded with respect to $(D, \alpha,\beta)$ the family also satisfies the following conditions which are closer to the notion in metric
geometry:
\begin{itemize}
\item[(1)] There is a constant $D$ such that $\textrm{diam} \, X\le D$ for all $X\in \mathfrak{X}$.
\item[(2)] For every $\epsilon>0$ there exists a natural number $N=N(\epsilon)$ such that every $X\in \mathfrak{X}$ contains an $\epsilon$-net consisting of no more than $N$ points.
\end{itemize}
\end{remark}

Our objective is to prove that $(\mathfrak{X}/\!\sim,d_{GH})$ is a compact metric space, where $\sim$ denotes identification by isometry. This property is equivalent to sequential compactness \cite[17G]{willard70}.

As a consequence, we expect $(\mathfrak{X}/\!\sim,d_{GH})$ to be separable. The remaining results of this section prove this property which also follows from the proof of sequential compactness  provided in the next section.

\begin{definition}
Causets and $\epsilon$-nets are said to be {\it rational} if the distance function takes only rational values.
\end{definition}

\begin{proposition} \label{zsv}
The family of rational causets is denumerable. Moreover, it is  Gromov-Hausdorff dense in the family of causets.
\end{proposition}

\begin{proof}
Let us prove the latter statement.
Let $n$ be the (finite) cardinality of the causet $S=\{x_1,\cdots, x_n\}$, then the $n(n-1)/2$ numbers $d(x_i,x_j)$, for $i,j=1,\cdots, n$, completely describe the causet. They are subject to the constraints (i) and (iii) of Def.\ \ref{D1}. Each causet of cardinality $n$ can then be represented by a point in $\mathbb{R}^{n(n-1)/2}$. Let $k\le n(n-1)/2$ be the number of distances different from zero. We are going to consider the space $\mathbb{R}^k$ as we are not going to perturb distances that are zero.

 Definition \ref{D1}(iii) is an open condition.  Indeed, the distinction metric $\gamma$ takes a minimal positive value $\alpha$ among pairs of distinct points. If the distances $d(x_i,x_j)$ are perturbed by less than half this value, (iii) still holds. Let $\epsilon>0$. We shall perturb each positive $d(x_i,x_j)$ by less than $\textrm{min}(\alpha/2, \epsilon)$. Observe that the inequalities  in Definition \ref{D1}(i) involve only those distances which are positive.

We first perturb it so as to satisfy all inequalities (i) in a strict sense.
Each distance $d(x_i,x_j)>0$ can be thought as a link connecting $x_i$ to $x_j$. We can move from $x_i$ to $x_j$ by a chain with a maximal number of links (observe that each chain has a finite number of elements due to chronology, i.e.\ boundedness of $d$, and that by the same reason no chain passes twice from the same point). If $t_{ij}$ is the maximal length of the chain we replace $d(x_i,x_j)$ with $d(x_i,x_j)+\delta t_{ij}^2$, where $\delta\le \textrm{min}(\alpha/2,\epsilon)/[n(n-1)/2]^2$. In this way all inequalities  (i) are satisfied in a strict sense. Those inequalities give now an open condition on $\mathbb{R}^k$, thus taking into account that condition (iii) is also open we can indeed find  a point in $\mathbb{Q}^k$ arbitrary close to our values. In conclusion, the positive distances can be perturbed to become rational while preserving properties (i) and (iii). Moreover, the perturbation of each distance is chosen to be less that $\epsilon$. As a consequence the newly obtained rational causet $\tilde S$ satisfies $d_{GH}(S,\tilde S)\le \epsilon$.

For the former statement, observe that $\mathbb{Q}^{n(n-1)/2}$ has the cardinality of $\mathbb{N}$ and the cardinality of $\mathbb{N}$ copies of $\mathbb{N}$ is that of $\mathbb{N}$, which proves the claim.
\end{proof}

Let $\mathfrak{B}$ be the set of bounded Lorentzian metric spaces that contain $i^0$. We have (a similar result holds for spaces that do not contain $i^0$)
\begin{proposition}
 $(\mathfrak{B}/\!\sim, d_{GH})$ is a separable metric space.
\end{proposition}

\begin{proof}
We already know from Corollary\ \ref{bof} that  $(\mathfrak{B}/\!\sim, d_{GH})$ is a metric space.  By Proposition\ \ref{zsv} and   Corollary\  \ref{vpa}  the rational causets are dense in $\mathfrak{B}$ w.r.t.\ the Gromov-Hausdorff topology, hence $(\mathfrak{B}/\!\sim, d_{GH})$ is separable.
\end{proof}

\begin{proposition}
 For a uniformly totally bounded family $\mathfrak{X}$ the semi-metric space $(\mathfrak{X}, d_{GH})$ is separable.
\end{proposition}

\begin{proof}
Since $\mathfrak{X}\subset \mathfrak{B}$, the metric topology of $(\mathfrak{X}, d_{GH})$ coincides with the topology induced from the metric topology of the semi-metric space $(\mathfrak{B}, d_{GH})$. As the latter is second-countable so is the topology of $(\mathfrak{X}, d_{GH})$, which is thus separable.

As an alternative argument note that, by  Remark \ref{remark2}, point \ref{bjd}, $d_{GH}(\tilde{S}_k, X)\le 2\alpha_k$ as $\tilde S_k$ is  an  $\alpha_k$-net of $X$ as it contains the $\alpha_k$-net $S_k$ of $X$.
 Further, by  Remark \ref{remark2}, point \ref{bkd}  and Proposition \ref{vkx}, we can find a subset $\check S_k$ of $\tilde S_k$ which is a causet, hence a bounded Lorentzian metric space,  with $d_{GH}(\check S_k,\tilde S_k)=0$. It is still true that $\check S_k$ admits for every $j$ an $\alpha_j$-net  in  $\check S_k$ which counts no more than $\beta_j$ points (they are obtained from Proposition \ref{vkx} applied to  the analogous sets for $\tilde S_k$).
 This means that $\check S_k$ belongs to $\mathfrak{X}(D, \alpha, \beta)$.
Now we can decrease all positive distances so as to preserve the $\alpha_j$-nets, via a redefinition of the following form $d(x_i,x_j)-\delta \sqrt{t_{ij}}$, $\delta>0$, with $t_{ij}$ interpreted as  in Proposition \ref{zsv} and $\delta$ sufficiently small. This operation brings the inequality of the type of Definition \ref{D1}(i) to a strict form, and all distances can be replaced by a rational value by preserving conditions Definition \ref{D1}(i),(iii) and the $\alpha_j$-nets. In conclusion, as $\alpha_k\to 0$, we find a rational causet in $\mathfrak{X}(D, \alpha, \beta)$ at arbitrarily small Gromov-Hausdorff distance from $X$, which proves the desired result as the rational causets are countable.
\end{proof}

\subsection{Compactness of uniformly totally bounded families}

\begin{theorem}\label{thm_precomp}
 For any uniformly totally bounded family $\mathfrak{X}$ the metric space $(\mathfrak{X}/\!\sim, d_{GH})$ is compact.
\end{theorem}

\begin{remark}
Due to the equivalence between the Lorentzian diameter and the distinction-metric diameter cf.\ Prop.\ \ref{eqd}, a  uniformly totally bounded family in our sense is uniformly totally bounded in the classical sense \cite{burago01}. Thus it is possible to use the classical precompactness theorem \cite[Thm. 7.4.15]{burago01} to infer that a subsequence  converges to a metric space. However, by using this strategy one cannot conclude that the limit is a bounded Lorentzian metric space.
\end{remark}

The remainder  of this section  is devoted to the proof of the sequential compactness of $\mathfrak{X}$.

 In this proof environment we shall also obtain some useful lem\-mas/the\-o\-rems/co\-rol\-laries.

Let $\mathfrak{X}$ be uniformly totally bounded with respect to $(D,\alpha,\beta)$, see Definition \ref{osj}.  Further, let  $\{X_m\}\subseteq \mathfrak{X}$ be a
sequence. We know that $(X_m,\mathscr{T}_m)$ is compact. For each positive integer $k$, $X_m$ admits an $\alpha_k$-net $S_m^{(k)}$  of at most $\beta_k$ points. By adding
points arbitrarily to each net we can assume, without loss of generality, that each $\alpha_k$-net consists of exactly $\beta_k$ points. Set $N_k:=\sum_{s=1}^k \beta_s$.

Let us order the (distinct) elements of $S_m^{(k)}$ from $N_{k-1}+1$ to $N_k$ in some (arbitrary) way, so that we can denote $S_m^{(k)}=\{x_{i,m}, i=N_{k-1}+1,
\cdots, N_k\}$.

The disjoint union
\[
\mathscr{S}_m:=\{x_{i,m}\}_i
\]
runs over elements that cover the set $S_m:= \cup_k  S_m^{(k)}$. The set $S_m$ being a union of $\alpha_k$-nets, where $(X_m,\gamma_m)$ is a metric space, is necessarily dense in $X_m$ as $\mathscr{T}_m$ is the topology induced by $\gamma_m$.

The next construction is obtained considering  the limit of the Kuratowski embeddings $X_m\to \mathcal{B}\times \mathcal{B}$ with respect to $\mathscr{S}_m$.

For each pair of positive integers $(i,j)$ we consider the sequence
\[
\{d(x_{i,m},x_{j,m})\}_m
\]
where all numbers belong to the compact set $[0,D]$. We can pass to a
subsequence of $X_m$ so that the sequence converges to some number which we denote $d_{i,j}\le D$. By a Cantor diagonal procedure, since the pairs  $(i,j)$ are countable, we can assume, without loss of generality, to have defined all numbers $d_{i,j}$ in this fashion.

This induces functions $E_i:=(e_i,e^i)\colon \mathbb{N}\to [0,D]$ via $e_i(j):=d_{i,j}$ and $e^i(j):=d_{j,i}$.
It can happen that $E_i=E_j$ for some $i\ne j$. Define
\[
S_\infty:=\bigcup_i\{E_i\}\subseteq \mathcal{B}\times \mathcal{B}
\]
and
\[
X:= \overline{S_\infty}
\]
to be the closure of $S_\infty$ in $(\mathcal{B}\times \mathcal{B},\dist\nolimits_\infty)$. A function $d\colon X\times X\to [0,D]$ is defined as
\[
d(E,F):= \lim_{r\to \infty} d_{i_r,j_r},
\]
where $E_{i_r}\to E$ and $E_{j_r}\to F$ for $r\to \infty$.

\begin{lemma}\label{lem_indep}
The function $d$ is well defined, i.e. the limit of $d_{i_r,j_r}$ exists and is independent of the sequences.
\end{lemma}

\begin{proof}
Let $E_{i_r'}=(e_{i_r'},e^{i_r'})$ and $E_{j_r'}=(e_{j_r'},e^{j_r'})$ be other sequences converging to $E$ and $F$ respectively. Then we have
\begin{align*}
|d_{i_r',j_r'}-d_{i_r,j_r}|&\le |d_{i_r',j_r'}-d_{i_r,j_r'}|+|d_{i_r,j_r'}-d_{i_r,j_r}|\\
&=|e_{i_r'}(j_r')-e_{i_r}(j_r')|+|e^{j_r'}(i_r)-e^{j_r}(i_r)|\\
&\le \|E_{i_r'}-E_{i_r}\|_\infty +\|E_{j_r'}-E_{j_r}\|_\infty\\
&\to 0
\end{align*}
for $r\to \infty$.
\end{proof}

\begin{remark} \label{rem_comp_1}
Note that every $F\in X$ can be written $F=(f,f')$ for suitable $f,f': \mathbb{N}\to [0,D]$. We have $d(F,E_i)=f(i)$ and $d(E_i,F)=f'(i)$. Let us prove the former equation, the latter being analogous. Let $E_{j_r}\to F$, then
\[
d(F,E_i)=\lim_{r\to \infty} d_{j_r, i}= \lim_{r\to \infty} e_{j_r}(i) =f(i).
\]
\end{remark}

\begin{lemma}\label{lem_comp_1}
The function $d$ satisfies the reverse triangle inequality, i.e.
\[
d(E,F)+d(F,G)\le d(E,G)
\]
if $d(E,F),d(F,G)>0$.
\end{lemma}

\begin{proof}
Choose sequences $E_{i_r}\to E$, $E_{j_r}\to F$ and $E_{k_r}\to G$. Since $d(E,F)$ as well as $d(F,G)$ are positive we have $d(E_{i_r},E_{j_r}),d(E_{j_r},E_{k_r})>0$ for almost all
$r\in\mathbb{N}$. By the reverse triangle inequality in the spaces $X_m$ and the definition of $d_{i,j}$ we have
\begin{align*}
d(E_{i_r},E_{j_r})+d(E_{j_r},E_{k_r})&= d_{i_r,j_r}+d_{j_r,k_r} \\
&=\lim\nolimits_m d(x_{i_r,m},x_{j_r,m})+ \lim\nolimits_m  d(x_{j_r,m},x_{k_r,m})\\
 &=  \lim\nolimits_m [d(x_{i_r,m},x_{j_r,m})+ d(x_{j_r,m},x_{k_r,m})] \\
&\le \lim\nolimits_m d(x_{i_r,m},x_{k_r,m})=d_{i_r,k_r} =d(E_{i_r},E_{k_r})
\end{align*}
for almost all $r\in \mathbb{N}$. The claim now follows by taking the limit $r\to \infty$.
\end{proof}

\begin{lemma}\label{lem_comp_2}
The function $d\colon X\times X\to [0,D]$ is continuous in the product topology, where the topology on $X$ is that induced from $(\mathcal{B}\times \mathcal{B},\dist_\infty)$.
\end{lemma}

\begin{proof}
Let $F_{n}$ and $G_{n}$ be sequences in $X$ converging to $F$ and $G$ respectively. For every $n$ choose sequences
$\{E_{i_{r}^n}\}_r$ and $\{E_{j_{r}^n}\}_r\subseteq S_\infty$ converging to $F_{n}$ and $G_{n}$ respectively.

Next choose for every $n\in \mathbb{N}$ a $r_n\in \mathbb{N}$ with
\[
|d_{i_{r_n}^n,j_{r_n}^n}-d(F_{n},G_{n})|<\frac{1}{n}, \quad  \dist\nolimits_{\infty}(E_{i_{r_n}^n},F_n),\, \dist\nolimits_{\infty}(E_{j_{r_n}^n},G_n) <\frac{1}{n}.
\]
This induces sequences $\{E_{i_{r_n}^n}\}_n$ and $\{E_{j_{r_n}^n}\}_n\subseteq S_\infty$ converging to $F$ and $G$,
respectively. By Lemma \ref{lem_indep} follows
\[
d(E_{i_{r_n}^n},E_{j_{r_n}^n})=d_{i_{r_n}^n,j_{r_n}^n}\to d(F,G)
\]
and therefore $d(F_{n},G_{n})\to d(F,G)$.
\end{proof}

\begin{lemma}
The set $S_\infty^{(k)}:=\bigcup\{E_j: N_{k-1}+1\le j\le N_k \}  \subset S_\infty$ is  an $\alpha_k$-net of $X$.
\end{lemma}

\begin{proof}
Assume the claim is false. Then, since $S_\infty$ is dense in $X$, there exists $E_l\in S_\infty$ with $\|E_l-E_i\|_\infty> \alpha_k $ for all
$i=N_{k-1}+1,\ldots,N_k$. Choose $r_i\in
\mathbb{N}$ with $|E_l(r_i)-E_i(r_i)|>\alpha_k $. Then follows that
\[
|d_{l,r_i}-d_{i,r_i}|>\alpha_k \text{ or }|d_{r_i,l}-d_{r_i,i}|>\alpha_k .
\]
The same is true for
\[
|d(x_{l,m},x_{r_i,m})-d(x_{i,m},x_{r_i,m})|\text{ or }|d(x_{r_i,m},x_{l,m})-d(x_{r_i,m},x_{i,m})|
\]
for $m$ sufficiently large (note that $l$ is fixed while $i$ takes only a finite number of possible values), thus contradicting
that $S_m^{(k)}$  is an  $\alpha_k$-net in $X_m$.
\end{proof}

\begin{corollary}\label{cor_comp_1}
$X$ with the topology induced from $(\mathcal{B}\times \mathcal{B},\dist_\infty)$ is compact.
\end{corollary}

\begin{proof}
The topology induced on $X$ from $(\mathcal{B}\times \mathcal{B},\dist_\infty)$ coincides with the metric topology of $(X,\dist_\infty\vert_{X\times X})$. A subset of a metric space is compact iff it is sequentially compact.

Let us consider a sequence $P_i\in X$. For each $k$,  $X$ has a finite $\alpha_k $-net, thus we can pass to a subsequence so that the subsequence enters and remains in  a $\alpha_k $-closed ball for sufficiently large $i$. Via a Cantor diagonal argument the same can be assumed for every $k$, in particular the subsequence is Cauchy and so converges as  $(\mathcal{B}\times \mathcal{B},\dist_\infty)$
is complete. As $X$ is closed, the limit belongs to $X$. This shows that $X$ is compact.
\end{proof}

\begin{theorem}
The space $(X,d)$ is a bounded Lorentzian-metric space.
\end{theorem}

\begin{proof}
We check the characterization of bounded Lorentzian-metric space of Definition \ref{D1}.

(i). The triangle inequality follows from Lemma \ref{lem_comp_1}.

(ii'). Observe that if $X$ is endowed with the topology induced from
$(\mathcal{B}\times \mathcal{B},\dist_\infty)$, then by   Corollary \ref{cor_comp_1}  $X$ is compact and by  Lemma \ref{lem_comp_2} $d:X\times X \to [0,\infty)$ is continuous in the product topology on $X\times X$. Finally,  the sets   $\{d\ge \epsilon\}$ are compact, as they are closed subsets of the compact set $X\times X$.

(iii).  Suppose that $F,G\in X$ are not distinguished. For every $i$ we have in particular, $d(F,E_i)=d(G,E_i)$ and $d(E_i,F)=d(E_i,G)$, thus from  Remark \ref{rem_comp_1} we conclude  that $f=g$ and $f'=g'$,
 where $F=(f,f')$ and $G=(g,g')$, hence  $F= G$.
\end{proof}

\begin{proof}[Proof of Theorem \ref{thm_precomp}]
Finally, Theorem \ref{thm_precomp}  follows from Proposition \ref{prop_nets}. Indeed, for each $k$ the result $S^{(k)}_M\xrightarrow{{\small GH}}   S^k_\infty$ follows from
$d(E_{i}, E_j)=d_{i,j}=\lim_m d(x_{i,m},x_{j,m})$, for $N_{k-1}+1\le i,j\le N_k$. Since the possible pairs $(i,j)$ are finite in number, for every $\epsilon>0$ we can find $m$ sufficiently
large such that the correspondence
\[
R_m:=\{(x_{i,m}, E_i): N_{k-1}+1\le i\le N_k\} \subset S^{(k)}_m\times S^{(k)}_\infty
\]
has distortion less than $\epsilon$.

Since $i^0_m\in X_m$, by taking $m$ such that $d_{GH}(X_m,X)\le \epsilon$ we infer that there is a point $O_m\in X$ such that $\textrm{max} [d(O_m,P), d(P,O_m)]\le 2 \epsilon$ for every
other point $P$. Since $X$ is compact and $d$ is continuous, consideration of a limit point $O$ of $O_m$  easily shows that $X$ includes the spacelike boundary point.
\end{proof}

\section*{Acknowledgements}
This work started in the Autumn of 2019, the authors benefitting of the hospitality of some institutions: the Ruhr-Universit\"at Bochum, the Universit\`a degli Studi di Firenze, and the Centro de Giorgi  in Pisa (under the program research in pairs). Ettore Minguzzi is partially supported by  MIUR PRIN 2022 project 2022JJ8KER - ``Contemporary perspectives on geometry and gravity''. Stefan Suhr is partially supported by the Deutsche Forschungsgemeinschaft (DFG, German Research Foundation) -- Project-ID 281071066 -- TRR 191.

\section*{Declarations}

\begin{itemize}
%\item Funding
\item {\bf Data availability}. All data generated or analyzed during this study are included in this published article.
\item {\bf Conflict of interest}. The authors state that there is no conflict of interest.
%\item Ethics approval
%\item Consent to participate
%\item Consent for publication
%\item Availability of data and materials
%\item Code availability
\item {\bf Authors' contributions}. The authors contributed equally to this work.
\end{itemize}

%\bibliography{../../bibliografie/simultaneity,../../bibliografie/libri,../../bibliografie/miei,../../bibliografie/mieiPrep,../../bibliografie/mieiProc}
%\bibliographystyle{plain}

\end{document}